\documentclass[aap, preprint]{imsart}

\usepackage[T1]{fontenc}	
\RequirePackage{amsthm,amsmath,amsfonts,amssymb}	
\RequirePackage[colorlinks,citecolor=blue,urlcolor=blue]{hyperref}
\RequirePackage{graphicx}

\startlocaldefs
\usepackage{upgreek}
\usepackage{mathrsfs}
\usepackage{pdfsync}
\usepackage[usenames,dvipsnames]{xcolor}
\usepackage[inline,shortlabels]{enumitem}
\usepackage{dsfont}									
\usepackage{mathtools, esint}
\usepackage{microtype}
\usepackage{verbatim}

\usepackage{subcaption}

\allowdisplaybreaks

\usepackage{xparse}

\ExplSyntaxOn
\NewDocumentCommand{\makeabbrev}{mmm}
 {
  \yoruk_makeabbrev:nnn { #1 } { #2 } { #3 }
 }

\cs_new_protected:Npn \yoruk_makeabbrev:nnn #1 #2 #3
 {
  \clist_map_inline:nn { #3 }
   {
    \cs_new_protected:cpn { #2 } { #1 { ##1 } }
   }
 }
 \ExplSyntaxOff

\makeabbrev{\textbf}{tbf#1}{a,b,c,d,e,f,g,h,i,j,k,l,m,n,o,p,q,r,s,t,u,v,w,x,y,z,A,B,C,D,E,F,G,H,I,J,K,L,M,N,O,P,Q,R,S,T,U,V,W,X,Y,Z}

\makeabbrev{\textbf}{bf#1}{a,b,c,d,e,f,g,h,i,j,k,l,m,n,o,p,q,r,s,t,u,v,w,x,y,z,A,B,C,D,E,F,G,H,I,J,K,L,M,N,O,P,Q,R,S,T,U,V,W,X,Y,Z}

\makeabbrev{\textsf}{tsf#1}{a,b,c,d,e,f,g,h,i,j,k,l,m,n,o,p,q,r,s,t,u,v,w,x,y,z,A,B,C,D,E,F,G,H,I,J,K,L,M,N,O,P,Q,R,S,T,U,V,W,X,Y,Z}

\makeabbrev{\mathsf}{mss#1}{a,b,c,d,e,f,g,h,i,j,k,l,m,n,o,p,q,r,s,t,u,v,w,x,y,z,A,B,C,D,E,F,G,H,I,J,K,L,M,N,O,P,Q,R,S,T,U,V,W,X,Y,Z}

\makeabbrev{\mathfrak}{mf#1}{a,b,c,d,e,f,g,h,i,j,k,l,m,n,o,p,q,r,s,t,u,v,w,x,y,z,A,B,C,D,E,F,G,H,I,J,K,L,M,N,O,P,Q,R,S,T,U,V,W,X,Y,Z}

\makeabbrev{\mathrm}{mrm#1}{a,b,c,d,e,f,g,h,i,j,k,l,m,n,o,p,q,r,s,t,u,v,w,x,y,z,A,B,C,D,E,F,G,H,I,J,K,L,M,N,O,P,Q,R,S,T,U,V,W,X,Y,Z}

\makeabbrev{\mathbf}{mbf#1}{a,b,c,d,e,f,g,h,i,j,k,l,m,n,o,p,q,r,s,t,u,v,w,x,y,z,A,B,C,D,E,F,G,H,I,J,K,L,M,N,O,P,Q,R,S,T,U,V,W,X,Y,Z}

\makeabbrev{\mathcal}{mc#1}{A,B,C,D,E,F,G,H,I,J,K,L,M,N,O,P,Q,R,S,T,U,V,W,X,Y,Z}

\makeabbrev{\mathbb}{mbb#1}{A,B,C,D,E,F,G,H,I,J,K,L,M,N,O,P,Q,R,S,T,U,V,W,X,Y,Z}

\makeabbrev{\mathscr}{ms#1}{A,B,C,D,E,F,G,H,I,J,K,L,M,N,O,P,Q,R,S,T,U,V,W,X,Y,Z}
\makeabbrev{\mathrm}{#1}{
Id,id,ran,rk,diag,stab,ann,conv,pr,ev,tr,End,Hom,sgn,im,op,can,fin,ext,red,tot,rank,
rot,usc,lsc,Lip,lip,bLip,AC,loc,
supp,Opt,Adm,Cpl,Geo,GeoOpt,GeoAdm,GeoCpl,reg,
co,Ric,Exp,dExp,dist,seg,Seg,cut,fcut,Cut,SDiff,Iso,Isom,diam,cl,Homeo,Diff,Der,vol,inj,relint,
var,law,Var,Poi,Gam,pa,so,iso,fs,inv,pqi,mix,cov,
}

\newcommand{\T}{\tau}
\newcommand{\set}[1]{\left\{#1\right\}}							
\newcommand{\tset}[1]{\big\{#1\big\}}							
\newcommand{\Cb}{\mcC_b}									
\newcommand{\Mb}{\mathscr M_b}
\newcommand{\Mbp}{\mathscr M_b^+}
\newcommand{\semicolon}{\,\,\mathrm{;}\;\,}
\newcommand{\av}[1]{\left\langle#1\right\rangle}
\DeclareMathOperator*{\osc}{osc}
\numberwithin{equation}{section}

\newcommand{\Neu}{\textsc{n}}
\newcommand{\Dir}{\textsc{d}}

\newcommand{\restr}[1]{\big\rvert_{#1}}
\newcommand{\seq}[1]{\left(#1\right)}
\newcommand{\tseq}[1]{\big(#1\big)}
\newcommand{\ttseq}[1]{(#1)}
\newcommand{\tonde}[1]{\left(#1\right)}
\newcommand{\quadre}[1]{\left[#1\right]}
\newcommand{\ttonde}[1]{\big(#1\big)}

\newcommand{\abs}[1]{\left\lvert#1\right\rvert}
\newcommand{\tabs}[1]{\big\lvert#1\big\rvert}
\newcommand{\scalar}[2]{\left\langle #1\, \middle \vert\, #2 \right\rangle}
\newcommand{\tscalar}[2]{\big\langle #1 \big | #2 \big\rangle}
\newcommand{\ttscalar}[2]{\langle #1 | #2 \rangle}
\newcommand{\norm}[1]{\left\lVert#1\right\rVert}
\newcommand{\tnorm}[1]{\big\lVert#1\big\rVert}
\newcommand{\dom}[1]{\mathscr{D}(#1)}
\newcommand{\emparg}{{\,\cdot\,}}
\newcommand{\emp}{\varnothing}
\newcommand{\eqdef}{\coloneqq}
\newcommand{\defeq}{\eqqcolon}
\newcommand{\car}{\mathds{1}}

\newcommand{\bdvol}{\sigma_{\partial\Omega}}

\newcommand{\diff}{\mathrm{d}}

\newcommand{\fstop}{\; \text{.}}
\newcommand{\comma}{\; \text{,}\;\;}

\newcommand{\eps}{\varepsilon}
\newcommand{\bd}{{\partial}}

\newcommand{\purple}[1]{{\color{black}{#1}}}
\newcommand{\rosso}[1]{{\color{black}{#1}}}

\newcommand{\Peps}{}

\newcommand{\cA}{\ensuremath{\mathscr A}} 
 
\newcommand{\cC}{\ensuremath{\mathcal C}} 

\newcommand{\cF}{\ensuremath{\mathcal F}} 
\newcommand{\cG}{\ensuremath{\mathcal G}}

\newcommand{\cJ}{\ensuremath{\mathcal J}} 
\newcommand{\cK}{\ensuremath{\mathcal K}} 
\newcommand{\cL}{\ensuremath{\mathcal L}} 
\newcommand{\cM}{\ensuremath{\mathcal M}} 
\newcommand{\cN}{\ensuremath{\mathcal N}}

\newcommand{\cR}{\ensuremath{\mathcal R}} 
\newcommand{\cS}{\ensuremath{\mathcal S}} 
 
\newcommand{\cU}{\ensuremath{\mathcal U}} 
\newcommand{\cV}{\ensuremath{\mathcal V}} 
\newcommand{\cW}{\ensuremath{\mathcal W}} 
\newcommand{\cX}{\ensuremath{\mathcal X}} 
\newcommand{\cY}{\ensuremath{\mathcal Y}}

\newcommand{\E}{\ensuremath{\mathbb{E}}}

\newcommand{\N}{\ensuremath{\mathbb{N}}}
\newcommand{\Z}{\ensuremath{\mathbb{Z}}}

\newcommand{\R}{\ensuremath{\mathbb{R}}}
\renewcommand{\P}{\ensuremath{\mathbb{P}}}

\newcommand{\bx}{\mathbf{x}}
\newcommand{\by}{\mathbf{y}}

\newcommand{\bn}{\mathbf{n}}

\newcommand{\sy}{\textrm{sym}}
\renewcommand{\:}{\,\textrm{\normalfont:}\,}

\newcommand{\dd}{\text{\normalfont d}}

\newcommand{\stat}{\text{\normalfont stat}}

\newcommand{\axz}{\ensuremath{\alpha_\eps^{xz}}}
\newcommand{\ayz}{\ensuremath{\alpha_\eps^{yz}}}

\newcommand{\LipOmega}{\ensuremath{M_\Omega}}

\newtheorem{theorem}{Theorem}[section]
\newtheorem{lemma}[theorem]{Lemma}
\newtheorem{proposition}[theorem]{Proposition}
\newtheorem{corollary}[theorem]{Corollary}

\theoremstyle{definition}
\newtheorem{assumption}[theorem]{Assumption}
\newtheorem{definition}[theorem]{Definition}

\theoremstyle{remark}
\newtheorem{remark}[theorem]{Remark}

\usepackage[textwidth=30mm]{todonotes}

\endlocaldefs

\begin{document}

\begin{frontmatter}
\title{Scaling limits of random walks, harmonic profiles, and stationary non-equilibrium states in Lipschitz domains\texorpdfstring{${}^\S$}{S}
}
\runtitle{Random walks, harmonic profiles, and SNS in Lipschitz domains}

\thankstext{t1}{This author gratefully acknowledges funding by the Austrian Science Fund (FWF) grant F65, by the European Research Council (ERC, grant agreement No 716117, awarded to Prof.~Dr.~Jan Maas).
\purple{
He also gratefully acknowledges funding of his current position by the Austrian Science Fund (FWF) grant ESPRIT 208.
}
}

\thankstext{t2}{This author gratefully acknowledges funding by the Hausdorff Center for Mathematics at the University of Bonn.
Part of this work was completed while this author was a member of the Institute of Science and Technology Austria.
He gratefully acknowledges funding of his position at that time by the Austrian Science Fund (FWF) grants F65 and W1245.
}

\thankstext{t3}{This author gratefully acknowledges funding by the Lise Meitner fellowship, Austrian Science Fund (FWF): M3211. Part of this work was completed while  funded  by the European Union's Horizon 2020 research and innovation programme under the Marie-Sk\l{}odowska-Curie grant agreement No.~754411.
}

\thankstext{t4}{The authors are very grateful to Antonio Agresti for many useful conversations about boundary value problems in Lipschitz domains.
The first named author wishes to thank Kazuhiro Kuwae for a useful discussion about the paper~\cite{KuwShi03}.
\purple{He is also grateful to Alessandra Faggionato, Lorenzo Bertini, and Giada Basile for fruitful conversations on the subject of this paper.}
The third named author is grateful to  Patr\'{\i}cia Gon\c{c}alves for some fruitful conversations on an earlier draft of this paper, and to  Claudio Landim for kindly pointing out the reference \cite{landim_hydrodynamical1996}.
\rosso{The authors wish to express their gratitude to three anonymous reviewers for their careful reading and very helpful suggestions.}
}

\begin{aug}
\author[A]{\fnms{Lorenzo} \snm{Dello Schiavo}\ead[label=e1]{lorenzo.delloschiavo@ist.ac.at}\thanksref{t1}},
\author[B]{\fnms{Lorenzo} \snm{Portinale}\ead[label=e2]{portinale@iam.uni-bonn.de}\thanksref{t2}},
and
\author[C]{\fnms{Federico} \snm{Sau}\ead[label=e3]{federico.sau@units.it}\thanksref{t3}}

\runauthor{L.\ Dello Schiavo, L.\ Portinale, F.\ Sau}


\address[A]{Institute of Science and Technology Austria (ISTA) \printead{e1}}
\address[B]{IAM Bonn \printead{e2}}
\address[C]{Dipartimento di Matematica e Geoscienze -- Universit\`a degli Studi di Trieste \printead{e3}}
\end{aug}

\begin{abstract}
We consider the open symmetric exclusion (SEP) and inclusion (SIP) processes on a bounded Lipschitz domain~$\Omega$, with both fast and slow boundary.
For the random walks on~$\Omega$ dual to SEP/SIP we establish: a functional-CLT-type convergence to the Brownian motion on~$\Omega$ with either Neumann (slow boundary), Dirichlet (fast boundary), or Robin (at criticality) boundary conditions; 
the discrete-to-continuum convergence of the corresponding harmonic profiles.
As a consequence, we rigorously derive the hydrodynamic and hydrostatic limits for SEP/SIP on~$\Omega$, and analyze their stationary non-equilibrium fluctuations.
\purple{All scaling limit results for SEP/SIP concern finite–dimensional
	distribution convergence only, as our duality techniques do not require to
	establish tightness for the fields associated to the particle systems.}
\end{abstract}



\end{frontmatter}

\section{Introduction}
Stationary non-equilibrium states (SNS) of open microscopic interacting particle systems play a major role in the development of a macroscopic theory of thermodynamical fluctuations out of equilibrium. In particular, exactly solvable systems  serve as important models to link the emergence of long-range correlations at the microscale, as first observed by H.~Spohn~\cite{spohn_long_1983}, and the non-locality of the action functionals at the macroscale (see, e.g., the surveys~\cite{bertini_macroscopic_2015,derrida_microscopic2011}, the recent articles \cite{bouley2021thermodynamics,frassek_duality2020}, and references therein).

For one-dimensional nearest-neighbor systems coupled with  reservoirs at the two ends of the chain, a number of rigorous results for a broad class of interacting systems is available. For the \emph{symmetric  exclusion process} (SEP), for instance, Derrida \emph{et al.} \cite{derrida_exact_1993-1} provided explicit matrix representations of SNS,  while~\cite{eyink_hydrodynamics_1990, derrida_large2002, landim_stationary_2006} proved	 scaling limits (hydrodynamics,  large deviations and fluctuations, respectively); more recently, sharp convergence to stationarity in total variation has been shown in~\cite{gantert2020mixing,gonccalves2021sharp}. The introduction of an additional parameter tuning the interaction rate of the bulk of the system with the boundary  has also received an increasing attention in the past decade, leading to new macroscopic scenarios (see, e.g., \cite{franco_phase_2015TAMS, baldasso_exclusion_2017, goncalves_hydrodynamics_2019,derrida_large2021}).  Moving to  other one-dimensional systems, we mention, among many:  the open \emph{asymmetric exclusion} and \emph{symmetric harmonic processes}, for which exact solutions have been derived  (see, e.g., 	\cite{derrida_exact_1993-1,frassek2021exact}); the \emph{KMP model} and the \emph{symmetric inclusion process} (SIP), whose structure of  $k$-point correlations and scaling limits   have been considered (see, e.g., \cite{kipnis_heat_1982, bertini_large_deviations_stochastic_model_heat_flow2005, carinci_duality_2013-1, franceschini2020symmetric}).

On more general geometries, mainly due to the lack of closed-form expressions,  the literature is far more sparse.
Hydrodynamic limits on $d$-dimensional hypercubes have been recently studied in \cite{xu2021hydrodynamic} for SEP in contact with slow reservoirs. As for  microscopic properties of SNS, a universal factorized form of the $k$-point correlations and cumulants has been shown on general finite  graphs coupled with two particle reservoirs for  SEP and SIP \cite{floreani_boundary2020}.

\paragraph{Main results} 
In the present work, we provide for the first time a complete characterization of the scaling limits for SNS (hydrostatics and corresponding fluctuations) in the generality of both fast and slow boundaries in Lipschitz domains.

In spite of the lack of explicit representations for SNS,
\begin{itemize}
	\item we rigorously derive the \emph{hydrostatic limit}, \textbf{Theorem~\ref{th:hydrostatic}},
	\item and analyze the corresponding \emph{fluctuations}, \textbf{Theorem~\ref{th:fluctuations-stat}}, 
\end{itemize}   for SEP and SIP on arbitrary bounded Lipschitz domains in $\R^d$, $d\ge 2$, in contact with both fast and slow reservoirs. The reservoirs are placed so to approximate the Lipschitz boundary, and we do not require  their densities to attain only two values, but let them vary continuously in space. In close relation with such scaling limits for SNS,
\begin{itemize}
	\item  we prove the \emph{hydrodynamic limit} for  these particle systems when starting out of stationarity, \textbf{Theorem~\ref{t:MainHydrodynLim}},
	\item and show that the $k$-point stationary correlations vanish uniformly in space,  thus, 	establishing a weak form of \emph{local equilibrium} for SNS, \textbf{Theorem~\ref{t:LocalEquilibrium}}.
\end{itemize}  

Due to the roughness of the Lipschitz boundary, standard discrete-to-continuum approximations testing against smooth functions do not suffice to close the evolution equations for the microscopic systems.
We overcome this difficulty by improving the existing techniques in proving hydrodynamic limits, only relying on the continuity of test functions and mild solution representations of the empirical density fields. The main feature of the microscopic systems which allows this decomposition is \emph{duality} (see, e.g., \cite{carinci_duality_2013-1,floreani_boundary2020}), a form of exact solvability: by means  of a few \textquoteleft dual\textquoteright\ interacting particles   having the original external reservoirs as absorbing states, we give all moments of the original particle systems an alternative probabilistic representation. Thanks to  duality, all steps in the proof of the hydrodynamic and hydrostatic limits can be recast into results for the \textquoteleft simpler\textquoteright\ dual processes.
Moreover, our approach via duality \purple{--- although much less adaptable to perturbations of the particle dynamics ---}  comes with some advantages over the so-called `entropy' and `relative entropy methods' \cite{kipnis_scaling_1999}: on the one side, it avoids the technical steps of establishing
	 replacement lemmas near the boundary; on the other side, it grants more general convergence results, as it allows us to also deal with degenerate reservoirs' densities.  
	 \rosso{Finally, let us note that  tightness of the discrete fields in the path space does not follow from our methods.
	 In fact, it is one other advantage of our approach that we can prove limit theorems for (the finite-dimensional distributions of) the discrete fields without establishing tightness for their paths.}
%
%
%

The  main steps in our argument are the following two convergence results: 
\begin{itemize}
	\item  a \emph{functional central limit theorem} (FCLT) for a single random walk, \textbf{Theorem~\ref{t:MainSemigroups}};
	\item   a \emph{uniform convergence} of discrete \emph{\purple{harmonic functions}} \purple{(also denominated as \emph{harmonic profiles} or \emph{harmonics})} to their continuum counterparts, \textbf{Theorem~\ref{th:harmonic_conv}}.
\end{itemize}  Depending on the intensity of the rate of absorption at the boundary, in the continuum limit we recover Brownian motions and harmonics corresponding to different boundary conditions: Neumann for a sufficiently slow interaction, Dirichlet for a fast one, and Robin at the threshold between the two.
As already mentioned above, such convergence results are derived without assuming smoothness of test functions near the Lipschitz boundary, but only by proving and exploiting the fact that continuous functions up to the boundary behave nicely under the action of the Brownian motions' semigroups. Further, we emphasize that, while in the one-dimensional nearest-neighbor case explicit formulas for both discrete and continuous harmonics are known, in our setting no such expressions are available, nor is the convergence of discrete harmonics a direct consequence of the finite-time horizon FCLTs of the corresponding random walks.

Next to deriving the hydrostatic and hydrodynamic limits, we extract properties of suitable intrinsic nuclear spaces and establish  CLTs for the fluctuations associated to SNS. The limiting fields are centered Gaussian, with  covariances  determined without relying on regularity assumptions other than the Lipschitz boundary and continuity of the boundary data. In particular, covariances   in the Robin and Dirichlet regimes are identified requiring neither strong differentiability nor existence of weak derivatives near the boundary for neither the test functions nor the harmonics.

About our method, we believe the convergence results on random walks and harmonics and the techniques we use to prove them to be of independent interest, as potentially generalizable to other classes of spaces (e.g.\ fractal-like) for which suitable continuity estimates of discrete and continuum heat kernels can be established. 
For the statements on fluctuations, however, we expect that it is not possible to improve our results to rougher (e.g.\ H\"older) domains, since, for example, we crucially use compatibility properties of $L^2$- and $\Cb$-Laplacians which fail on general non-Lipschitz domains (see~\S\ref{sec:generality-Lipschitz} and~\S\ref{s:stat-flu}).

Finally, besides providing a concrete example of scaling limits of particle systems in rough domains with external reservoirs, we extend the usual arguments  \cite{de_masi_mathematical_1991, kipnis_scaling_1999} so to ensure convergence in contexts in which the (predictable) quadratic variations of the so-called Dynkin's martingales associated to the particle systems contain unbounded terms.
This allows us to include in our analysis \emph{singular} (i.e., delta-like) initial conditions which relax at positive times (see, e.g., \cite{landim_hydrodynamical1996} for hydrodynamics with non-relaxing Dirac measures forming due to particles' slow-down), and scaling limits of systems with fast boundary and an unbounded number of particles per site.
This improvement is achieved by exploiting the smoothening action of the random walks' semigroups, and employing mild (in place of weak) solutions  to the hydrodynamic equations.

\paragraph{Organization of the paper}
The paper is organized as follows. In~\S\ref{sec:setting-models}, we introduce the continuum and discrete geometric setting, as well as the interacting particle systems coupled with external reservoirs.
In~\S\ref{sec:main-results}, we present the main results of the paper: the convergence results for random walks, and for discrete harmonics are the contents of~\S\ref{sec:duality-RW-BM} and~\S\ref{sec:harmonic-profiles}, respectively. 
The  hydrodynamic and hydrostatic limits, as well as the local equilibrium for SNS, are all stated in~\S\ref{sec:HDL-HSL}, while the stationary non-equilibrium fluctuations in~\S\ref{s:stat-flu}.
The remaining sections are devoted to the proofs of the main results: 
after deriving some auxiliary results in \S\ref{sec:auxiliary-RW-h} \purple{(part of which is postponed to~\S\S\ref{sec:proof-equicontinuity},\ref{sec:proof-Equivalences})}, the proofs of the main results in~\S\ref{sec:duality-RW-BM} and~\S\ref{sec:harmonic-profiles} are contained in~\S\ref{sec:semigroup-conv-proofs} and~\S\ref{sec:harmonic_conv}, respectively.
In~\S\ref{s:proofs-IPS}, we present the necessary duality relations, related properties, and the proofs of the results from~\S\ref{sec:HDL-HSL}; in~\S\ref{sec:proof-flu}, those of the CLTs for the stationary non-equilibrium fluctuations.
Let us stress that the proofs of scaling limits in~\S\S\ref{s:proofs-IPS}, \ref{sec:proof-flu} may be read independently from \S\S\ref{sec:auxiliary-RW-h}--
\ref{sec:harmonic_conv}.
Finally, known and new results on Laplacians, their corresponding semigroups and intrinsic function spaces on bounded Lipschitz domains are collected in~\S\ref{sec:appendix-laplacians}.

\section{Setting and models}\label{sec:setting-models}
Everywhere in the following we let~$d\geq 2$ be an integer, and~$\eps\in (0,1)$.
All the asymptotic notation is understood in the limit~$\eps\to 0$.
\paragraph{General notation}
Let $\N\eqdef \set{1,2,\ldots}$, $\N_0\eqdef \N\cup \set{0}$, $\R^+\eqdef(0,\infty)$ and  $\R_0^+\eqdef \R^+\cup\set{0}$.
For all $B\subseteq \R^d$, we denote by~$\#B\in \overline\N_0\eqdef \N_0\cup\set{\infty}$ the cardinality of~$B$.
Throughout this work,~$\abs{\emparg}$~stands for the Euclidean norm, while $\dist$, resp.~$\dist_{\rm H}$, denotes the usual Euclidean distance between points and sets, resp.\ the Hausdorff distance between sets. For a set~$U\subset\R^d$, $\cC_b(U)$~denotes the space of bounded and continuous functions on $U$. If $U$ is open and $k \in \overline\N_0$,  $\cC^k_c(U)$, resp.\ $\cC^k(\overline U)$, denotes the subspace of $\cC^k(U)$ of compactly supported functions, resp.\ of functions whose  derivatives of all orders up to $k$ continuously extend to the closure of~$U$.
When $k=0$, we write  $\cC(\overline U)=\cC^0(\overline U)$ and let $\cC_0(U)$ denote the subspace of $\cC(\overline U)$ of identically vanishing functions on $\partial U$.
For all $f \in \cC_b(U)$ and any finite, resp.\ finite signed,  measure $\mu \in \Mbp(U)$, resp. $\Mb(U)$, we write $\scalar{\mu}{f}\eqdef \int f\, \dd \mu$, and denote by $\scalar{\emparg}{\emparg}_H$ the inner product on a Hilbert space $H$. All throughout, $C, C', C_1, C_2, \ldots$ denote positive constants whose value is unimportant and may change from line to line.

\subsection{Lattice approximations of bounded Lipschitz domains}
If not stated otherwise, $\Omega$~always denotes a (connected) \emph{bounded Lipschitz domain}.
Without loss of generality, we assume that $0\in \Omega$.
We denote by~$\mu_\Omega$ the standard Lebesgue measure on~$\Omega$, and by $\bdvol$ the \emph{surface measure} of~$\partial\Omega$, i.e., the restriction to $\partial\Omega$ of the $(d-1)$-dimensional Hausdorff measure $\mcH^{d-1}$ on $\R^d$.

\subsubsection{The discrete domain \texorpdfstring{$\Omega_\eps$}{OmegaEps} and its volume measure \texorpdfstring{$\mu_\eps$}{MuEps}} \label{sec:discrete-domain-volume-measure}
Let $\overline{\eps\Z^d}$ be the union of all closed line segments joining nearest neighbors in $\eps\Z^d$, and $\Omega_\eps^*$ be the connected component of $\Omega\cap \overline{\eps\Z^d}$ containing the origin.
Set $\Omega_\eps\coloneqq \Omega_\eps^*\cap \eps\Z^d$. 
For~$x,y\in\Omega_\eps$ we write~$x\sim y$ to indicate that~$\abs{x-y}=\eps$, the dependence on~$\eps$ being implicit in the notation~$\sim$ and always apparent from context.

Further let $\mu_\eps\coloneqq \eps^d \sum_{x\in \Omega_\eps}\delta_x$ be the (discrete) `volume measure' of $\Omega_\eps$.
For any uniformly bounded and equicontinuous family $\cF$ in $\cC_b(\R^d)$, we have that
\begin{equation}
\lim_{\eps\downarrow0} \sup_{f\in \cF}\abs{\scalar{\mu_\eps}{f}-\scalar{\mu_\Omega}{f}}=0\fstop
\end{equation}
As a consequence, there exists~$C=C_\Omega\ge 1$ such that
\begin{equation}\label{eq:VolumeBound}
C^{-1}\le \mu_\eps(\Omega_\eps)\le C\fstop 
\end{equation}

\subsubsection{The discrete inner boundary \texorpdfstring{$\partial \Omega_\eps$}{PartialOmegaEps} and its measure  \texorpdfstring{$\sigma_\eps$}{SigmaEps}}\label{sss:DiscrInnerBd}
Letting~${\rm deg}_\eps(x)$ denote the degree of the vertex~$x \in \Omega_\eps$, we define the \emph{discrete inner boundary}
\begin{equation}
\partial \Omega_\eps\coloneqq \set{x\in \Omega_\eps: {\rm deg}_\eps(x)<2d} \fstop
\end{equation}
Note that~$\partial\Omega_\eps\neq \emp$, since~$0\in\Omega_\eps$. Since $\Omega$ is a bounded Lipschitz domain,
\begin{equation}\label{eq:bd:lipschitz-count}
\#\partial\Omega_\eps = O(\eps^{1-d})\fstop
\end{equation}

To a family of weights~$\set{\alpha_\eps(x)}_{x\in\partial\Omega_\eps}\subset \R^+$ we associate the discrete measure on~$\partial\Omega_\eps$
\begin{equation}\label{eq:SigmaEps}
\sigma_\eps \coloneqq \eps^{d-1} \sum_{x\in \partial\Omega_\eps} \alpha_\eps(x)\, \delta_x \fstop
\end{equation}
Everywhere in the following, we assume
\begin{itemize}
\item \emph{uniform ellipticity}: there exists $C\ge 1$ such that, \purple{for all~$\eps \in (0,1)$,}
\begin{equation}\label{eq:bd_unif_ellipticity}
C^{-1}\le \alpha_\eps(x)\le C\comma \qquad x \in \partial\Omega_\eps\semicolon 
\end{equation}
\item\emph{weak convergence to the  surface measure of $\partial\Omega$}: for any  continuous  $f \in \cC_b(\R^d)$, 
\begin{equation}\label{eq:bd_weak_conv}
\lim_{\eps\downarrow0} \abs{\scalar{\sigma_\eps}{f}-\scalar{\sigma_{\partial\Omega}}{f}}=0\fstop
\end{equation}
\end{itemize}
\noindent We call any measure~$\sigma_\eps$ as in~\eqref{eq:SigmaEps} and satisfying \eqref{eq:bd_unif_ellipticity}--\eqref{eq:bd_weak_conv} a \emph{discrete inner-boundary measure} of~$\Omega_\eps$.

\paragraph{An explicit construction of $\partial\Omega_\eps$ and $\sigma_\eps$}\label{ex:BoundaryMeasures}
\purple{Let~$(\LipOmega,N_\Omega)$ be the \emph{Lipschitz character} of~$\Omega$, see e.g.~\cite[Ch.~8]{shen_periodic_2018}.
Informally,~$N_\Omega$ is the minimal number of Lipschitz charts covering~$\Omega$ and with Lipschitz constant at most~$M_\Omega$.}
The results in~\cite[Lem.~2.1]{fan2016discrete}  and \cite[Lem.~2.4]{chen2017hydrodynamic} guarantee that  there exist $C\ge 1$ and finite measurable partitions~$\mcA_\eps$ of~$\partial\Omega$ such that
\begin{itemize}
\item the partition~$\mcA_\eps$ is \emph{uniformly elliptic}, i.e., 
\begin{equation}
C^{-1} \le  \frac{\sigma_{\partial\Omega}(A)}{\eps^{d-1}}\le C\comma\qquad A \in \mathcal A_\eps\semicolon
\end{equation}

\item for any uniformly bounded and equicontinuous family $\mcF$ in $\cC_b(\R^d)$, 
\begin{equation}
\lim_{\eps\downarrow 0} \sup_{f \in \mcF}\sum_{A \in \mcA_\eps} \sigma_{\partial\Omega}(A)\, \osc_A f=0\comma \qquad \osc_A f\eqdef \sup_A f-\inf_A f\semicolon
\end{equation}
\item for all $A \in \mcA_\eps$, there exist at most $C$ points $x \in \partial \Omega_\eps$ such that ${\rm dist}(x,A)\leq \eps (1+\LipOmega^2)^{1/2}$.
\end{itemize}
As a consequence of the definition of $\partial \Omega_\eps$, for all $x \in \partial\Omega_\eps$, there exists $A \in \mathcal A_\eps$ such that  ${\rm dist}(x,A)\le \eps$, and at most $C$ sets $A \in \mathcal A_\eps$ such that ${\rm dist}(x,A)\le \eps$.
Hence,  one can always associate to each~$A \in \mcA_\eps$ a non-empty set  $\partial\Omega_\eps^A\subset \partial\Omega_\eps$ such that
\[
\dist_{\rm H}(\partial\Omega_\eps^A,A)\leq \eps (1+\LipOmega^2)^{1/2} \comma \qquad \bigcup_{A\in \mcA_\eps} \partial\Omega_\eps^A= \partial \Omega_\eps\fstop
\]
Finally, letting
\begin{equation}
\alpha_\eps(x)\coloneqq \car_{\partial\Omega_\eps}(x) \sum_{\substack{A \in \mcA_\eps\\ \partial\Omega_\eps^A\ni x}}  \frac{1}{\#\partial\Omega_\eps^A}\, \frac{\sigma_{\partial\Omega}(A)}{\eps^{d-1}}\comma\qquad x \in \Omega_\eps\comma
\end{equation}
the corresponding discrete inner-boundary measure~$\sigma_\eps$ satisfies~\eqref{eq:bd_unif_ellipticity}--\eqref{eq:bd_weak_conv} by construction.

\subsubsection{The discrete outer boundary \texorpdfstring{$\partial_e\Omega_\eps$}{PartialEOmegaEps} and its measure \texorpdfstring{$\sigma_{e,\eps}$}{SigmaEEps}}
Next to the set~$\partial\Omega_\eps$ and its measure $\sigma_\eps$, we consider the \emph{discrete outer boundary} $\partial_e\Omega_\eps$ of $\Omega_\eps$ with its associated measure~$\sigma_{e,\eps}$.
Precisely, let~$\partial_e\Omega_\eps \subset \R^d\setminus \Omega$ be any \purple{(up to) countable} set.
For each $x \in \partial\Omega_\eps$, we let $q_\eps(x,\emparg)$ be a  probability kernel on $\partial_e\Omega_\eps$ satisfying 
\begin{equation}\label{eq:bd_ext}
	\lim_{\eps\downarrow 0}\sup_{x\in \partial\Omega_\eps} \sum_{z\in \partial_e\Omega_\eps} q_\eps(x,z) \abs{z-x}=0\comma
\end{equation}
and we define
\begin{equation}
\sigma_{e,\eps}\eqdef
\eps^{d-1}\sum_{z\in \partial_e\Omega_\eps} \tonde{\sum_{x\in \partial\Omega_\eps}\alpha_\eps(x)\, q_\eps(x,z)} \delta_z 
\fstop
\end{equation}
\purple{Note that the precise definition of the set $\partial_e \Omega_\eps$ is not important: what matters is the behaviour of the probability kernel $q_\eps$ as in \eqref{eq:bd_ext}.}
\begin{remark}\label{rem:bd_ext}
By~\eqref{eq:bd_ext}, the convergence in~\eqref{eq:bd_weak_conv} holds true with~$\sigma_{e,\eps}$ in place of~$\sigma_\eps$. 
\end{remark}

Whenever~$x\in \partial\Omega_\eps$, $z\in \partial_e\Omega_\eps$ and~$q_\eps(x,z)\neq 0$, we write~$z \sim x$ and we set
\begin{equation}
\axz\coloneqq \alpha_\eps(x)\, q_\eps(x,z)\fstop
\end{equation}

\begin{remark}
	Several choices of $\partial_e\Omega_\eps$ and $\set{q_\eps(x,\emparg)}_{x\in \partial\Omega_\eps}$ (thus, of~$\sigma_{e,\eps}$) are possible.
For instance,~$\partial_e\Omega_\eps$ may be chosen as~$\eps\Z^d \cap (\R^d\setminus\overline\Omega)$ and $q_\eps(x,\emparg)$ as the uniform measure on points in~$\partial_e\Omega_\eps$ at distance from $x\in\partial\Omega_\eps$ at most $\eps (1+\LipOmega^2)^{1/2}$.
\end{remark}

Finally, let us denote by~$\overline\Omega_\eps\eqdef \Omega_\eps\cup \partial_e \Omega_\eps$ the totality of~$\Omega_\eps$ and its outer boundary points.
A graphic representation of the lattice approximations of a Lipschitz domain is given in Figure~\ref{fig:Domains}.

\begin{figure}[tb!]
\begin{subfigure}{.25\textwidth}
\includegraphics[width=\linewidth]{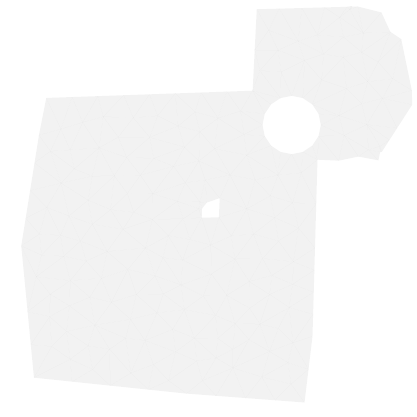}
\caption{$\Omega\subset \quadre{-0.5,0.5}^{\times 2}$}
\label{fig:Domains:1}
\end{subfigure}
\qquad\qquad
\begin{subfigure}{.25\textwidth}
\includegraphics[width=\linewidth]{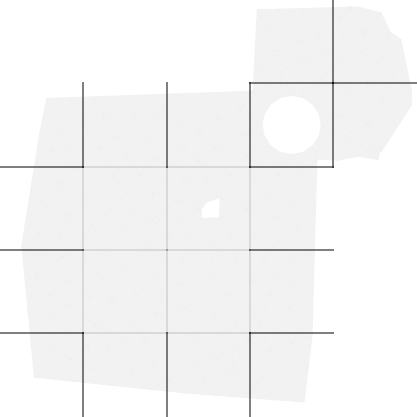}
\caption{$\eps=0.2$}
\label{fig:Domains:2}
\end{subfigure}
\qquad\qquad
\begin{subfigure}{.25\textwidth}
\includegraphics[width=\linewidth]{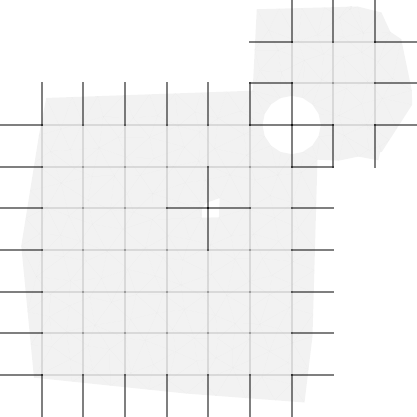}
\caption{$\eps=0.1$}
\label{fig:Domains:3}
\end{subfigure}

\vspace{.5cm}

\begin{subfigure}{.25\textwidth}
\includegraphics[width=\linewidth]{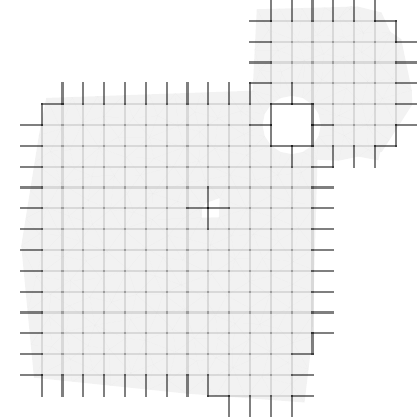}
\caption{$\eps=0.05$}
\label{fig:Domains:4}
\end{subfigure}
\qquad\qquad
\begin{subfigure}{.25\textwidth}
\includegraphics[width=\linewidth]{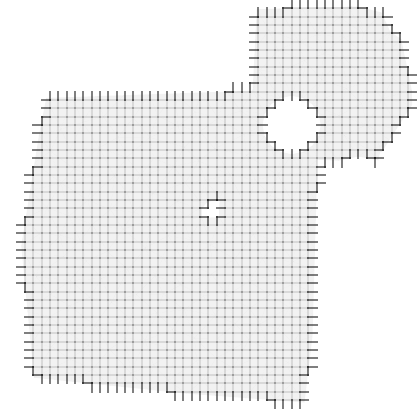}
\caption{$\eps=0.02$}
\label{fig:Domains:5}
\end{subfigure}
\qquad\qquad
\begin{subfigure}{.25\textwidth}
\includegraphics[width=\linewidth]{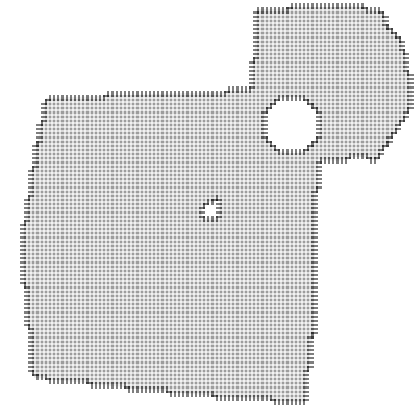}
\caption{$\eps=0.01$}
\label{fig:Domains:6}
\end{subfigure}

\caption{A bounded Lipschitz domain~$\Omega$, in light gray, and its lattice approximations for different values of~$\eps$.
The set~$\Omega_\eps^*$ is \emph{not} connected in Fig.s~\ref{fig:Domains:2},~\ref{fig:Domains:3}, while it is connected in Fig.s~\ref{fig:Domains:4}--\ref{fig:Domains:6}.
\purple{For sufficiently small $\eps \in (0,1)$, the outer boundary~$\partial_e\Omega_\eps$ has been chosen to be the set of points in~$\eps\Z^2\cap \big( \R^2 \setminus \overline \Omega \big)$ having at least one neighbor in $\overline \Omega$}.
Inner edges are depicted in light gray, outer edges in dark gray.
}
\label{fig:Domains}
\end{figure}

\subsection{Symmetric exclusion and inclusion processes in contact with reservoirs}\label{sec:particle_systems}
For~$\sigma=\pm 1$ define the \emph{configuration space}~$\Xi^{\eps,\sigma}$ by~$\Xi^{\eps,-1}\eqdef \set{0,1}^{\Omega_\eps}$ and~$\Xi^{\eps,1}\eqdef \N_0^{\Omega_\eps}$.
For~$\eta\in \Xi^{\eps,\sigma}$, we interpret $\eta(x)$ as the number of particles at $x \in \Omega_\eps$, and further  denote by~$\eta^{x,\emparg}$, resp.~$\eta^{\emparg,x}$, the \emph{annihilation}, resp.\ \emph{creation}, of a particle at $x\in \Omega_\eps$, viz.
\begin{align*}
\eta^{x,\emparg}(z)\eqdef& \begin{cases}\eta(x)-1 & \text{if } z=x, \eta(x)>0\\ \eta(z) & \text{otherwise} \end{cases}\comma
\\
\eta^{\emparg,x}(z)\eqdef& \begin{cases}\eta(x)+1 & \text{if } z=x,\sigma=1 \text{ or } z=x,\sigma=-1, \eta(x)=0 \\
\eta(x)  & \text{otherwise} \end{cases} \fstop
\end{align*}
Similarly, we let
$\eta^{x,y}\eqdef (\eta^{x,\emparg})^{\emparg,y}$. 

\subsubsection{Particle dynamics}\label{sec:particle_dynamics}
Throughout this work, unless stated otherwise, we fix~$\beta \in \R$ and~$\vartheta\in \Cb(\R^d)$ with~$\vartheta\in [0,1]$ if~$\sigma=-1$, $\vartheta\in \R_0^+$ if~$\sigma=1$.	
We consider the continuous-time Markov processes~$\eta^{\eps,\beta,\sigma,\vartheta}_t$ with state spaces~$\Xi^{\eps,\sigma}$ and laws~$\tseq{\mbbP^{\eps,\beta,\sigma,\vartheta}_\eta}_{\eta\in\Xi^{\eps,\sigma}}$ defined by the \emph{infinitesimal generator}~$\cL^{\eps,\beta,\sigma,\vartheta}$ acting on $f:\Xi^{\eps,\sigma}\to \R$ as
\begin{align}\label{eq:gen}
\ttonde{\cL^{\eps,\beta,\sigma,\vartheta}f}(\eta)&\coloneqq \eps^{-2}\sum_{x\in\Omega_\eps} \sum_{\substack{y\in \Omega_\eps\\ y\sim x}} \eta(x)\ttonde{1+\sigma \eta(y)}\ttonde{f(\eta^{x,y})-f(\eta)}
\\
\nonumber
&\qquad +\eps^{\beta-2}\sum_{x\in \partial\Omega_\eps}\sum_{\substack{z\in \partial_e\Omega_\eps\\z\sim x}}\axz\, 	\eta(x)\ttonde{1+\sigma \vartheta(z)}\ttonde{f(\eta^{x,\emparg})-f(\eta)}
\\
\nonumber
&\qquad +\eps^{\beta-2}\sum_{x\in \partial\Omega_\eps}\sum_{\substack{z\in \partial_e\Omega_\eps\\z \sim x}}\axz\,\vartheta(z)\ttonde{1+\sigma \eta(x)}\ttonde{f(\eta^{\emparg,x})-f(\eta)}\fstop
\end{align}
For simplicity of notation, we often omit the specification of~$\sigma$ and~$\vartheta$.

The infinitesimal generators in \eqref{eq:gen} describe the evolution of  particle systems on $\Omega_\eps$,  in which particles jump and interact within $\Omega_\eps$ on a diffusive timescale,  and are created and annihilated at $\partial\Omega_\eps$ at rates proportional to $\eps^{\beta-2}$.  
The parameter $\sigma =\pm 1$ stands for two different types of interaction. When~$\sigma=-1$, the process~$\eta^{\eps,\beta}_t$ is called \emph{symmetric exclusion process} (SEP) \emph{with particle reservoirs}.
Here, particles evolve on the bulk of the system jumping to nearest-neighbor sites at rate $\eps^{-2}$, while being subject to the \emph{exclusion rule}, i.e., jumps to already occupied sites get suppressed. When~$\sigma=1$, the process~$\eta^{\eps,\beta}_t$ is called \emph{symmetric inclusion process} (SIP) \emph{with particle reservoirs}. In this case,  particles on $\Omega_\eps$   jump  to nearest-neighbor sites at rate $\eps^{-2}$, as well as join the position of each nearest-neighbor particle with the same rate.
This mechanism of mutual particle attraction goes under the name of \emph{inclusion rule}. 
For both processes,  particle creation and annihilation rules mimic those governing the evolution within the bulk, with sites~$z \in \partial_e\Omega_\eps$ representing  unlimited particle reservoirs set at  density equal to $\vartheta(z)\sum_{x\in\partial\Omega_\eps} \axz$. 

Both particle systems are well-posed. Indeed, $\eta^{\eps,\beta}_t$ is a finite-state Markov chain for~$\sigma=-1$;
when~$\sigma=1$, albeit $\Xi^\eps$ is countable, the chain is non-explosive (see, e.g., \cite[Prop.\ 2.1]{franceschini2020symmetric}).

\begin{figure}
\begin{subfigure}{.25\textwidth}
\includegraphics[width=\linewidth]{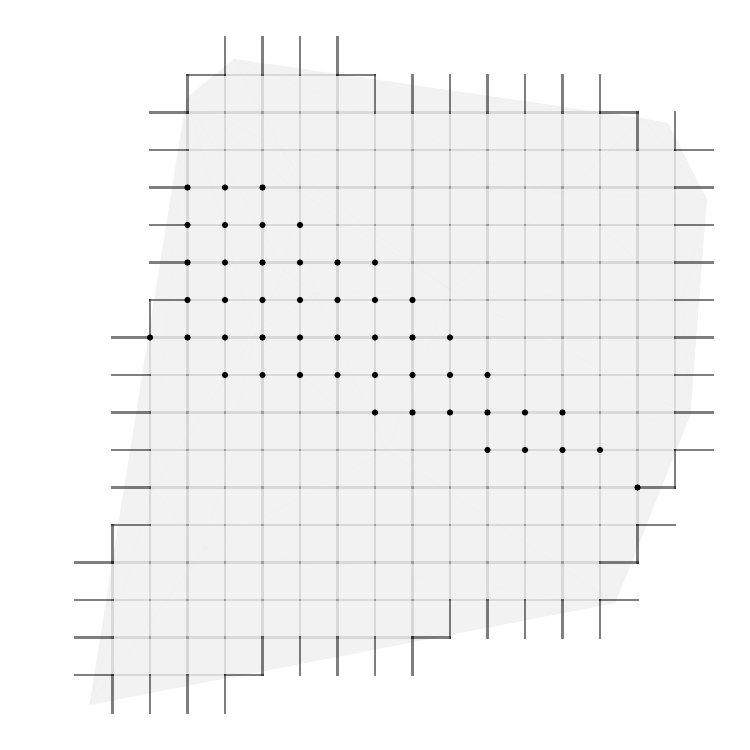}
\caption{$t=0$}
\end{subfigure}
\qquad\qquad
\begin{subfigure}{.25\textwidth}
\includegraphics[width=\linewidth]{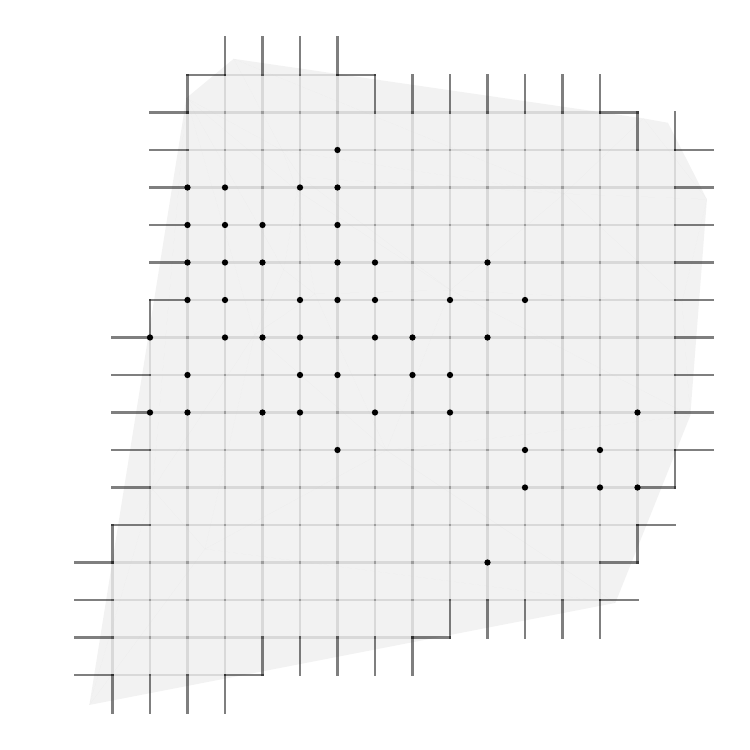}
\caption{$t=1$}
\end{subfigure}
\qquad\qquad
\begin{subfigure}{.25\textwidth}
\includegraphics[width=\linewidth]{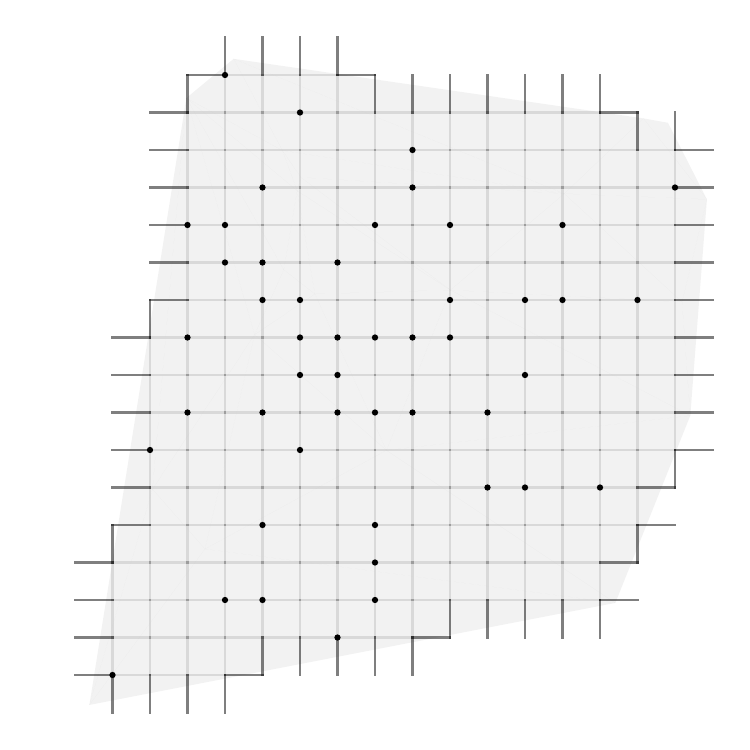}
\caption{$t=100$}
\end{subfigure}
\caption{A time sample of a fixed SEP path with~$\beta=100$ and~$\vartheta\equiv1/3$.}
\end{figure}

\begin{figure}
\begin{subfigure}{.25\textwidth}
\includegraphics[width=\linewidth]{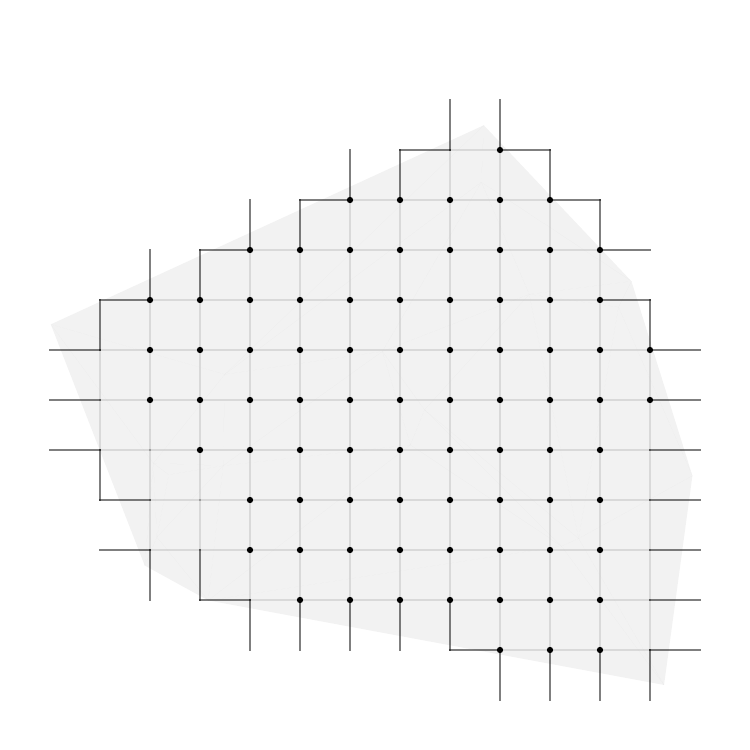}
\caption{$t=0$}
\end{subfigure}
\qquad\qquad
\begin{subfigure}{.25\textwidth}
\includegraphics[width=\linewidth]{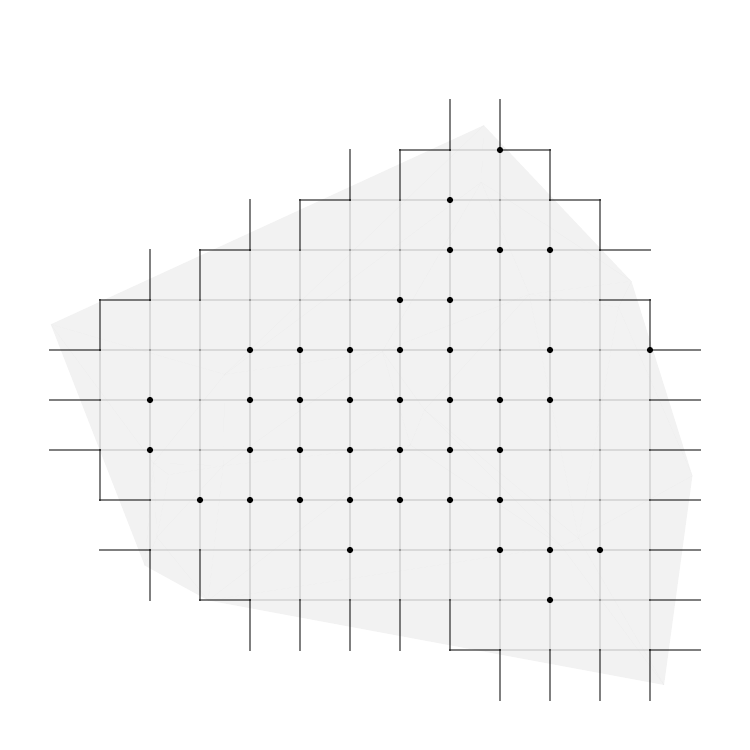}
\caption{$t=1$}
\end{subfigure}
\qquad\qquad
\begin{subfigure}{.25\textwidth}
\includegraphics[width=\linewidth]{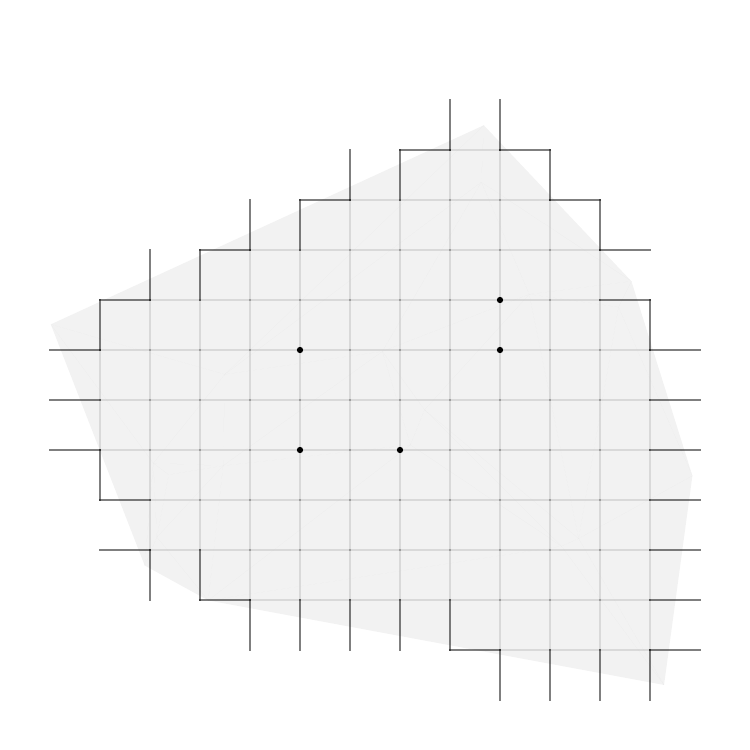}
\caption{$t=100$}
\end{subfigure}
\caption{A time sample of a fixed SEP path with~$\beta=-1$ and~$\vartheta\equiv0$.}
\end{figure}

\subsubsection{Stationary non-equilibrium states}	\label{sec:ness}
For both SEP and SIP described above, it is well-known (see, e.g., \cite[\S3]{floreani_boundary2020}) that there exists a unique invariant measure, denoted by $\nu_\stat^{\eps,\beta}$. 
Unless the function $\vartheta$ is  constant on $\partial_e\Omega_\eps$ ---~in which case $\nu^{\eps,\beta}_\stat$ is fully available and in product form~---, these invariant measures are in general \emph{not} explicit.
(An exception to this is represented by  SEP and its asymmetric variants on one-dimensional lattices, for which a special \emph{matrix formulation} due to Derrida \emph{et al.}~\cite{derrida_exact_1993-1} is available.)

\begin{remark}[SEP$(\alpha)$ and SIP$(\alpha)$] Variants of SEP and SIP parameterized by $\alpha$ (in $\N$ if $\sigma = -1$, in $\R^+_0$ if $\sigma =1$) have also been recently considered (see, e.g., \cite{floreani_boundary2020,franceschini2020symmetric}). Our models correspond to the choice $\alpha\equiv 1$; however, all our main results and proofs carry over to this most general setting with no substantial change. 
\end{remark}

\section{Main results}\label{sec:main-results}
Everywhere in the following  we denote by~`$\bd$' either of the abbrevations~`$\Neu$', for \emph{Neumann boundary conditions},~`$\Dir$' for \emph{Dirichlet boundary conditions}, or~`$\varrho$' for \emph{Robin boundary conditions} with~$0< \varrho\in\Cb(\partial\Omega)$.
When~$\beta$ and~$\bd$ appear in the same statement, it is always tacitly understood that
\begin{align*}
\bd=\Dir \quad \text{if\quad$\beta<1$}\comma \qquad \bd=\varrho\eqdef 1 \quad \text{if\quad$\beta=1$}\comma \qquad \bd=\Neu \quad \text{if\quad$\beta>1$}\fstop
\end{align*}
In particular, in all convergence statements for Robin boundary conditions (Thm.s~\ref{t:MainSemigroups}, \ref{th:harmonic_conv},  \ref{t:MainHydrodynLim}, \ref{th:hydrostatic}, and~\ref{th:fluctuations-stat}), the function~$\varrho$ is constantly equal to~$1$.

\subsection{Duality, random walks and Brownian motions}\label{sec:duality-RW-BM}
The particle systems~$\eta^{\eps,\beta}_t$ introduced in~\S\ref{sec:particle_systems} are in duality with purely absorbing interacting (labeled) particle systems~$\mbfX^{\eps,\beta,k}_t$, with~$k\in \N$ (see, e.g., \cite[\S2.2]{floreani_boundary2020}). As we show in \S\ref{s:proofs-IPS} below, these  dual processes  play a key role in our	 proofs of  scaling limits for the particle systems $\eta^{\eps,\beta}_t$.

In the `bulk'~$\Omega_\eps$, the systems~$\mbfX^{\eps,\beta,k}_t$ follow the same interaction rules as~$\eta^{\eps,\beta}_t$.
At the inner boundary~$\partial\Omega_\eps$ however, the interaction of~$\mbfX^{\eps,\beta,k}_t$ with the outer boundary~$\partial_e\Omega_\eps$ only consists of the absorbtion of particles at~$\partial_e\Omega_\eps$;
furthermore, the jump rates do depend on the underlying structure of $\overline\Omega_\eps$ and its discrete outer-boundary measure $\sigma_{e,\eps}$, but \emph{not} on the parameter $\vartheta \in \cC_b(\R^d)$.

In this section, we only consider one-particle dual systems (i.e., $X^{\eps,\beta}_t=\mbfX^{\eps,\beta,k=1}_t$ consisting of the random walks in~$\overline\Omega_\eps$ defined in~\S\ref{sss:RandomWalks} below), and  refer to \S\ref{s:duality} for  the complete definitions and properties of the dual processes for general~$k\in \N$.
In~\S\ref{sss:FCLT} we prove a functional central limit theorem for $X^{\eps,\beta}_t$, with the Brownian motion on~$\Omega$ with boundary condition~$\bd$ (see~\S\ref{sss:BrownianMotions}) as its limiting process.
	
\subsubsection{Random walks}\label{sss:RandomWalks}
For~$f\in\R^{\overline\Omega_\eps}$ and every~$x,y\in\Omega_\eps$ we write
\begin{equation}
\nabla^\eps_{x,y} f\eqdef \eps^{-1}\ttonde{f(y)-f(x)}\fstop
\end{equation}
We denote by
$
\ttonde{\ttseq{X^{\eps,\beta}_t}_{t\ge 0} , \ttseq{\mbfP^{\eps,\beta}_x}_{x\in \overline\Omega_\eps}}
$
the continuous-time random walk in the Skorokhod space~$\mcD(\R^+_0;\overline\Omega_\eps)$ with infinitesimal generator
\begin{equation}\label{eq:generator_RW}
\begin{aligned}
A^{\eps,\beta}f(x)=&\ \car_{\Omega_\eps}\!(x)\,\eps^{-1}\sum_{\substack{y\in \Omega_\eps \\ y\sim x}} \nabla^\eps_{x,y}f
+ \car_{\partial\Omega_\eps}\!(x)\, \eps^{\beta-1}\sum_{\substack{z\in \partial_e\Omega_\eps\\z\sim x}}\axz \nabla^\eps_{x,z}f\fstop
\end{aligned}
\end{equation}
We further let~$\ttseq{P^{\eps,\beta}_t}_{t\ge 0}$ be the corresponding Markov semigroup on~$\R^{\overline\Omega_\eps}$, with corresponding heat kernel 
\begin{equation}
p^{\eps,\beta}_t(x,y)\coloneqq \mbfP_x^{\eps,\beta}\ttonde{X^{\eps,\beta}_t=y}  \comma \qquad x, y \in \overline\Omega_\eps\comma  t \geq 0\fstop
\end{equation}

We stress that~$A^{\eps,\beta}f\equiv 0$ on $\partial_e\Omega_\eps$, i.e., the outer boundary~$\partial_e\Omega_\eps$ is a set of  absorbing states for the random walk~$X^{\eps,\beta}_t$.	
Furthermore,~$A^{\eps,\beta}$, and thus~$P^{\eps,\beta}_t$, globally fixes the space of functions $f:\overline\Omega_\eps\to \R$ identically vanishing on~$\partial_e\Omega_\eps$.
We identify the latter space with $L^p(\Omega_\eps)$ for any $p\in [1,\infty]$, endowed with the standard norm \purple{(w.r.t.\ $\mu_\eps$, cf.\ \S\ref{sec:discrete-domain-volume-measure})}
\begin{equation}
\norm{f}_{L^p(\Omega_\eps)}^p\coloneqq \eps^d\sum_{x\in \Omega_\eps}\abs{f(x)}^p\comma \quad p\in [1,\infty) \comma \qquad \norm{f}_{L^\infty(\Omega_\eps)}\coloneqq \sup_{x\in \Omega_\eps}\abs{f(x)}\fstop
\end{equation}
With a slight abuse of notation, we further let~$A^{\eps,\beta}$ and~$P^{\eps,\beta}_t$ act in the obvious way on functions in $L^p(\Omega_\eps)$.
Note that~$A^{\eps,\beta}$ and~$P^{\eps,\beta}_t$ are self-adjoint in $L^2(\Omega_\eps)$.
Finally, recalling the definitions \purple{of the measures $\mu_\eps$ and $\sigma_\eps$ in~\S\ref{sec:discrete-domain-volume-measure} and \eqref{eq:SigmaEps}, respectively, and of the generator $A^{\eps,\beta}$ in \eqref{eq:generator_RW},}  the corresponding Dirichlet form $\mcE^{\eps,\beta}(f,g)\purple{	\eqdef\scalar{f}{-A^{\eps,\beta}g}_{L^2(\Omega_\eps)}}$  on $L^2(\Omega_\eps)$ reads as 
\begin{equation}\label{eq:dirichlet-form-RW}
	\mcE^{\eps,\beta}(f,g)\eqdef \frac{\eps^d}{2}\sum_{\substack{x,y\in\Omega_\eps\\
	x\sim y}} \nabla^\eps_{x,y}f\, \nabla^\eps_{x,y}g + \eps^{\beta-1} \scalar{\sigma_\eps}{fg}\comma \qquad f, g \in L^2(\Omega_\eps)\fstop 
\end{equation}

\subsubsection{Brownian motions}\label{sss:BrownianMotions}
In the following let~$H^1_0(\Omega)$ and~$H^1(\Omega)$ denote the standard Sobolev spaces on~$\Omega$.
Recall that the trace operator~$\emparg\restr{\partial\Omega}$ is well-defined on~$H^1(\Omega)$ as a \emph{compact} operator~$\emparg\restr{\partial\Omega}\colon H^1(\Omega)\to L^2(\partial\Omega)$.
This follows from, e.g., \cite{gagliardo1957caratterizzazioni}, combining~\cite[Teor.~1.I, Eqn.~(1.4)]{gagliardo1957caratterizzazioni} and the observation in~\cite[\S4\textparagraph2]{gagliardo1957caratterizzazioni}.

\paragraph{Dirichlet forms, Laplacians and heat semigroups}
Let~$\varrho\colon \partial \Omega\to [0,\infty]$ be a continuous function.
Denote by~$(\mcE^\varrho, H^1(\Omega))$ the bilinear form
\begin{align}\label{eq:FormRobin}
\mcE^{\varrho}(f,g)\eqdef & \int_\Omega \nabla f\cdot \nabla g\, \diff\mu_\Omega +\int_{\partial\Omega} \varrho f\restr{\partial \Omega}\, g\restr{\partial \Omega} \diff \bdvol \comma\qquad f,g\in H^1(\Omega)\comma
\end{align}
well-defined and closable on~$L^2(\Omega)$ by compactness of the trace operator. 
Its closure, denoted by~$\ttonde{\mcE^\varrho,\dom{\mcE^\varrho}}$, is a Dirichlet form on~$L^2(\Omega)$, with (negative) generator the  self-adjoint operator~\mbox{$\ttonde{\Delta^\varrho,\dom{\Delta^\varrho}}$}, the  \emph{Robin Laplacian with boundary condition~$\varrho$}.
We further let~\mbox{$\ttonde{\Delta^\Dir,\dom{\Delta^\Dir}}$}, resp.\ \mbox{$\ttonde{\Delta^\Neu,\dom{\Delta^\Neu}}$}, denote the familiar \emph{Dirichlet}, resp.\ \emph{Neumann}, \emph{Laplacian} on~$\Omega$, resp.\ corresponding to the cases~$\varrho\equiv +\infty$ and~$\varrho\equiv 0$.
The domains~$\dom{\mcE^\Dir}$ and~$\dom{\mcE^\Neu}$ respectively coincide with~$H^1_0(\Omega)$ and~$H^1(\Omega)$.

Let~$\Delta^\bd $ denote either~$\Delta^\Dir$,~$\Delta^\Neu$, or~$\Delta^\varrho$.
We respectively write~$\ttonde{\mcE^\bd ,\dom{\mcE^\bd }}$, $\ttseq{P^\bd _t}_{t\ge 0}$, and $R^\bd_\zeta\eqdef(\zeta-\Delta^\bd)^{-1}$  with $\zeta >0$,   for the corresponding Dirichlet form, heat semigroup and resolvent on $L^2(\Omega)$. 

\paragraph{Brownian motions} In the following, we denote by~$\ttseq{B^\bd_t}_{t\geq 0}$ the Brownian motion on~$\Omega$ with boundary condition~$\bd$, i.e., the Markov diffusion process properly associated with the Dirichlet form~$\ttonde{\mcE^\bd,\dom{\mcE^\bd}}$ on~$L^2(\Omega)$. 
For different choices of~$\bd$ we have therefore that $B^\bd_t$ is the Brownian motion on~$\overline\Omega$
\begin{itemize}[leftmargin=2em]
\item[($\Neu$)] with normally reflected boundary conditions, \purple{in the generality of Lipschitz domains see e.g.,~\cite{bass_hsu1991,Che92,fukushima_construction1996};}
\item[($\varrho$)] with Robin boundary conditions driven by~$\varrho$, \purple{in the generality of Lipschitz domains see e.g.,~\cite{Mat19};}
\item[($\Dir$)] stopped at~$\partial\Omega$.
\end{itemize}
Finally, we denote by $\ttseq{\mbfP^\bd_x}_{x\in \overline\Omega}$ and $\ttseq{\mbfE^\bd_x}_{x\in\overline\Omega}$ their laws and the corresponding expectations.

\subsubsection{Convergences}\label{sss:FCLT}
Let~$\Pi_\eps\colon \cC(\overline\Omega)\to L^\infty(\Omega_\eps)$ be defined by~$(\Pi_\eps f)(x)=f(x)$ for every~$x\in \Omega_\eps$.
Whenever needed,~$\Pi_\eps f$ is assumed to be extended by~$0$ outside of~$\Omega_\eps$.
Let~$\cC^\bd$ be either~$\cC(\overline\Omega)$ (for Neumann and Robin boundary conditions) or~$\cC_0(\Omega)$ (for Dirichlet boundary conditions), always regarded as a Banach space endowed with the uniform norm.
Finally, denote by~$P^{\bd,c}_t$ the part of~$P^\bd_t$ on~$\mcC^\bd$.
The operators~$P^{\bd,c}_t$ satisfy~$P^{\bd,c}_t\colon \mcC^\bd\to\mcC^\bd$ and form $\mcC^0$-semigroups on~$\mcC^\bd$.
We denote by~$\ttonde{\Delta^{\bd,c},\dom{\Delta^{\bd,c}}}$ the corresponding $\mcC^\bd$-generators.
We refer to~\S\ref{ss:CbLaplacians} for the details of these constructions.

We are now ready to state our first convergence result, in which we write, cf.~\cite{ethier_kurtz_1986_Markov}:
\begin{equation}\label{eq:EKConvergence}
f_\eps \to f \qquad \text{if } \qquad f_\eps\in L^\infty(\overline\Omega_\eps)\comma f\in \cC(\overline\Omega)\comma \quad \lim_{\eps\downarrow 0} \norm{f_\eps-\Pi_{\eps} f}_{L^\infty(\Omega_\eps)}=0\fstop
\end{equation}

\begin{theorem}[\S\ref{sec:semigroup-conv-proofs}]\label{t:MainSemigroups}
Fix~$f\in \cC^\bd$. Then, $P^{\eps,\beta}_t \Pi_\eps f\to P^{\bd,c}_t f$ locally uniformly in~$t$ on~$\R^+_0$.
\end{theorem}

\purple{In Theorem~\ref{t:Equivalence} (see \S\ref{sec:proof-Equivalences} in Appendix) we provide many equivalent assertions to the one in Theorem~\ref{t:MainSemigroups}, which will prove crucial in establishing this result.}

\subsection{Harmonic profiles}\label{sec:harmonic-profiles}
Recall the definition of $\vartheta \in \cC_b(\R^d)$ from \S\ref{sec:particle_dynamics}.

We define the \emph{discrete harmonic measure}
\begin{equation}
p^{\eps,\beta}_\infty(x,y)\eqdef \lim_{t\to \infty}p^{\eps,\beta}_t(x,y)\comma\qquad x, y \in \overline\Omega_\eps\comma
\end{equation}
and the \emph{discrete harmonic profile}~$h^{\eps,\beta}$ on~$\overline\Omega_\eps$ with boundary condition~$\vartheta$
 on~$\partial_e\Omega_\eps$
\begin{equation}	\label{eq:harmonic_profiles_discrete}
h^{\eps,\beta}(x)\eqdef \sum_{y\in \overline \Omega_\eps} p^{\eps,\beta}_\infty(x,y)\, \vartheta(y)\comma \qquad x\in \overline\Omega_\eps\fstop
\end{equation}
Equivalently, $h^{\eps,\beta}:\overline\Omega_\eps\to \R_0^+$ is the unique solution to the boundary value problem
\begin{align}\label{eq:bvp}
\left\{\begin{array}{rcll}
A^{\eps,\beta}h&=&0 &\text{on } \overline\Omega_\eps
\\
h&=&\vartheta \qquad&\text{on } \partial_e\Omega_\eps
\end{array}	\right.\fstop
\end{align}

We denote by~$h^\bd\in\mcC(\overline\Omega)$ the (continuum) \emph{harmonic profile} on~$\Omega$, i.e., the distributional solution to
\begin{equation}\label{eq:Dir_problem}
		\begin{cases}
			\Delta h=0 &\text{on}\ \Omega \\
			\mcB_\bd &\text{on}\ \partial\Omega
		\end{cases}\fstop
	\end{equation}
where~$\mcB_\bd$ denotes either of the weak boundary conditions, understood in the $\bdvol$-a.e.\ sense,
\begin{equation}\label{eq:BoundaryCondition}
	(\Dir) \quad u=\vartheta\comma\qquad
	(\varrho) \quad \partial_\mbfn u+\varrho (u-\vartheta)=0 \comma \qquad 
	(\Neu) \quad \begin{dcases}\partial_\mbfn u=0\\
		 \displaystyle\scalar{\bdvol}{u}=\scalar{\bdvol}{\vartheta}
	 \end{dcases}\fstop
\end{equation}
Indeed, since $\Omega$ is a bounded Lipschitz domain, suitable solutions $h^\bd$ exist and are unique (see e.g.~\cite[\S{II.4}]{DauLio90} or~\cite{AreBen99} for~$\bd=\Dir$, \cite{fan2016discrete} for~$\bd=\varrho$); more specifically:
\begin{itemize}
	\item[$(\Neu)$] $h^\Neu$ is a constant and coincides with the $\bdvol$-average of~$\vartheta$ on~$\partial\Omega$;
	\item[$(\varrho)$] Letting $t\mapsto L^\Neu_t$ denote the boundary local time of the  Brownian motion $B^\Neu_t$ on $\Omega$, the function~$h^\varrho$ satisfies, cf.~\cite[\S3]{fan2016discrete},
	\begin{equation}	\label{eq:harmonic_profile_R_LT}
		h^\varrho(x)= \mbfE^\Neu_x\quadre{\int_0^\infty (\varrho \vartheta)(B^\Neu_t) \exp\tonde{-\int_0^t \varrho\ttonde{B^\Neu_s} \dd 	L^\Neu_s} \dd L^\Neu_t}\comma\qquad x \in \Omega\semicolon
	\end{equation}
	\item[$(\Dir)$]  By Riesz--Markov--Kakutani Representation Theorem and the Maximum principle,
	 the correspondence~$\vartheta \mapsto h^\Dir$ uniquely defines the \emph{harmonic measures}~$\tset{p^\Dir_\infty(x,\emparg)}_{x\in \Omega}$ of the Brownian motion $B^\Dir_t$,  probability measures concentrated on $\partial\Omega$ such that
	\begin{align}
		h^\Dir(x)= \int_{\partial\Omega} \vartheta(y)\, p^\Dir_\infty(x,\diff y)\comma\qquad x \in \Omega \fstop
\end{align}
Further, the solution~$h^\Dir$ satisfies~$h^\Dir\in \cC(\overline\Omega)\cap \cC^\infty(\Omega)$ and thus $h^\Dir\vert_{\partial\Omega}=\vartheta\vert_{\partial\Omega}$ everywhere on~$\partial\Omega$, see~\cite[\S6.1 \& Ex.\ 6.1.2.b]{AreBatHieNeu11}.
\purple{However, in general} $D^k h^\Dir$ may be unbounded on $\Omega$ for some~$k\in \N$ due to the lack of boundary regularity; in particular,  boundary data in~$\mcC(\partial\Omega)$  only ensure~$h^\Dir \in H^{1/2}(\Omega)$, see~\cite[Thm.~5.1]{JerKen95}.
\end{itemize}

For the harmonic profiles, we prove the following convergence result.
Recall the definition of convergence of functions in~\eqref{eq:EKConvergence}.

\begin{theorem}[Convergence of harmonic profiles,~\S\ref{sec:harmonic_conv}]\label{th:harmonic_conv} 
$h^{\eps,\beta}\to h^\bd$ for every~$\beta\in\R$.
\end{theorem}

\begin{remark}[Relaxing conditions \eqref{eq:bd_unif_ellipticity}, \eqref{eq:bd_weak_conv} and \eqref{eq:bd_ext}]
	Not all the assumptions  on the discrete approximations of the Lipschitz boundary are strictly necessary to prove  Theorems~\ref{t:MainSemigroups} and  \ref{th:harmonic_conv}. More specifically, when $\beta > 1$, only the second inequality in \eqref{eq:bd_unif_ellipticity} is needed to prove Theorem~\ref{t:MainSemigroups}, while Theorem~\ref{th:harmonic_conv} holds true even with a weaker version of \eqref{eq:bd_weak_conv} (with any continuous linear functional $\varSigma:\cC(\partial\Omega)\to \R$ in place of $\sigma_{\partial\Omega}$). Similarly, when $\beta=1$, our results easily generalize to the case of $0<\varrho \in \cC(\partial \Omega)$, by replacing $\sigma_{\partial \Omega}$ in \eqref{eq:bd_weak_conv} with $\varrho\, \sigma_{\partial \Omega}$. When $\beta <1$, Theorem \ref{t:MainSemigroups} uses only  \eqref{eq:bd_unif_ellipticity} and \eqref{eq:bd_ext}.  
\end{remark}

\subsection{Hydrodynamic and hydrostatic limits}\label{sec:HDL-HSL}
In \S\ref{sec:HeatEquation}, we introduce the heat equations, the notions of mild solution and some of its main properties. The assumptions and statement of the hydrodynamic limits (Thm.~\ref{t:MainHydrodynLim}) is the content of \S\ref{sec:HDL}.  In \S\ref{sec:examples+HSL}, we present some examples to which Theorem \ref{t:MainHydrodynLim} applies, as well as the hydrostatic limits (Thm.~\ref{th:hydrostatic}).
Finally,  \S\ref{sss:local-equilibrium} is concerned with scaling limits of stationary correlations (Thm.~\ref{t:LocalEquilibrium}).

\subsubsection{Heat equations}\label{sec:HeatEquation}
Everywhere in this section, fix~$T>0$, and set
\begin{align*}
\Omega_T\eqdef \Omega\times [0,T]\comma\qquad \mathring\Omega_T\eqdef \Omega\times (0,T)\comma \qquad \mathring\partial_T\Omega\eqdef \partial\Omega\times (0,T)\fstop
\end{align*}
Further let~$\msM^\bd\eqdef (\mcC^\bd)^*$ denote the topological dual of~$\mcC^\bd$, always endowed with its weak* topology.
By Riesz--Markov Representation Theorem, we have the standard identifications~$\msM^\Dir\cong\ttonde{\Mb(\Omega),\T_\mrmv}$ and~$\msM^\Neu=\msM^\varrho\cong \ttonde{\Mb(\overline\Omega),\T_\mrmn}$, where~$\T_\mrmv$, resp.~$\T_\mrmn$, denotes the \emph{vague}, resp.~\emph{narrow}, topology on bounded (Radon) measures.
By, e.g.,~\cite[Thm.~4.2, Lem.~4.5]{kallenberg_random_2017},~$\msM^\bd$ is a Polish space.

We consider the following boundary-value problems for the heat equation with $\bd$-boundary conditions, $\vartheta \in \cC_b(\R^d)$ and~$\pi_0\in \msM^\bd$ as  boundary and initial data, respectively: 
\begin{align}\tag{$\mathrm{H}_{\bd,T}$}\label{eq:HeatEquation}
\begin{cases}
(\partial_t - \Delta) u=0 & \text{in } \mathring \Omega_T\comma
\\
u_0=\pi_0 & \text{in } \msM^\bd\comma
\\
\mathcal B_\bd & \text{on } \mathring\partial_T\Omega\comma
\end{cases}
\end{align}
where~$\mathcal B_\bd$ denotes either of the boundary conditions~\eqref{eq:BoundaryCondition} with~$u=u_t$ and $\varrho = 1$.

We now formulate a precise definition of $\msM^\bd$-valued solution to \eqref{eq:HeatEquation}.
To this end, let~$(P^{\bd,c}_t)^*\colon \msM^\bd\to\msM^\bd$ be the dual operator to~$P^{\bd,c}_t\colon \mcC^\bd\to\mcC^\bd$; note that, by the standard theory of $\mcC^0$-semigroups,~$(P^{\bd,c}_t)^*$ is a continuous\footnote{Since~$\mcC^\bd$ is not reflexive,~$(P^{\bd,c}_t)^*$ is only weakly*-continuous. This motivates our choice to endow~$\msM^\bd$ with its weak*-topology.} (in~$t$) positivity preserving semigroup.
\begin{definition}\label{def:HeatEquation}
We say that~$u\colon[0,T]\to\msM^\bd$ is \emph{the} \emph{mild solution} to~\eqref{eq:HeatEquation} starting at~$u_0=\pi_0\in\msM^\bd$ if
\begin{equation}\label{eq:mild-sol}
u_t = h^\bd \mu_\Omega+(P^{\bd,c}_t)^*\ttonde{\pi_0-h^\bd\mu_\Omega} \comma \qquad t\in [0,T]\fstop
\end{equation}
\end{definition}

In the following proposition, let~$E$ denote either~$L^2(\Omega)$ or~$\mcC^\bd$, and let~$\ttonde{\Delta^{\bd,E},\dom{\Delta^{\bd,E}}}$ be the $E$-Laplacian with boundary condition~$\bd$. Its proof is a consequence of the weak$^*$-continuity of $(P_t^{\bd,c})^*$, the regularity of $h^\bd$ (see~\S\ref{sec:harmonic-profiles}), and  \cite[Cor.~3.7.21]{AreBatHieNeu11}. 

\begin{proposition}\label{p:HeatEquation}
The mild solution to~\eqref{eq:HeatEquation} belongs to~$\mcC\ttonde{[0,T];\msM^\bd}$.
If additionally~$u_0=\rho_0\,\mu_\Omega$ with~$\rho_0-h^\bd\in E$, then the mild solution satisfies~$u_t=\rho_t \, \mu_\Omega$ for all~$t\in [0,T]$, and~$\rho_\emparg-h^\bd\in \mcC\ttonde{[0,T];E}\cap \mcC^\infty\ttonde{(0,T); E}\cap \mcC\ttonde{(0,T];\dom{\Delta^{\bd, E}}}$ is the unique solution to the Cauchy problem $ v'(t) = \Delta^{\bd,E} v(t)$ for $t>0$, and $v(0)=\rho_0 - h^\bd$.
\end{proposition}

Let us observe that the techniques we make use of in this work rely on the Lipschitz property of~$\Omega$ in an essential way, and we expect (a large part of) our results \emph{not} to hold if $\Omega$ is non-Lipschitz. \purple{We refer to 	 \S\ref{sec:generality-Lipschitz} below for more details.}

\subsubsection{Hydrodynamic limits}\label{sec:HDL}
Let	 $\ttseq{\nu_\eps}_\eps$ be a family of probability distributions on the configuration spaces $\ttseq{\Xi^\eps}_\eps$,   $\ttseq{\E_{\nu_\eps}}_\eps$  being the corresponding expectations.
We make the following assumptions:
\begin{assumption}[Second-moment bounds]\label{as:secondII} We assume that
\begin{align}\label{eq:4th_moment_1}
\limsup_{\eps\downarrow  0} \mbbE_{\nu_\eps}\quadre{\ttonde{\eps^d\sum_{x\in\Omega_\eps}\eta(x)}^2}<\infty\comma
&&&
\limsup_{\eps\downarrow 0} \eps^{d}\max_{x\in \Omega_\eps}\mbbE_{\nu_\eps}\quadre{\eta(x)^2}<\infty\fstop
\end{align}
\end{assumption}
For SEP ($\sigma=-1$) Assumption~\ref{as:secondII} trivially holds.

\begin{assumption}[Weak law of large numbers]\label{as:WLLN}
There exists $\pi_0 \in \msM^\bd$ such that
\begin{equation}
\lim_{\eps\downarrow 0}\nu_\eps\set{\eta\in \Xi^\eps: \abs{\scalar{\pi_0}{f}-\eps^d\sum_{x\in \Omega_\eps}\eta(x)\,f(x)}\geq\delta}=0\comma \qquad \delta > 0\comma f \in \mcC^\bd\fstop
\end{equation}
\end{assumption}

Let $\cX^{\eps,\beta}_t$ be the 
normalized \emph{empirical density fields} corresponding to the paths of~$\eta^{\eps,\beta}_t$,	viz., 
\begin{equation}\label{eq:DensityField}
t\in \R^+_0\longmapsto	\mcX^{\eps,\beta}_t\eqdef \eps^d\sum_{x\in \Omega_\eps} \eta^{\eps,\beta}_t(x)\, \delta_x\fstop 
\end{equation}
We are now ready to state the main result of this section.

\begin{theorem}[Hydrodynamic limit, \S\ref{sec:proof-HDL}]\label{t:MainHydrodynLim}
Let~$\ttseq{\nu_\eps}_\eps$ satisfy Assumptions~\ref{as:secondII} and~\ref{as:WLLN}. 
Further, for all $T>0$, let~$u\in\mcC([0,T]; \msM^\bd)$ be the mild solution of~\eqref{eq:HeatEquation} starting at~$u_0=\pi_0\in\msM^\bd$, and~$\mcX^{\eps,\beta}$ be the $\mcD([0,T];\msM^\bd)$-valued random empirical density fields as in~\eqref{eq:DensityField} starting at a configuration~$\eta_0^{\eps,\beta}=\eta \in \Xi^\eps$ randomly distributed as~$\nu_\eps$.
Then, 
\begin{align*} 
\mcX^{\eps,\beta} \xRightarrow[\ \eps \downarrow 0 \ ]{\ {\rm fdd} \ } u \comma
\end{align*}
where~$\xRightarrow[]{\ {\rm fdd} \ }$ denotes (weak)  convergence of finite-dimensional distributions.
\end{theorem}

\subsubsection{Examples and hydrostatic limits}\label{sec:examples+HSL}
We list a few relevant examples for which the hydrodynamic limit holds.

\paragraph{Product measures associated to a slowly-varying profile} Fix a continuous  (bounded)	 function $g :\overline\Omega\to \R^+_0$ (or $[0,1]$ if $\sigma=-1$) which will play the role of  initial limiting density profile.
For all~$x \in \Omega_\eps$, consider a probability distribution $\nu_\eps^x$ on $\N_0$ (or $\set{0,1}$ if $\sigma=-1$), and set $\nu_\eps\coloneqq \otimes_{x\in \Omega_\eps}\nu_\eps^x$. If we assume the moment bounds~\eqref{eq:4th_moment_1} and further that
\begin{equation}
\mbbE_{\nu_\eps}\quadre{\eta(x)}=g(x)\comma \qquad x \in \Omega_\eps\comma
\end{equation}
then $\ttseq{\nu_\eps}_\eps$ constructed above satisfies Assumption~\ref{as:WLLN} with~$\pi_0\eqdef g\mu_\Omega$.

\paragraph{Deterministic piles of particles $(\sigma=1)$} 
Fix~$x\in\Omega$, and let~$\ttseq{x_\eps}_\eps\subset \Omega$ be a family such that $x_\eps \in \Omega_\eps$ and $\lim_{\eps\downarrow 0}\abs{x-x_\eps}=0$.
Then, the distributions~$\ttseq{\nu_\eps}_\eps$, each concentrated at the particle configuration $\eta^\eps \in \Xi^\eps$ given by
\begin{equation}
\eta^\eps(y)\coloneqq \begin{cases}
\eps^{-d/2} &\text{if}\ \abs{y-x}\le \eps^{1/2}\\
0&\text{otherwise}
\end{cases}
\end{equation}
satisfy Assumption~\ref{as:secondII}, as well Assumption~\ref{as:WLLN} with~$\pi_0\eqdef v_d \,\delta_x$, where~$v_d$ denotes the volume of the Euclidean unit ball in~$\R^d$.

\paragraph{Stationary non-equilibrium states and hydrostatic limits} 
Letting  $\nu^{\eps,\beta}_\stat$ denote the (generally implicit) stationary measures described in \S\ref{sec:ness},  a standard duality argument and the boundedness of $\vartheta$ (see, e.g., \cite[App.~A]{floreani_boundary2020}) readily imply that Assumption \ref{as:secondII} holds for $\ttseq{\nu_\eps}_\eps=\ttseq{\nu_\stat^{\eps,\beta}}_\eps$.
Additionally, we show the less trivial fact that~$\ttseq{\nu^{\eps,\beta}_\stat}_\eps$ satisfies a weak law of large numbers:

\begin{theorem}[\S\ref{sec:hydrostatic}]\label{th:hydrostatic}
$\ttseq{\nu^{\eps,\beta}_\stat}_\eps$ satisfies Assumption~\ref{as:WLLN} with $\pi_0=h^\bd \mu_\Omega$.
\end{theorem}
\noindent Combined with Theorem~\ref{t:MainHydrodynLim}, this proves the so-called \emph{hydrostatic limit}, i.e., the assertion of Theorem~\ref{t:MainHydrodynLim} with~$\nu_\eps\eqdef \nu^{\eps,\beta}_\stat$ and~$\pi_0\eqdef h^\bd \mu_\Omega$, the latter being the stationary solution to the heat equation with appropriate boundary conditions.

\subsubsection{Local equilibrium for SNS}\label{sss:local-equilibrium}

As a refinement of the weak laws of large numbers presented in Theorem~\ref{th:hydrostatic}, we establish some weak forms of local equilibrium (see, e.g., \cite[\S3]{kipnis_scaling_1999}) for  SNS.

 Recall that in equilibrium (that is, whenever $\vartheta\in \cC_b(\R^d)$ is constant), $\nu^{\eps,\beta}_\stat$ is explicit and in product form,  given by products of i.i.d.\ Bernoulli, resp.\ geometric, distributions if $\sigma = -1$, resp.\ $\sigma =1$ (see, e.g., \cite{floreani_boundary2020}). For a general $\vartheta \in \cC_b(\R^d)$,  this  no longer holds, although a local approximation of this sort remains valid at large scales. The precise statement of this fact is the content of Theorem~\ref{t:LocalEquilibrium} below. In what follows, for every $\vartheta \in \cC_b(\R^d)$ given as in~\S\ref{sec:setting-models} and  corresponding 	$h^{\eps,\beta}$ as in~\eqref{eq:harmonic_profiles_discrete},  we define
 \begin{equation}
 	\nu^{\eps,\beta}_x \sim \begin{cases}
 		{\rm Bern}\tonde{h^{\eps,\beta}(x)} &\text{if}\ \sigma=-1\\
 		{\rm Geom}\tonde{\tonde{1+h^{\eps,\beta}(x)}^{-1}} &\text{if}\ \sigma = 1
 	\end{cases}\comma
 \end{equation} 
and
$
	\nu^{\eps,\beta}_\otimes \eqdef \otimes_{x\in \Omega_\eps}\, \nu^{\eps,\beta}_x$.

\begin{theorem}[Local equilibrium, \S\ref{sec:hk-conv}]\label{t:LocalEquilibrium}
Fix $k\in \N$, $k\ge 2$. Then, 
\begin{equation}
\lim_{\eps\downarrow 0}\sup_{x_1,\ldots,x_k \in \Omega_\eps}\abs{	\E_{\nu^{\eps,\beta}_\stat}\quadre{\prod_{i=1}^k \eta(x_i)}- \E_{\nu^{\eps,\beta}_\otimes}\quadre{\prod_{i=1}^k \eta(x_i)}}=0\fstop
\end{equation}
\end{theorem}
\begin{remark}[Two-point stationary correlations]\label{r:TwoPointsStationaryCorr}
	As a particular case of Theorem~\ref{t:LocalEquilibrium},
	\begin{equation}
		\lim_{\eps\downarrow 0} \sup_{\substack{x,y \in \Omega_\eps\\ x \neq y}}\abs{\E_{\nu^{\eps,\beta}_\stat}\quadre{\tonde{\eta(x)-\E_{\nu^{\eps,\beta}_\stat}\quadre{\eta(x)}}\tonde{\eta(y)-\E_{\nu^{\eps,\beta}_\stat}\quadre{\eta(y)}}}}=0\fstop
	\end{equation}
\end{remark}

\subsection{Stationary non-equilibrium fluctuations}\label{s:stat-flu} Theorem \ref{th:hydrostatic} proves  weak laws of large numbers for  the empirical density fields  at stationarity. In this section, we analyze the corresponding  fluctuations. We present these limit theorems in \S\ref{sec:FCLT-flu}, after presenting in \S\ref{sec:TestF} the function and distribution spaces needed to establish these results. 
\subsubsection{Nuclear Fr\'echet spaces of test functions and their duals}\label{sec:TestF}

We introduce here some  spaces of test functions $\cS^\bd(\Omega)$ and their duals $\cS^\bd(\Omega)'$ well-adapted to the spectral properties of~$-\Delta^\bd$.
Such spaces generalize those in, e.g., \cite[Ch.\ 11]{kipnis_scaling_1999}, \cite{landim_stationary_2006}: there	 $\Omega$ is either the $d$-dimensional torus or $[0,1]$, and $\cS^\bd(\Omega)$ and $\cS^\bd(\Omega)'$ are constructed via the explicit knowledge of an orthonormal system of Laplacian's eigenfunctions. 
	We show that the spaces~$\cS^\bd(\Omega)$ are cores for the~$\cC^\bd$-Laplacians, and that~$(\cS^\bd(\Omega),L^2(\Omega),\cS^\bd(\Omega)')$  is a Gel'fand triple \emph{compatible}	with~$P^\bd_t$ \cite[Def.~1.3.5]{kallianpur_xiong_1995}; thus,  these spaces are natural candidates for the study of fluctuations around the hydrodynamic limit.

In the following, denote by~$\psi^\bd _n$ the eigenfunction of~$-\Delta^\bd $ with eigenvalue~$\lambda^\bd _n\geq 0$. 
For each~$s\in\R$, let~$Q^\bd_s \eqdef (1-\Delta^\bd)^s$ be defined via the spectral resolution of~$-\Delta^\bd $, 
and set~$\psi^{\bd,s}_n\eqdef (1+\lambda^\bd _n)^{-s/2}\psi^\bd _n$.
Finally, denote by~$H^\bd _s(\Omega)$ the Hilbert completion of the linear span of~$\ttseq{\psi^\bd_n}_n$ with respect to the pre-Hilbert norm~$\norm{\psi}_{H^\bd_s(\Omega)}\eqdef \tscalar{\psi}{Q^\bd_s\psi}^{1/2}_{L^2(\Omega)}$ and note that~$\ttseq{\psi^{\bd,s}_n}_n$ is a \textsc{cons} for~$H^\bd_s(\Omega)$ for each~$s$.
For~$s\geq 0$ we have that~$H^\bd_s(\Omega)=\dom{Q^\bd _{s/2}}$, and~$H^\bd_s(\Omega)'\cong H^\bd_{-s}(\Omega)$ as Hilbert spaces.

A proof of the next proposition is given in the Appendix (\S\ref{ss:CbLaplacians}).
\begin{proposition}[Test functions]\label{p:TestF}
	Let~$\varrho\in\R^+$. The following assertions hold:
	\begin{enumerate}[$(i)$]
		\item\label{i:p:TestF:1} the following spaces are countably Hilbert nuclear Fr\'echet spaces
		\begin{equation*}
			\mcS^\bd (\Omega)\eqdef \cap_s H^\bd _s(\Omega) \qquad \text{and} \qquad \mcS^\bd(\Omega)'=\cup_s H^\bd _s(\Omega)\semicolon
		\end{equation*}
		
		\item\label{i:p:TestF:2} $\ttonde{\mcS^\bd (\Omega), L^2(\Omega), \mcS^\bd(\Omega)'}$ is a countably Hilbert nuclear \purple{Gel'fand triple};
		
		\item\label{i:p:TestF:3} $P^\bd _t\colon \mcS^\bd (\Omega)\rightarrow \mcS^\bd (\Omega)$ for all~$t\geq 0$;
		
		\item\label{i:p:TestF:4} $\ttonde{-\Delta^\bd ,\mcS^\bd (\Omega)}$ is essentially self-adjoint on~$L^2(\Omega)$;
		
		\item\label{i:p:TestF:5} $\cC_c^\infty(\Omega)\subset \mcS^\bd (\Omega)\subset \cC^\infty(\Omega)\cap \cC^\bd$;
		
		\item\label{i:p:TestF:6} functions in~$\mcS^\bd (\Omega)$ satisfy the following boundary conditions:
		\begin{subequations}
			\begin{align}\label{eq:SBoundaryDir}
				\Delta^k f\equiv&\ 0 \quad \text{on~} \partial\Omega \quad \text{for all } k\geq 0 &&\text{if~} f\in \cS^\Dir(\Omega)\comma
				\\
				\label{eq:SBoundaryRob}
				\partial_\bn f+\varrho f\equiv&\ 0 \quad \text{on~} \partial\Omega &&\text{if~} f\in \cS^\varrho(\Omega)\comma
				\\
				\label{eq:SBoundaryNeu}
				\partial_\bn f\equiv&\ 0 \quad \text{on~} \partial\Omega &&\text{if~} f\in \cS^\Neu(\Omega) \fstop
			\end{align}
		\end{subequations}
		Here~\eqref{eq:SBoundaryDir} holds everywhere on~$\partial\Omega$, whereas~\eqref{eq:SBoundaryRob} and~\eqref{eq:SBoundaryNeu} ought to be interpreted in the~$\bdvol$-a.e.\ sense.
		
		\item\label{i:p:TestF:7}
		The space~$\cS^\bd(\Omega)$ is a core for the $\cC^\bd$-Laplacian~$\Delta^{\bd,c}$.
	\end{enumerate}
\end{proposition}

\subsubsection{Gaussian fields and limit theorems}\label{sec:FCLT-flu}
Recall the definition of~$h^{\eps,\beta}$ from \S\ref{sec:harmonic-profiles}, \purple{and} let~$\eta^{\eps,\beta}_0$ be randomly distributed as $\nu^{\eps,\beta}_\stat$ and~$\eta^{\eps,\beta}_t$ be the stochastic path starting at~$\eta^{\eps,\beta}_0$.
We introduce the \emph{empirical fluctuation fields at stationarity}
\begin{equation}\label{eq:fluctuation-fields}
	t\in \R^+_0\longmapsto \cY_t^{\eps,\beta}\coloneqq \eps^{d/2} \sum_{x\in \Omega_\eps} \tonde{\eta^{\eps,\beta}_t(x)-h^{\eps,\beta}(x)}\delta_x \fstop
\end{equation}
By stationarity,  $\cY_t^{\eps,\beta} \overset{\text{\normalfont d}}= \cY_0^{\eps,\beta}$ for every $t\in \R^+_0$.  Our aim is to show that such fields \purple{weakly converge in the sense of finite-dimensional distributions} to a $\cS^\bd(\Omega)'$-valued Gaussian field $\cY^\bd_t$, which we now describe.

Recalling $h^\bd$ from \S\ref{sec:harmonic-profiles}, we start by  defining
\begin{equation}\label{eq:chi}
	\chi^\bd\eqdef h^\bd\ttonde{1+\sigma h^\bd}\in \cC(\overline\Omega)\fstop
\end{equation}
For all $f, g \in \R^{\Omega_\eps}$ and $x\in \Omega_\eps$, further let
\begin{equation}\label{eq:carre-def}
	\varGamma^{\eps,\beta}(f,g)(x)\eqdef \frac{1}{2}\sum_{\substack{y\in \Omega_\eps\\ y\sim x}} \nabla^\eps_{x,y}f\,	 \nabla^\eps_{x,y}g + \car_{\partial\Omega_\eps}(x)\, \eps^{\beta-2} \alpha_\eps(x)\, f(x)\, g(x)
\end{equation}
be the \emph{carr\'e du champ} associated to $A^{\eps,\beta}$, and set~$\varGamma^{\eps,\beta}(f)\eqdef \varGamma^{\eps,\beta}(f,f)\geq 0$.
We have that~$\eps^d\sum_{x\in\Omega_\eps}\varGamma^{\eps,\beta}(f,g)(x)=\mcE^{\eps,\beta}(f,g)$. Furthermore, for all $\varphi, \psi \in \R^{\Omega_\eps}$, Cauchy--Schwarz inequality yields
\begin{equation}\label{eq:cs-carre}
\tscalar{\varGamma^{\eps,\beta}(f,g)}{\varphi \psi}_{L^2(\Omega_\eps)}^2\le \tscalar{\varGamma^{\eps,\beta}(f)}{\varphi^2}_{L^2(\Omega_\eps)}\tscalar{\varGamma^{\eps,\beta}(g)}{\psi^2}_{L^2(\Omega_\eps)} \fstop 
\end{equation}

  The covariances of the Gaussian field $\cY^\bd$  arise as  suitable limits of these \emph{carr\'e du champ}. More specifically, for all $f, g \in \cS^\bd(\Omega)$, we show that there exist~$\ttseq{\tau_\eps}_\eps\subset [1,\infty)$, $\ttseq{f_\eps}_\eps$ and~$\ttseq{g_\eps}_\eps$ for which $\E^\bd\quadre{\scalar{\cY^\bd_0}{f}\scalar{\cY^\bd_0}{g}}$ is the limit of
\begin{equation}\label{eq:conv-carre}
\lim_{t\to \infty}\lim_{\eps\downarrow 0}\scalar{\int_0^{\tau_\eps t} 2 \varGamma^{\eps,\beta}(P^{\eps,\beta}_s f_\eps, P^{\eps,\beta}_s g_\eps)\, \dd s }{\Pi_\eps\chi^\bd}_{L^2(\Omega_\eps)}
\end{equation}
plus some non-vanishing boundary terms in the Robin regime.

If~$\beta>1$, then $\chi^\Neu$ is constant and the limit in \eqref{eq:conv-carre} is  $ \chi^\Neu \scalar{f}{g}_{L^2(\Omega_\eps)}$. 
If~$\beta\le 1$, the function~$\chi^\bd$ is not necessarily constant.
Letting $\varGamma^\bd$ denote the continuum \emph{carr\'e du champ} associated to $\Delta^\bd$,  the candidate limit in \eqref{eq:conv-carre} formally reads $\tscalar{\int_0^\infty 2\varGamma^\bd(P^{\bd,c}_s f,P^{\bd,c}_s g)\, \dd s}{\chi^\bd}_{L^2(\Omega)}$.
However, due to the  lack of smoothness of $P^\bd_s f$  and $\chi^\bd$ near the Lipschitz boundary, 	
we have to replace the $L^2(\Omega)$- scalar product with the pairing between $\Mb(\overline\Omega)$  and $\cC(\overline\Omega)$.
Indeed, regarding $\varGamma^{\eps,\beta}(f_\eps,g_\eps)$ as an element of~$\Mb(\overline\Omega)$, 	we establish convergence in \eqref{eq:conv-carre} when $\beta\le1$ through the following two steps:
\begin{itemize}
	\item the family of measures associated to~$\tseq{\int_0^{\tau_\eps t}2 \varGamma^{\eps,\beta}(P^{\eps,\beta}_sf_\eps,P^{\eps,\beta}_sg_\eps)\, \dd s}_{\eps,t}$ is tight in $\Mb(\overline\Omega)$;
	\item for every $\varphi \in \cS^\Neu(\Omega)$, 
	\begin{equation}\label{eq:3.27}
		\begin{aligned}
		&\lim_{t\to \infty}\lim_{\eps\downarrow 0}\scalar{\int_0^{\tau_\eps t} 2 \varGamma^{\eps,\beta}(P^{\eps,\beta}_s f_\eps, P^{\eps,\beta}_s g_\eps)\, \dd s }{\Pi_\eps\varphi}_{L^2(\Omega_\eps)}\\
		&\qquad= \int_0^\infty -\mcE^\Neu(P^\bd_sf\, P^\bd_s g, \varphi)+ \mcE^\bd(\varphi\, P^\bd_s f,P^\bd_s g)+\mcE^\bd(\varphi\, P^\bd_s g, P^\bd_s f)\, \dd s  \fstop
		\end{aligned}
	\end{equation}
Note that the integrand on the right-hand side above is well-defined by combining Lemma~\ref{l:Domination},  $P^{\bd,c}_s f \in \msD(\Delta^\bd)\subset \msD(\mcE^\bd)\subset \msD(\mcE^\Neu)=H^1(\Omega)$, and \cite[Prop.\ 9.4]{brezis_functional_2011}.
\end{itemize}

By density of~$\cS^\Neu(\Omega)$ in $\cC(\overline\Omega)$, the functional of~$\varphi$ defined by the right-hand side of~\eqref{eq:3.27} identifies a unique signed measure, here denoted by~$\gamma^\bd_{f,g}$: for all $\varphi\in \cC(\overline\Omega)$ and $\varphi_n\in \cS^\Neu(\Omega)$ such that $\norm{\varphi-\varphi_n}_{\cC(\overline\Omega)}\to 0$ as $n\to \infty$,
\begin{equation*}
	\gamma^\bd_{f,g}(\varphi)=\lim_{n\to \infty}\int_0^\infty -\mcE^\Neu(P^\bd_s f\, P^\bd_s g,\varphi_n)+\mcE^\bd(\varphi_n\, P^\bd_s f,P^\bd_s g)+\mcE^\bd(\varphi_n\, P^\bd_s g, P^\bd_s f)\, \dd s\fstop
\end{equation*}
	
Before stating the main result of this section, set, for every $f, g \in \cS^\bd(\Omega)$,	
\begin{equation}\label{eq:iota}
	\iota^\varrho_{f,g}(\varphi)\eqdef  \scalar{\tonde{\int_0^\infty (P^\varrho_s f)(P^\varrho_s g)\, \dd s}\varrho\,\sigma_{\partial\Omega}}{ \varphi}\comma\qquad \varphi \in \cC(\overline\Omega)\comma
\end{equation}
and
	\begin{equation}\label{eq:covariance}
	{\rm cov}_\bd(f,g)\eqdef\begin{cases}
		\chi^\Neu\scalar{f}{g}_{L^2(\Omega)} &\text{if}\ \bd=\Neu\\
		\gamma_{f,g}^{\varrho}(\chi^\varrho)+  \iota^\varrho_{f,g}\tonde{(\vartheta-h^\bd)(1+2\sigma h^\bd)} &\text{if}\ 	\bd=\varrho\\
		\gamma^\Dir_{f,g}(\chi^\Dir) &\text{if}\ \bd=\Dir
	\end{cases}\fstop
\end{equation}
\begin{theorem}[Stationary non-equilibrium fluctuations, \S\ref{sec:proof-flu}]\label{th:fluctuations-stat}
For all $T>0$, there exists a unique centered Gaussian process~$\cY^\bd$ in~$\cC([0,T];\cS^\bd(\Omega)')$ with covariances  $	\E^\bd\quadre{\tscalar{\cY^\bd_t}{f}\tscalar{\cY^\bd_s}{g}}$ given, for  $f, g \in \cS^\bd(\Omega)$ and $0\le s \le t$, by $\cov_\bd(P^\bd_{t-s}f,g)$ as in \eqref{eq:covariance}. \purple{Furthermore}, for $\cY^{\eps,\beta}$ being the $\mcD([0,T];\cS^\bd(\Omega)')$-valued random variable given in \eqref{eq:fluctuation-fields}, we have that
\begin{align*}
	\cY^{\eps,\beta} \xRightarrow[\ \eps \downarrow 0 \ ]{\ {\rm fdd} \ } \cY^\bd \comma
\end{align*}
where~$\xRightarrow[]{\ {\rm fdd} \ }$ denotes (weak) convergence of finite-dimensional distributions.
\end{theorem}

\section{Auxiliary results}\label{sec:auxiliary-RW-h}
In this section we present some auxiliary results on discrete heat kernels and corresponding semigroups in \S\ref{sec:aux-RW} and \S\ref{sec:equicontinuity}\purple{, all required for the proofs in \S\S\ref{sec:semigroup-conv-proofs},\ref{sec:harmonic_conv}}.

\subsection{Random walks with \texorpdfstring{$\beta=\infty$}{BetaInfty} and Feynman--Kac formula}\label{sec:aux-RW}

The definition of random walks $X^{\eps,\beta}$ in \S\ref{sss:RandomWalks} with~$\beta\in \R$ readily extends to the case $\beta=\infty$ (in that case, conventionally, we let~$\eps^\infty\eqdef0$ in \eqref{eq:generator_RW}). The resulting continuous-time Markov processes $X^{\eps,\infty}$ evolve within~$\Omega_\eps$ as the~$X^{\eps,\beta}$ do, with the jumps to~$\partial_e\Omega_\eps$ being suppressed.  	These random walks have been extensively studied in \cite{chen2017hydrodynamic, fan2016discrete}. More specifically, by proving in \cite[\S5.1]{chen2017hydrodynamic} a class of $\eps$-independent relative isoperimetric inequalities, the authors in \cite{chen2017hydrodynamic,fan2016discrete} derive several properties about these random walks, the most important ones for our work being listed below for reference. Here, $C, C' >0$ and $a, b \in (0,1)$ denote constants depending only on $\Omega\subset \R^d$.
 \begin{itemize}
 	\item \emph{Nash inequalities and ultracontractivity of the semigroups \cite[Thm.~5.8]{chen2017hydrodynamic}:} for all $f\in \R^{\Omega_\eps}$,
 	\begin{equation}\label{eq:ultracontractivity}
 		\norm{P^{\eps,\infty}_t f}_{L^\infty(\Omega_\eps)}\le C\ttonde{\ttonde{ t^{1/2} \vee \eps}^{-d}+ 1} \norm{f}_{L^1(\Omega_\eps)}\comma \qquad t >0 \semicolon
 	\end{equation}
 \item \emph{mixing  \cite[Prop.~5.9]{chen2017hydrodynamic}}:
 \begin{align}	\label{eq:mixing_time_estimate}
 	\limsup_{\eps \downarrow 0} 
 	\sup_{x \in \Omega_\eps}
 	\norm
 	{
 		\frac{p_t^{\eps,\infty}(x,\cdot)}{\eps^d}
 		-
 		\frac1{\mu_\eps(\Omega_\eps)}	
 	}_{L^\infty(\Omega_\eps)} \defeq \phi_t \xrightarrow{t \to \infty} 0
 	\purple{\semicolon}
 \end{align}
 \item 	\emph{heat kernel upper bound on the inner boundary \cite[Lem.\ 2.13]{chen2017hydrodynamic}:} 
 \begin{equation}\label{eq:HKbd}
 	\sup_{x\in\Omega_\eps} \eps^{-1} \sum_{y\in \partial\Omega_\eps} p^{\eps,\infty}_t(x,y)\le C\ttonde{t^{-1/2}\purple{\wedge \eps^{-1}}}\comma\qquad 	 t > 0\semicolon
 \end{equation}
\item \emph{exit-time estimate \cite[Cor.\ 4.2]{fan2016discrete}:}  for all $\gamma>0$,
\begin{equation}\label{eq:exit-time}
\sup_{x\in \Omega_\eps}	\mbfP_x^{\eps,\infty}\tonde{\sup_{s\in[0,t]}\abs{X^{\eps,\infty}_s-x}>\gamma}\le C \exp\tonde{-\frac{C'\gamma}{t^{1/2} \vee \eps}}\comma \qquad  t >0 \semicolon
\end{equation}

\item \emph{H\"older's equicontinuity of heat kernel \cite[Thm.~5.12]{chen2017hydrodynamic}:}	for all $x,y \in \Omega_\eps$,
\begin{equation}\label{eq:holder}
	\sup_{z\in \Omega_\eps}\eps^{-d}\abs{p^{\eps,\infty}_t(x,z)-p^{\eps,\infty}_t(y,z)}\le C \frac{\abs{x-y}^a}{t^{b/2}\ttonde{t^{d/2}\wedge 1}} \comma\qquad  t >0\fstop
\end{equation}
 \end{itemize}

It is well-known that the random walks $X^{\eps,\beta}$ with $\beta\in \R$  are related to their counterparts with $\beta=\infty$ through the following \emph{Feynman--Kac representation formula}:  for all~$f\in \R^{\Omega_\eps}$, $x \in \Omega_\eps$ and $t\ge0$,
\begin{equation}\label{eq:feynman-kac}
	P^{\eps,\beta}_tf(x)= \mbfE_x^{\eps,\infty}\quadre{f(X^{\eps,\infty}_t) \exp\tonde{-\eps^{\beta-1}\int_0^t V_\eps(X^{\eps,\infty}_s)\, \dd s}}\comma
\end{equation} 
where, recalling $\alpha_\eps$ introduced in \S\ref{sss:DiscrInnerBd}, we set
\begin{equation}\label{eq:alpha_eps}
	V_\eps(x)\coloneqq  \eps^{-1}\car_{\partial\Omega_\eps}\!(x)\, \alpha_\eps(x) \comma\qquad x \in \Omega_\eps\fstop
\end{equation}
As a consequence of \eqref{eq:feynman-kac}, for all $x, y \in \Omega_\eps$ and $t\ge 0$, we have 
\begin{equation}\label{eq:domination_HK}
	p^{\eps,\beta}_t(x,y)\leq p^{\eps,\beta'}_t(x,y)\le p^{\eps,\infty}_t(x,y)\comma\qquad \beta\leq \beta'\comma
\end{equation}
as well as	
\begin{equation}\label{eq:feynman-kac_dt}
	\frac{\dd}{\dd t}P^{\eps,\beta}_t \car_{\Omega_\eps}(x)=\mbfE^{\eps,\infty}_x\quadre{-\eps^{\beta-1}V_\eps(X^{\eps,\infty}_t)\exp\tonde{-\eps^{\beta-1}\int_0^t V_\eps(X^{\eps,\infty}_s)\, \dd s}}	\fstop
\end{equation}
Finally, by combining the second inequality in \eqref{eq:bd_unif_ellipticity} and \eqref{eq:HKbd}, 
\begin{equation}\label{eq:HKbdV}
\sup_{x\in \Omega_\eps}\mbfE_x^{\eps,\infty}\quadre{V_\eps(X^{\eps,\infty}_t)}=	\sup_{x\in\Omega_\eps}\sum_{y\in \Omega_\eps} p^{\eps,\infty}_t(x,y)\, V_\eps(y)\le C\ttonde{t^{-1/2}\purple{\wedge\eps^{-1}}}\comma\qquad t >0\fstop
\end{equation}

\subsection{Equicontinuity of heat semigroups}\label{sec:equicontinuity}
By \eqref{eq:exit-time} and \eqref{eq:holder}, we get, for all $f \in \cC(\overline\Omega)$, 
\begin{equation}\label{eq:equicontinuity_neumann_eps}
\lim_{\delta\downarrow 0} \sup_{\eps\in (0,1)}\sup_{t\geq 0}\sup_{\substack{x,y \in \Omega_\eps\\\abs{x-y}<\delta}} \abs{P^{\eps,\infty}_t \Pi_\eps f(x)-P^{\eps,\infty}_t \Pi_\eps f(y)}=0 \fstop
\end{equation}
An analogous statement holds true for the strongly continuous semigroup $P^{\Neu,c}_t$ on $\cC(\overline\Omega)$.

As a refinement, we show the spatial equicontinuity of discrete semigroups with modulus of continuity  \emph{independent of $\beta\in \R$}.

\begin{proposition}[Equicontinuity of discrete semigroups, \S\ref{sec:proof-equicontinuity}]\label{pr:equi_semi_disc}
Fix $f \in \cC(\overline\Omega)$. Then, 
\begin{equation}\label{eq:pr:equi_semi_disc-robin}
	\lim_{\delta\downarrow 0}\sup_{\eps\in (0,1)}\sup_{\beta\ge 1}\sup_{t\geq 0}  \sup_{\substack{x,y \in \Omega_\eps\\\abs{x-y}<\delta}} \abs{P^{\eps,\beta}_t\Pi_\eps f(x)-P^{\eps,\beta}_t \Pi_\eps f(y)}=0\comma
\end{equation}
\begin{equation}
\label{eq:pr:equi_semi_disc}
\lim_{\delta\downarrow 0}
{
\sup_{\eps\in (0,1)}
	\sup_{\beta\in \R}
	\sup_{t \geq t_0}
}
	 \sup_{\substack{x,y \in \Omega_\eps\\\abs{x-y}<\delta}} \abs{P^{\eps,\beta}_t\Pi_\eps f(x)-P^{\eps,\beta}_t \Pi_\eps f(y)}=0\comma\qquad t_0>0\fstop
\end{equation}
\end{proposition}

The proof of the proposition above \purple{is postponed to the appendix (\S\ref{sec:proof-equicontinuity}), and} immediately adapts to the continuum setting, yielding the following.
\begin{proposition}[Equicontinuity of continuum semigroups]\label{pr:equi_semi_cont}
	Fix $f \in \cC(\overline\Omega)$. Then, 
	\begin{equation}\label{eq:equi_semi_cont}
		\lim_{\delta\downarrow 0}\sup_{\varrho \ge 0} \sup_{t\geq t_0}\sup_{\substack{x,y \in \Omega\\ \abs{x-y}<\delta}} \abs{P^\varrho_t f(x)-P^\varrho_t f(y)}=0	\comma\qquad t_0>0 \fstop
	\end{equation}
An analogue of \eqref{eq:equi_semi_cont} holds with $t_0=0$ if 
replacing $\sup_{\varrho\ge 0}$ with $\sup_{\varrho\geq \varrho_*}$\, for some $\varrho_*>0$.
\end{proposition}

\section{Proof of Theorem~\ref{t:MainSemigroups}}\label{sec:semigroup-conv-proofs}
In this section we prove Theorem~\ref{t:MainSemigroups}, and divide its proof into three parts, each one of them addressing a different regime of $\beta\in \R$ ($\beta>1$ in \S\ref{sec:semigroup-conv-neumann}, $\beta=1$ in \S\ref{sec:semigroup-conv-robin} and $\beta <1$ in \S\ref{sec:semigroup-conv-dirichlet}). 
In order to lighten the notation, we omit  the evaluation  $\Pi_\eps:\cC(\overline\Omega)\to L^\infty(\Omega_\eps)$; moreover, with a slight abuse of notation,  here and until the end of \S\ref{sec:proof-flu}, we do not explicitly distinguish between  $L^2$- and $\cC^\bd$-semigroups, letting  $P^\bd_t$ indicate either of them; an analogous convention holds  for the corresponding generators and resolvents.

\subsection{Proof of Theorem~\ref{t:MainSemigroups}: Neumann regime}\label{sec:semigroup-conv-neumann}

Throughout this section, fix~$\beta>1$.
By the triangle inequality, for all $f \in \cC(\overline\Omega)$ and $T>0$,
\begin{align*}
\sup_{t\in [0,T]} \tnorm{P^{\eps,\beta}_t  f- P^\Neu_t f}_{L^\infty(\Omega_\eps)} &\leq \sup_{t\in [0,T]} \tnorm{P^{\eps,\beta}_t  f-P^{\eps,\infty}_t  f}_{L^\infty(\Omega_\eps)}\\
&\qquad +  \sup_{t\in [0,T]} \norm{P^{\eps,\infty}_t f - P^\Neu_t f}_{L^\infty(\Omega_\eps)} \fstop
\end{align*}
Theorem~\ref{t:MainSemigroups} thus follows as soon as we show that both terms in the right-hand side above vanish as~$\eps\to 0$.
These claims are immediate consequences of Lemmas~\ref{l:UnifCont} and~\ref{lemma:neumann_conv} below.

\begin{lemma}\label{l:UnifCont}
For every~$f\in\cC(\overline\Omega)$ and $T>0$, there exists $C=C(\Omega,f,T)>0$ such that
\begin{align*}
\sup_{t\in [0,T]} \norm{P^{\eps,\beta}_t f-P^{\eps,\infty}_t f}_{L^\infty(\Omega_\eps)}\le	 C\,\eps^{\beta-1} \fstop
\end{align*}
\begin{proof}
By~\eqref{eq:feynman-kac} and the inequality $1-e^{-a}\le a$ for $a\ge 0$, for all $x\in \Omega_\eps$ and $t\in [0,T]$,
		\begin{align*}
			\abs{P^{\eps,\beta}_t f(x)-P^{\eps,\infty}_t f(x)}&\le \norm{f}_{\cC(\overline\Omega)} \mbfE_x^{\eps,\infty}\quadre{1-\exp\tonde{\eps^{\beta-1}\int_0^t V_\eps(X^{\eps,\infty}_s)\, \dd s}}\\
			&\qquad \le \norm{f}_{\cC(\overline\Omega)}\mbfE_x^{\eps,\infty}\quadre{\eps^{\beta-1}\int_0^T V_\eps(X^{\eps,\infty}_s)\, \dd s}
			\fstop
\end{align*}
The inequality \eqref{eq:HKbdV} yields the desired result.
	\end{proof}
\end{lemma}

\begin{lemma}\label{lemma:neumann_conv}
For every~$f\in\cC(\overline\Omega)$ and every~$T>0$,
	\begin{equation*}
		\lim_{\eps\downarrow 0}\sup_{t\in [0,T]} \norm{P^{\eps,\infty}_t  f-  P^\Neu_t f}_{L^\infty(\Omega_\eps)}=0 \fstop
	\end{equation*}
	\begin{proof} 
	Thanks to Theorem~\ref{t:Equivalence} (which holds also with $\beta=\infty$ and $\bd=\Neu$), the desired claim is equivalent to show assertion~\ref{i:t:Equivalence:4}, that is 
	$
	\mcE^{\eps,+\infty}\to \mcE^\Neu
	$
	compactly (in the sense of Definition~\ref{def:KS_convergence}). This is a straightforward consequence of \cite[Thm.~6.1 and Rmk.~6.5]{AliCic04} and Proposition~\ref{p:ACtoKS}.
	\end{proof}
\end{lemma}

\begin{remark}[Dyadic  lattice approximations and local CLT]\label{rmk:dyadic_lclt}
The claim in Lemma \ref{lemma:neumann_conv} has already been proven  in \cite{burdzy_discrete_2008}  for \emph{dyadic} lattice approximations (i.e., when $\eps=2^{-k}$, $k \in \N$). Reasoning as in the proof of \cite[Thm.\ 5.13]{chen2017hydrodynamic}, Lemma \ref{lemma:neumann_conv} and the equicontinuity of the heat kernels in \eqref{eq:holder} ensure the validity of the \emph{local CLT} \cite[Thm.\ 5.13]{chen2017hydrodynamic} for the random walks~$X^{\eps,\infty}$ for \emph{every} lattice approximation.
\end{remark}

\subsection{Proof of Theorem~\ref{t:MainSemigroups}: Robin regime}\label{sec:semigroup-conv-robin}
Fix $f \in \mcC(\overline\Omega)$.
It is well known (see, e.g.,~\cite[Eq.~(3.2)]{fan2016discrete}) that~$P^\varrho_t f$ admits the following stochastic representation in terms of the reflected Brownian motion~$B^\Neu_t$ and its boundary local time $L^\Neu_t$:
\begin{equation}\label{eq:semigr_Robin_localtime}
P^\varrho_t f(x)= \mbfE_x^\Neu\quadre{f(B_t^\Neu)\,  e^{-\varrho L_t^\Neu}}\comma\qquad x \in \Omega\comma  t \geq 0\fstop
\end{equation}
Furthermore, by \eqref{eq:feynman-kac},
\begin{equation}
P^{\eps,\beta=1}_t f(x)= \mbfE_x^{\eps,\infty}\quadre{f(X_t^{\eps,\infty}) \exp\tonde{-\int_0^t V_\eps(X_s^{\eps,\infty})\, \dd s}}\comma\qquad x \in \Omega_\eps\comma t \geq 0\fstop
\end{equation}

By Remark~\ref{rmk:dyadic_lclt} and the assumption in \eqref{eq:bd_weak_conv}, the proof of \cite[Thm.~3.1]{fan2016discrete} carries over to our setting: for all $x \in \overline\Omega$ and $\ttseq{x_\eps}_\eps$ such that $x_\eps \in \Omega_\eps$ and $x_\eps \to x$, we have that 
\begin{align}	\label{eq:joint_conv_RW_LT}
	\seq{\tonde{X^{\eps,\infty}_t, \int_0^t V_\eps(X^{\eps,\infty}_s)\, \dd s}}_t \xRightarrow[\ \eps \downarrow 0 \ ]{\ \mrmd \ } \tseq{(B^\Neu_t, L^\Neu_t)}_t 
\end{align}
in $\mcD(\R^+_0;\overline\Omega)\times \mcC(\R^+_0;\R^+_0)$; in particular, 
\begin{equation}
\lim_{\eps\downarrow 0} P^{\eps,\beta=1}_t f(x_\eps)= P^{\varrho=1}_t f(x)
\end{equation} 
which, by Proposition~\ref{pr:equi_semi_disc}, holds uniformly in space.
Then, the uniformity over bounded intervals of time is \ref{i:t:Equivalence:2}$\implies$\ref{i:t:Equivalence:1} of Theorem \ref{t:Equivalence}. This concludes the proof.

\begin{corollary}[Spectral bound,~$\beta\leq 1$]\label{c:spectral_bound_DR}
\purple{With the definitions adopted in Theorem~\ref{t:Equivalence}\ref{i:t:Equivalence:8} below}, 
	there exists $ \underline\lambda_0>0$ satisfying
\begin{align*}
	\lambda_0^\Dir\ge \lambda_0^\varrho\ge \underline\lambda_0\comma\qquad 
\lambda_0^{\eps,\beta}\geq \lambda_0^{\eps,\beta=1}\geq \underline\lambda_0 \fstop
\end{align*}
\begin{proof}
By, e.g., \cite{daners_faber2006}, we get $\lambda^\varrho_0>0$; the inequality $\lambda_0^\Dir\ge \lambda_0^\varrho$ follows by the monotonicity of Dirichlet forms (Lem.\ \ref{l:Domination}). Analogously (cf.\ \eqref{eq:domination_HK}),  $\lambda_0^{\eps,\beta}\ge\lambda_0^{\eps,\beta=1}$. As for the last  inequality, we combine 
 Theorem~\ref{t:MainSemigroups} for~$\beta=1$ with Theorem~\ref{t:Equivalence}\ref{i:t:Equivalence:8}, implying that $\lambda^{\eps,\beta}_0\to\lambda^\varrho_0$.
\end{proof}
\end{corollary}

\subsection{Proof of Theorem~\ref{t:MainSemigroups}: Dirichlet regime}\label{sec:semigroup-conv-dirichlet}
Throughout this section, fix~$\beta<1$.
In this case, we divide the proof of Theorem~\ref{t:MainSemigroups} into two parts.
As a first step (\S\ref{sec:conv_dir_smooth}), we assume that~$\Omega$ is a bounded \emph{smooth} domain and prove the graph-convergence of the generators, which implies Theorem~\ref{t:Equivalence} (Theorem~\ref{t:Equivalence}\ref{i:t:Equivalence:3} $\implies$ \ref{i:t:Equivalence:1}).
This approach via generators' convergence heavily relies on the existence of a core for $\Delta^\Dir=\Delta^{\Dir,c}$ consisting of smooth functions \textit{with continuous derivatives up to the boundary}.
As a second step (\S\ref{sec:conv_dir_lipschitz}), we drop the assumption of smoothness of the domain $\Omega$.
We prove the semigroups' convergence on the bounded Lipschitz domain $\Omega$ 
approximating it from the inside with smooth domains.

\subsubsection{Case of a bounded smooth domain \texorpdfstring{$\Omega$}{Omega}}\label{sec:conv_dir_smooth} In this section, we assume $\Omega \subset \R^d$ to be a bounded $\cC^\infty$-domain. In this case, we recall that there exists  a  core of smooth functions up to the boundary for the Dirichlet generator $(\Delta^\Dir,\cC_0(\Omega))$ associated to the $\cC_0$-semigroup~$P_t^\Dir$ on $\cC_0(\Omega)$. As a concrete instance of such a core, we choose $\cS^\Dir(\Omega)$ constructed in Proposition~\ref{p:TestF}\ref{i:p:TestF:1} and note that, by the smoothness of $\Omega$ and the classical Sobolev embeddings,  $\cS^\Dir(\Omega)\subset \cC_0(\Omega)\cap \cC^\infty(\overline\Omega)$ (see, e.g.,~\cite[Thm.~2.20]{AmbCarMas18} or~\cite[Thm.~2.5.1.1]{Gri85}).

Recall the definition of the resolvent~$R^{\eps,\beta}_\zeta\coloneqq\tonde{\zeta -A^{\eps,\beta}}^{-1}=\int_0^\infty e^{-\zeta t} P_t^{\eps,\beta}  \dd t$, for $\zeta \in \R^+$. In the next lemmas, for every $f \in \cS^\Dir(\Omega)$, setting $f_\eps\eqdef \zeta_\eps R^{\eps,\beta}_{\zeta_\eps} f$, we show that
$
	f_\eps \to f
$ and 
$
	A^{\eps,\beta}f_\eps \to	 \Delta^\Dir f
	$,
for a suitable choice \eqref{eq:zeps2} of $\zeta_\eps \to \infty$.
\begin{lemma}\label{l:ConvergenceGeneratorAll}
Fix~$f\in\cC_0(\Omega)$. Then~$\zeta_\eps R^{\eps,\beta}_{\zeta_\eps} f\to f$ for every family~$\zeta_\eps \to \infty$. 
	\begin{proof}
		By the triangle inequality, 
		\begin{align}\nonumber \label{eq:IJ}
			\tnorm{\zeta_\eps R^{\eps,\beta}_{\zeta_\eps} f- f}_{L^\infty(\Omega_\eps)}
			\le&\ \sup_{x\in\Omega_\eps} \abs{\zeta_\eps \int_0^\infty e^{-\zeta_\eps t} \sum_{y\in\Omega_\eps} p^{\eps,\beta}_t(x,y) \ttonde{f(x)-f(y)} \diff t}
			\\
			 &\ + \sup_{x\in\Omega_\eps}\abs{\zeta_\eps\int_0^\infty e^{-\zeta_\eps t} \ttonde{1-P^{\eps,\beta}_t \car_{\Omega_\eps}(x)} f(x)\, \dd t} \fstop
			\end{align}
		As for the first term on the right-hand side above for all $\delta >0$, we estimate it from above by
		\begin{align*}
		 &\sup_{\substack{x,y\in\Omega\\ \abs{x-y}<\delta}} \abs{f(x)-f(y)} +2 \norm{f}_{\cC_0(\Omega)} \sup_{x\in \Omega_\eps} \abs{\zeta_\eps\int_0^\infty e^{-\zeta_\eps t} \sum_{y\in\Omega_\eps: \abs{y-x}\ge\delta} p^{\eps,\beta}_t(x,y)\,\diff t} 
	\\
		&\leq 
 		\sup_{\substack{x,y\in\Omega\\ \abs{x-y}<\delta}} \abs{f(x)-f(y)}
            +
           C \norm{f}_{\cC_0(\Omega)}
           \tonde{
             \exp \tonde{ - \frac{C' \delta}{t_\eps^{1/2} \vee \eps }} + e^{-\zeta_\eps t_\eps} 
            }
            \comma
		\end{align*}
	 	where we employed the exit-time estimate \eqref{eq:exit-time},  for a given family~$\ttseq{t_\eps}_\eps \subset \R^+$ such that~$\lim_{\eps\downarrow 0} t_\eps=0$ and~$\lim_{\eps\downarrow 0} \zeta_\eps t_\eps =+\infty$.
		Since $f$ is uniformly continuous in $\overline\Omega$, letting first~$\eps \to 0$ and then $\delta \to 0$ shows that the first term in the right-hand side of \eqref{eq:IJ} vanishes. 

As for the second one, for~$\rho > 0$ and~$\Omega^\rho\eqdef \set{x\in \Omega: \dist(x,\partial\Omega)>\rho}$, arguing as before, we estimate it from above by
\begin{align*}		 		
	&\sup_{x \in \Omega\setminus \Omega^{\rho}}\abs{f(x)} 
	+ \norm{f}_{\cC_0(\Omega)}\sup_{x\in \Omega_\eps\cap \Omega^{\rho}}\abs{\zeta_\eps \int_0^\infty e^{-\zeta_\eps t} \ttonde{1-P^{\eps,\beta}_t \car_{\Omega_\eps}(x)} \dd t} 
\\
		&\leq 
	\sup_{x \in \Omega\setminus \Omega^{\rho}}\abs{f(x)}
	+ 
		 C \norm{f}_{\cC(\overline \Omega)}
		\tonde{
			\exp \tonde{ - \frac{C' \rho}{t_\eps^{1/2} \vee \eps }} + e^{-\zeta_\eps t_\eps}
		} \fstop 
\end{align*}
Finally, since $f \in \cC_0(\Omega)$,  
letting first $\eps \to 0$ and then $\rho \to 0$, we conclude the proof.	\end{proof}
\end{lemma}

\begin{lemma}\label{l:ConvergenceGeneratorDir}
	Let~$\ttseq{\zeta_\eps}_\eps\subset \R^+$ be such that $\zeta_\eps \to \infty$ and 
\begin{equation}\label{eq:zeps2}
		\zeta_\eps  \in o\ttonde{\eps^{-(1\wedge   (1-\beta))}} \fstop
	\end{equation} 
	For~$f\in\cS^\Dir(\Omega)$,  set~$f_\eps\eqdef \zeta_\eps R^{\eps,\beta}_{\zeta_\eps}	 f$.
	Then,~$A^{\eps,\beta} f_\eps\to \Delta^\Dir f$.
	\begin{proof}
		By the triangle inequality, we have that
		\begin{align*}
			&\tnorm{A^{\eps,\beta} f_\eps- \Delta^\Dir f}_{L^\infty(\Omega_\eps)}= \tnorm{\zeta_\eps R^{\eps,\beta}_{\zeta_\eps} A^{\eps,\beta} f -  \Delta^\Dir f}_{L^\infty(\Omega_\eps)}
			\\
			&\qquad\le \tnorm{\zeta_\eps R^{\eps,\beta}_{\zeta_\eps} A^{\eps,\beta} f -\zeta_\eps R^{\eps,\beta}_{\zeta_\eps}  \Delta^\Dir f}_{L^\infty(\Omega_\eps)}
			+ \tnorm{\zeta_\eps R^{\eps,\beta}_{\zeta_\eps}  \Delta^\Dir f -  \Delta^\Dir f}_{L^\infty(\Omega_\eps)} \fstop
		\end{align*}
		Since~$\Delta^\Dir\cS^\Dir\subset \cS^\Dir\subset \cC_0(\Omega)$, the second term vanishes as~$\eps\to 0$ by Lemma~\ref{l:ConvergenceGeneratorAll} with~$ \Delta^\Dir f$ in place of~$f$.
		Thus, it suffices to show that
		\begin{equation*}
			\lim_{\eps\downarrow 0} \tnorm{\zeta_\eps R^{\eps,\beta}_{\zeta_\eps} A^{\eps,\beta} f -\zeta_\eps R^{\eps,\beta}_{\zeta_\eps} \Delta^\Dir f}_{L^\infty(\Omega_\eps)} =0 \fstop
		\end{equation*}
		To this end, we split the term into bulk and boundary contributions:
		\begin{align*}
			&\tnorm{\zeta_\eps R^{\eps,\beta}_{\zeta_\eps} A^{\eps,\beta}f -\zeta_\eps R^{\eps,\beta}_{\zeta_\eps} \Delta^\Dir f}_{L^\infty(\Omega_\eps)}
			\leq \sup_{y\in\Omega_\eps \setminus \partial\Omega_\eps} \tabs{A^{\eps,\beta} f(y)- \Delta^\Dir f(y)} 
\\
&\qquad \qquad  + \sup_{x\in\Omega_\eps} \abs{\zeta_\eps \int_0^\infty e^{-\zeta_\eps t} \sum_{y\in\partial\Omega_\eps} p^{\eps,\beta}_t(x,y) \ttonde{A^{\eps,\beta} f(y)- \Delta^\Dir f(y)}\, \diff t} 
		\comma\end{align*}
		where we used that~$\sum_{y\in \Omega_\eps\setminus\partial\Omega_\eps} p^{\eps,\beta}_t(x,y)\le 1$.
		Note that, on $\Omega_\eps\setminus \partial\Omega_\eps$, $A^{\eps,\beta}$ coincides with the $\eps$-discrete Laplacian $\Delta_\eps$ on $(\eps \Z)^d$.
		Furthermore, $\Delta^\Dir f$ coincides with the usual Laplacian of $f$ for all $f\in \cS^\Dir$. Hence, the first term on the right-hand side above equals
	$
	\norm{ \Delta_\eps f- \Delta f }_{L^\infty(\Omega_\eps \setminus \partial \Omega_\eps)}
	$,
		which vanishes in~$\eps$ since $f\in \cS^\Dir(\Omega)\subset \mcC^3(\overline \Omega)$.
		We claim that the second term also vanishes.
		Since~$\Delta^\Dir f\in\cS^\Dir\subset \cC_0(\Omega)$ implies 
	$\norm{ \Delta^\Dir f}_{L^\infty(\partial\Omega_\eps)}\rightarrow0$,	it suffices to show that
		\begin{align}\label{eq:l:ConvergenceGeneratorDir:1}
			\lim_{\eps\downarrow 0} \sup_{x\in\Omega_\eps} \abs{\zeta_\eps \int_0^\infty e^{-\zeta_\eps t}\sum_{y\in\partial\Omega_\eps} p^{\eps,\beta}_t(x,y) A^{\eps,\beta} f(y) \, \diff t}=0 \fstop
		\end{align}		
		First observe that, by $f\in   \cS^\Dir \subset \cC_0(\Omega) \cap \cC^1(\overline\Omega)$, we have $\sup_{y\in \partial\Omega_\eps}\abs{f(y)}\le C(f)\, \eps$.
		Combining this with the assumption in \eqref{eq:bd_unif_ellipticity}, we estimate uniformly in $y \in \partial \Omega_\eps$
		\begin{align*}
			\abs{A^{\eps,\beta} f(y)}= \abs{\eps^{-2}\sum_{\substack{z\in\Omega_\eps \\ z\sim y}} \ttonde{f(z)-f(y)} - \eps^{\beta-2} \sum_{\substack{z\in \partial_e\Omega_\eps \\ z\sim y}}\ayz f(y)}
			\leq C\,	 \ttonde{\eps^{-1} \vee \eps^{\beta-1}}\comma
		\end{align*}
	for some $C=C(f,d,\Omega)>0$  independent of $\eps >0$ . 
		As a consequence, using the assumption in \eqref{eq:zeps2}, we conclude the proof of~\eqref{eq:l:ConvergenceGeneratorDir:1} showing that
		\begin{equation}	\label{eq:bd_kernel_scaling}
			\tonde{\eps^{-1}\vee \eps^{\beta-1}} \sup_{x\in\Omega_\eps}\int_0^\infty \sum_{y\in \partial\Omega_\eps} p^{\eps,\beta}_t(x,y)\, \dd t\lesssim \tonde{\eps^{1-\beta}\vee \eps}\fstop
		\end{equation}
		By~\eqref{eq:feynman-kac} and~\eqref{eq:bd_unif_ellipticity}, recalling \eqref{eq:alpha_eps}, for $x \in \Omega_\eps$, we obtain that
		\begin{align*}
			(\eps^{-1} \!\vee\! \eps^{\beta-1})  \sum_{y\in\partial\Omega_\eps} p^{\eps,\beta}_t(x,y) 
			\le C (1 \! \vee \! \eps^\beta)	 \mbfE^{\eps,\infty}_x \! \quadre{V_\eps(X^{\eps,\infty}_{t}) \exp\tonde{\! -\eps^{\beta-1}\int_0^t V_\eps(X^{\eps,\infty}_s)\, \diff s}} \fstop  
		\end{align*}
	Recalling~\eqref{eq:feynman-kac_dt}, by integrating in time,  we obtain \eqref{eq:bd_kernel_scaling} and conclude. 
	\end{proof}
\end{lemma}
\begin{remark}\label{rem:conv_dir_smooth}
	The proofs of both lemmas above do not rely on \eqref{eq:bd_weak_conv}, but only on \eqref{eq:bd_unif_ellipticity}, $\beta <1$, the smoothness of $f \in \cC^3(\overline\Omega)$ and the fact that  $f$ and $\Delta^\Dir f \in \cC_0(\Omega)$.
\end{remark}

\subsubsection{Case of bounded Lipschitz domain \texorpdfstring{$\Omega$}{Omega}}\label{sec:conv_dir_lipschitz}
For a bounded Lipschitz domain $\Omega$, \cite[Thm.~1.12]{verchota_layer1984} (see also \cite[Thm.~8.1.5]{shen_periodic_2018}), there exist bounded smooth domains $\ttseq{U_n}_n$ such that~$\overline U_n\subset \Omega$ and $U_n \nearrow \Omega$, i.e., $\dist_{\rm H}(U_n,\partial\Omega)\to 0$ as $n\to \infty$. 
Moreover, due to the compactness of $\partial\Omega$ and $\partial U_n$, $\inf_{x\in \partial\Omega}\inf_{y\in U_n}\abs{x-y}=\delta_n>0$.

For each $n \in \N$, we introduce $P^{\Dir,n}_t$ as the Dirichlet semigroup on $\cC_0(U_n)$. In other words, $P^{\Dir,n}_t$ is the unique $\cC_0$-semigroup on $\cC_0(U_n)$ corresponding to the standard Brownian motion stopped upon exiting $U_n$; in particular, for all $f\in \cC_0(\Omega)$ such that $f\vert_{U_n}\in \cC_0(U_n)$,
\begin{equation}
	P^{\Dir,n}_tf(x)\coloneqq \mbfE^\Dir_x\quadre{f(X^\Dir_{t\wedge \tau_n})}\comma \qquad x \in \overline\Omega\comma t \ge 0\fstop
\end{equation}
Here and all throughout, $\tau_n$ is defined as the first exit time from $U_n$, viz.
\begin{equation*}
\mcD([0,\infty);\R^d) \ni \omega\longmapsto\tau_n[\omega]\coloneqq \inf\set{t\ge 0: \omega_t \notin U_n}\fstop
\end{equation*}
Similarly,  $P^{\eps,n}_t$ is the sub-Markov semigroup corresponding to the random walk $X^{\eps,\beta}_t$ on~$\overline\Omega_\eps$ stopped upon exiting $U_n$.  (Note that this definition is independent of the value of~$\beta \in \R$ provided that, for fixed $n \in \N$, $\eps\in (0,1)$ is sufficiently small.)

Now we turn  to the proof of Theorem~\ref{t:MainSemigroups} for $\beta<1$.
We observe that, since the semigroups~$P^{\Dir}_t$ and~$P^{\eps,\beta}_t$ are contraction semigroups  in $\cC_0(\Omega)$ and $L^\infty(\Omega_\eps)$, respectively,  and since $\norm{\emparg}_{L^\infty(\Omega_\eps)}\le \norm{\emparg}_{\cC_0(\Omega)}$, it suffices to prove Theorem~\ref{t:MainSemigroups} for $f$ in a dense subspace of $\cC_0(\Omega)$, e.g., for  $f \in \cC^\infty_c(\Omega)$.	For any such $f$, since $U_n\nearrow \Omega$, we have $f\vert_{\overline U_n} \in \cC^\infty_c(U_n) \subset \cC_0(U_n)$ for all $n\in \N$ large enough. For all such $n \in \N$ and for all $T > 0$, by the triangle inequality, 
\begin{align}
	\label{eq:split_smoothdom}
\begin{aligned}
\sup_{t\in [0,T]}&	\tnorm{P^{\eps,\beta}_t f- P^\Dir_t  f}_{L^\infty(\Omega_\eps)} \le\ \sup_{t\in [0,T]}\tnorm{P^{\eps,\beta}_t f- P^{\eps,n}_t f}_{L^\infty(\Omega_\eps)}\\
	+&\ \sup_{t\in [0,T]}\tnorm{P^{\eps,n}_t f- P^{\Dir,n}_t  f}_{L^\infty(\Omega_\eps)}
	+\ \sup_{t\in [0,T]}\tnorm{P^{\Dir,n}_t f-P^\Dir_t f}_{\cC_0(\Omega)}	\fstop
\end{aligned}
\end{align}

Note that
$
	\limsup_{\eps\downarrow 0} \sup_{t \in [0,T]} \norm{P^{\eps,n}_t f- P^{\Dir,n}_t f}_{L^\infty(\Omega_\eps)} =0
$ 
for every $ n \in \N$, as follows by the very same argument in Section \ref{sec:conv_dir_smooth} (see Remark \ref{rem:conv_dir_smooth}), since $U_n$ is a bounded smooth domain.	
As for the third term on the right-hand side of \eqref{eq:split_smoothdom}, we have that 
\begin{equation}\label{eq:dir_n1}
	\norm{P^{\Dir,n}_t f-P^\Dir_t f}_{\cC_0(\Omega)}=\sup_{x\in \Omega} \abs{\mbfE^\Dir_x\quadre{f(X^\Dir_{t\wedge \tau_n})-f(X^\Dir_t)}} = \sup_{x\in \Omega}\abs{\mbfE^\Dir_x\quadre{\car_{\tau_n < t}\, f(X^\Dir_t)}} ,
\end{equation}
for  $t \in [0,T]$. By the strong Markov property, 
\begin{equation}\label{eq:dir_n2}
\sup_{t \in [0,T]}	\sup_{x\in \Omega}\abs{\mbfE^\Dir_x\quadre{\car_{\tau_n< t}\, f(X^\Dir_t)}}
 \le \sup_{s \in [0,T]}\sup_{y \in \overline\Omega \setminus U_n} P^\Dir_s|f|(y)\fstop
\end{equation}
 Since $P^\Dir_s |f|\in \cC_0(\Omega)$ and  $(s,y) \mapsto P^\Dir_s |f|(y) \in   \cC \tonde{ [0,T]\times\overline\Omega}$, by \eqref{eq:dir_n1} and \eqref{eq:dir_n2},
\begin{equation*}
\lim_{n\to \infty}	\sup_{t \in [0,T]}\norm{P^{\Dir,n}_t f- P^\Dir_t f}_{\cC_0(\Omega)}=0\fstop
\end{equation*}

It remains to show that
\begin{equation}\label{eq:beta_n}
	\lim_{n\to \infty}\limsup_{\eps\downarrow 0} \sup_{t\in [0,T]}\tnorm{P^{\eps,\beta}_t f- P^{\eps,n}_t f}_{L^\infty(\Omega_\eps)}=0\fstop 
\end{equation}
Arguing as in \eqref{eq:dir_n1}--\eqref{eq:dir_n2}, and defining $(\Omega\setminus U_n)_\eps\eqdef\Omega_\eps \setminus (\Omega_\eps \cap U_n)$,
\begin{equation}
	\sup_{t\in [0,T]} \tnorm{P^{\eps,\beta}_t f- P^{\eps,n}_t f}_{L^\infty(\Omega_\eps)} \le 	\sup_{s \in [0,T]}\sup_{y \in (\Omega\setminus U_n)_\eps}  P^{\eps,\beta}_s|f|(y)\fstop
\end{equation}
Note that there exists $\delta >0$ such that,  for all $n \in \N$ large enough, $\dist_{\rm H}(\overline\Omega\setminus U_n, {\rm supp}(f))\ge \delta$. Hence, for $t \in (0,T)$ and $y \in (\Omega\setminus U_n)_\eps $, 	we have that
\begin{align*}\nonumber
	&\sup_{s \in [0,T]}   P^{\eps,\beta}_s|f|(y)
	\le\ \sup_{s \in [0,t]}  P^{\eps,\beta}_s|f|(y) + \sup_{s \in [t,T]}P^{\eps,\beta}_s|f|(y)\\
	\le&\  C \norm{f}_{\cC_0(\Omega)} e^{-C'\, \delta/ \sqrt t} + \norm{f}_{\cC_0(\Omega)}  
	\sup_{x \in (\Omega\setminus U_n)_\eps}  \mbfE^{\eps,\infty}_x\quadre{\exp\tonde{-\eps^{\beta-1}\int_0^t V_\eps(X_r^{\eps,\infty})\, \dd r}}\comma
\end{align*}
where in the last inequality we used the exit-time estimate \eqref{eq:exit-time}. Letting first $\eps\to 0$ and then~$n\to \infty$, by $\beta<1$, \eqref{eq:joint_conv_RW_LT}, and Proposition~\ref{pr:equi_semi_disc},  
\begin{align*}
&\ \limsup_{n\to \infty}\ \limsup_{\eps\downarrow 0}	\sup_{x \in (\Omega\setminus U_n)_\eps} \mbfE^{\eps,\infty}_x\quadre{\exp\tonde{-\eps^{\beta-1}\int_0^t V_\eps(X_r^{\eps,\infty})\, \dd r}}  \\
\leq&\ \limsup_{n\to \infty}\ \limsup_{\eps\downarrow 0}	\sup_{x \in (\Omega\setminus U_n)_\eps} \mbfE^{\eps,\infty}_x\quadre{\exp\tonde{-\varrho\int_0^t V_\eps(X_r^{\eps,\infty})\, \dd r}} 
=
\sup_{x\in \partial\Omega} \mbfE^\Neu_x\quadre{e^{-\varrho\, 	L^\Neu_t}} \comma
\end{align*}
for all $\varrho >0$.
Hence, 
\begin{align*}
\limsup_{n\to \infty} \ \limsup_{\eps\downarrow 0}	 \sup_{\substack{s \in [0,T] \\   y \in (\Omega\setminus U_n)_\eps}}  P^{\eps,\beta}_s|f|(y)\le \norm{f}_{\cC_0(\Omega)} \tonde{C e^{-C'\, \delta/\sqrt t}  +  \sup_{x \in \partial\Omega} \mbfE^\Neu_x\quadre{e^{-\varrho\, L^\Neu_t}}}\fstop
\end{align*}
Taking the limits  $\varrho\to \infty$  and $t \to 0$, we conclude the proof by showing that
\begin{align*}
	\lim_{\varrho\to \infty}\sup_{x\in \partial\Omega} \mbfE^\Neu_x\quadre{e^{-\varrho\, L^N_t}}=0 \comma
	\qquad t > 0\fstop
\end{align*}

By \eqref{eq:semigr_Robin_localtime},  $\mbfE^\Neu_\emparg \big[ {e^{-\varrho\, L^N_t}} \big] = P^{\varrho}_t \car_\Omega \in \cC(\overline\Omega)$ for $t>0$ and $\varrho >0$.
In view of Lemma~\ref{l:robin_to_tutto}($b_\Dir$) and Proposition~\ref{pr:equi_semi_cont}, we have that $\lim_{\varrho \to \infty} \norm {P_t^\varrho \car_\Omega - P_t^\Dir \car_\Omega}_{\mcC(\overline \Omega)}=0$, which concludes the proof of Theorem~\ref{t:MainSemigroups} ($\beta<1$) since $P_t^\Dir \car_\Omega \equiv 0$ everyhere on~$\partial \Omega$.

\section{Proof of Theorem~\ref{th:harmonic_conv}}\label{sec:harmonic_conv}
Let us now turn to the convergence of the discrete harmonic profiles $h^{\eps,\beta}$ to the continuum ones $h^\bd$, both introduced in \S\ref{sec:harmonic-profiles}.
As done in \S\ref{sec:semigroup-conv-proofs} for the proof of Theorem \ref{t:MainSemigroups}, we divide the proof of Theorem \ref{th:harmonic_conv} according to the boundary conditions: \S\ref{sec:harmonic_proof_neumann} is devoted to the case $\beta >1$, \S\ref{sec:harmonic_proof_robin} to $\beta = 1$, and \S\ref{sec:harmonic_dirichlet} to $\beta <1$. Finally,  the same notational conventions adopted in \S\ref{sec:semigroup-conv-proofs} hold all throughout this section.

\subsection{Proof of Theorem~\ref{th:harmonic_conv}: Neumann regime}\label{sec:harmonic_proof_neumann}
Throughout this section, fix~$\beta>1$. Let
	\begin{equation}\label{eq:h-spatial_average}\bar h^{\eps,\beta}\coloneqq \frac{\eps^d}{\mu_\eps(\Omega_\eps)}\sum_{x\in \Omega_\eps} h^{\eps,\beta}(x)
		\end{equation} denote the spatial average of $h^{\eps,\beta}$, and recall (see \S\ref{sec:harmonic-profiles}) 
	\begin{equation}\label{eq:h-Neu}
		h^\Neu= \av{\vartheta}_{\partial \Omega}\eqdef \fint_{\partial\Omega} \vartheta\, \dd \sigma_\Omega\fstop
	\end{equation} 
By the following triangle inequality 	 
\begin{equation}
	\tnorm{h^{\eps,\beta}-h^\Neu}_{L^\infty(\Omega_\eps)}\le \tnorm{h^{\eps,\beta}-\bar h^{\eps,\beta}}_{L^\infty(\Omega_\eps)}+\tabs{\bar h^{\eps,\beta}-h^\Neu}\comma
\end{equation}
the desired claim in Theorem \ref{th:harmonic_conv}  follows at once from Lemmas \ref{lemma:harmonic_neumann1} and \ref{lemma:harmonic_neumann2} below. 
\begin{lemma}\label{lemma:harmonic_neumann1}
Recall \eqref{eq:h-spatial_average}.
Then, 
$
\lim_{\eps\downarrow 0}\norm{h^{\eps,\beta}-\bar h^{\eps,\beta}}_{L^\infty(\Omega_\eps)}=0$.
\begin{proof}
Since~$P^{\eps,\beta}_t h^{\eps,\beta} = h^{\eps,\beta}$, the triangle inequality yields, for all $x\in \Omega_\eps$ and $t>0$, 
\begin{align*}
&\tabs{h^{\eps,\beta}(x)-\bar h^{\eps,\beta}}\le \abs{\eps^d\sum_{y\in \Omega_\eps} \tonde{\frac{p^{\eps,\infty}_t(x,y)}{\eps^d}-\frac{1}{\mu_\eps(\Omega_\eps)}}h^{\eps,\beta}(y)}
\\
&\qquad\qquad +\abs{\sum_{y\in\Omega_\eps}\ttonde{p^{\eps,\beta}_t(x,y)-p^{\eps,\infty}_t(x,y)}h^{\eps,\beta}(y)}+\sum_{z\in\partial_e\Omega_\eps} p^{\eps,\beta}_t(x,z)\,\vartheta(z)
\\
& \qquad\le \norm{\vartheta}_{\cC_b(\R^d)} \tonde{\norm{\frac{p^{\eps,\infty}_t(x,\cdot)}{\eps^d}-\frac{1}{\mu_\eps(\Omega_\eps)}}_{L^1(\Omega_\eps)} + 2 \tonde{ 1- \mbfP^{\eps,\beta}_x\ttonde{X^{\eps,\beta}_t\in \Omega_\eps}} }\comma
\end{align*}
where the last inequality follows from $\tnorm{h^{\eps,\beta}}_{L^\infty(\Omega_\eps)}\le \norm{\vartheta}_{\cC_b(\R^d)}$, \eqref{eq:domination_HK} and 
	\begin{equation*}
	\sum_{z\in \partial_e\Omega_\eps} p^{\eps,\beta}_t(x,z)= 1-\sum_{y\in \Omega_\eps} p^{\eps,\beta}_t(x,y)= 1- \mbfP^{\eps,\beta}_x\ttonde{X^{\eps,\beta}_t\in \Omega_\eps}\fstop
\end{equation*} 
Passing to the supremum over $x\in \Omega_\eps$, the proof ends by taking first  $\eps\to 0$ and then $t\to \infty$. Indeed, by H\"older inequality,  \eqref{eq:VolumeBound} and \eqref{eq:mixing_time_estimate}, 
\begin{equation*}
	\lim_{t\to \infty}\limsup_{\eps\downarrow 0} \norm{\frac{p^{\eps,\infty}_t(x,\cdot)}{\eps^d}-\frac{1}{\mu_\eps(\Omega_\eps)}}_{L^1(\Omega_\eps)}=0\comma
\end{equation*}
while by \eqref{eq:feynman-kac} and the very same argument used in the proof of Lemma \ref{l:UnifCont}, 
\begin{equation*}
\limsup_{\eps\downarrow 0}	\sup_{x\in\Omega_\eps}1- \mbfP^{\eps,\beta}_x\ttonde{X^{\eps,\beta}_t\in \Omega_\eps}=0\comma\qquad t > 0\fstop \qedhere
\end{equation*}
\end{proof}
\end{lemma}

\begin{lemma}\label{lemma:harmonic_neumann2}
	Recall \eqref{eq:h-spatial_average} and \eqref{eq:h-Neu}. Then, $\lim_{\eps\to 0}\bar h^{\eps,\beta}= \av{\vartheta}_{\partial\Omega}$ .
\begin{proof}
	 Introduce the following function $h^\Neu_\eps:\overline\Omega_\eps\to [0,\infty)$:
	\begin{equation*}
		h^\Neu_\eps(x)\coloneqq \begin{cases}
			\av{\vartheta}_{\partial \Omega} &\text{if}\ x \in \Omega_\eps\\
			\vartheta(x) &\text{if}\ x \in \partial_e\Omega_\eps
		\end{cases}\fstop
	\end{equation*}

Further observe that, since  $h^{\eps,\beta}$ and $\purple{h^\Neu_\eps}$  coincide on $\partial_e\Omega_\eps$,  
\begin{equation}\label{eq:h_eps=h+W}
	h^{\eps,\beta}(x)= h^\Neu_\eps(x)+\int_0^\infty P^{\eps,\beta}_tA^{\eps,\beta}h^\Neu_\eps(x)\, \diff t\comma \qquad x \in \overline\Omega_\eps\fstop
\end{equation}
Hence,  by the definitions	 of $\bar h^{\eps,\beta}$ and $h^\Neu_\eps$, 
\begin{align*}
&	\abs{\bar h^{\eps,\beta}-\av{\vartheta}_{\partial \Omega} }= \abs{\frac{\eps^d}{\mu_\eps(\Omega_\eps)} \sum_{x\in \Omega_\eps} \ttonde{ h^{\eps,\beta}(x)-h^\Neu_\eps(x)}}\\
	&\qquad= \abs{\frac{\eps^d}{\mu_\eps(\Omega_\eps)}\sum_{x\in\Omega_\eps} \tonde{ \int_0^\infty P^{\eps,\beta}_t \car_{\Omega_\eps}(x)\, \diff t} A^{\eps,\beta}h^\Neu_\eps(x) }\\
	&\qquad	=\abs{\int_0^\infty\eps^{\beta-1}\tonde{\frac{\eps^{d-1}}{\mu_\eps(\Omega_\eps)}\sum_{x\in \partial\Omega_\eps}  P^{\eps,\beta}_t\car_{\Omega_\eps}(x) \sum_{\substack{z\in \partial_e\Omega_\eps\\z \sim x}}\axz\tonde{\vartheta(z)-\av{\vartheta}_{\partial \Omega}}}\diff t }\comma	
\end{align*}
where the second identity uses the fact that~$A^{\eps,\beta}h^\Neu_\eps =0$ on $\partial_e\Omega_\eps$ and the symmetry of~$P^{\eps,\beta}_t$ on~$L^2(\Omega_\eps)$, while the third one uses that $h^\Neu_\eps$ is constant on $\Omega_\eps$. 

In view of \eqref{eq:bd_weak_conv} and Remark \ref{rem:bd_ext}, the conclusion follows by showing that there exists a uniformly bounded family $\ttseq{b_\eps}_\eps\subset \R^+_0$  such that
\begin{equation*}
\limsup_{\eps\downarrow 0} \norm{\frac{\eps^{\beta-1}}{\mu_\eps(\Omega_\eps)}\int_0^\infty P^{\eps,\beta}_t\car_{\Omega_\eps}(\emparg)\,\diff t- b_\eps}_{L^\infty(\Omega_\eps)}=0\fstop
\end{equation*}
In particular, by \eqref{eq:uniform_op_bound_beta} and Proposition~\ref{p:ground_states}\ref{i:p:ground_states:gs},  we can choose $b_\eps = \frac{\eps^{\beta-1}}{\mu_\eps(\Omega_\eps) \lambda_0^{\eps,\beta}}$, which is uniformly bounded by Proposition~\ref{p:ground_states}\ref{i:p:ground_states:ev} (see below).
\end{proof}
\end{lemma}

\subsubsection{Spectral bounds and ground states}
We conclude the proof of Theorem~\ref{th:harmonic_conv} for $\beta>1$ with a last proposition.
Recall, for all $\beta \in\R$, the definition of $\lambda_0^{\eps,\beta} \geq 0$ in Theorem~\ref{t:Equivalence}\ref{i:t:Equivalence:8}, and further define the \emph{ground state} $\psi_0^{\eps,\beta}$ as the unique positive function in $L^2(\Omega_\eps)$  solving 
\begin{equation}\label{eq:eigen_eq}
	-A^{\eps,\beta}\psi_0^{\eps,\beta}=\lambda_0^{\eps,\beta}\psi_0^{\eps,\beta}\comma 
\quad 
		\psi_0^{\eps,\beta} \in \arg \min 
			\left\{
				\mcE^{\eps,\beta}(f) \ : \ \norm{f}_{L^2(\Omega_\eps)}^2= 1 
			\right\} \fstop  
\end{equation}
Finally, set $\psi_0^{\eps,\infty} \eqdef \mu_\eps(\Omega_\eps)^{-1/2}\car_{\Omega_\eps}$. 
\begin{proposition}
\label{p:ground_states}
	For all $\beta >1$, the following properties hold true:
	\begin{enumerate}[$(i)$]
			\item \label{i:p:ground_states:ev}
		\emph{Spectral bound:} \ there exist $0<\underline{\lambda}_0 \leq \overline{\lambda}_0$ such that 
		$
			\displaystyle
		\underline{\lambda}_0\, \varepsilon^{\beta-1}	 \leq \lambda_0^{\varepsilon,\beta} \leq \overline{\lambda}_0\,  \eps^{\beta-1} \fstop
		$
		\item\label{i:p:ground_states:gs}
		\emph{Ground states:} \
		$
			\displaystyle 
		\lim_{\eps\downarrow 0}\tnorm{\psi_0^{\eps,\beta}-\psi_0^{\eps,\infty}}_{L^\infty(\Omega_\eps)}=0\fstop
		$
	\end{enumerate}
	\begin{proof}
		The upper bound in \ref{i:p:ground_states:ev} easily follows by choosing $f= \psi_0^{\eps,\infty}$ in \eqref{eq:eigen_eq}.
		We now show \ref{i:p:ground_states:gs}; we set $\overline \psi_0^{\eps,\beta}\coloneqq \frac{\eps^d}{\mu_\eps(\Omega_\eps)}
		\sum_{x\in \Omega_\eps}\psi_0^{\eps,\beta}(x)
		$, 	and claim that  
		\begin{equation}\label{eq:b0}
			\lim_{\eps\downarrow 0}\tnorm{\psi_0^{\eps,\beta}-\overline \psi_0^{\eps,\beta}}_{L^\infty(\Omega_\eps)}=0\fstop
		\end{equation}
		By $\psi_0^{\eps,\beta}(x)=e^{\lambda_0^{\eps,\beta}t} P^{\eps,\beta}_t \psi_0^{\eps,\beta}(x)$ for any $t\ge1$, we infer that
		\begin{align}
			\label{eq:b1}
			\tnorm{\psi_0^{\eps,\beta}-\overline \psi_0^{\eps,\beta}}_{L^\infty(\Omega_\eps)}
		\le  \tabs{e^{\lambda_0^{\eps,\beta}t}-1} \tnorm{\psi_0^{\eps,\beta}}_{L^\infty(\Omega_\eps)} + \max_{x\in \Omega_\eps}\abs{P  ^{\eps,\beta}_t \psi_0^{\eps,\beta}(x)-\overline \psi_0^{\eps,\beta}} \comma	
			\end{align}
		and, by the domination property~\eqref{eq:domination_HK} and \eqref{eq:ultracontractivity}, that
		\begin{equation}
				\label{eq:bound_unif_psi0_eps_beta}
			\sup_{\eps>0}	\tnorm{\psi_0^{\eps,\beta}}_{L^\infty(\Omega_\eps)}<\infty\fstop
		\end{equation}
		The first term on the right-hand side of \eqref{eq:b1} vanishes as $\eps \to 0$ by \eqref{eq:bound_unif_psi0_eps_beta} and the upper bound in \ref{i:p:ground_states:ev}.
		Moreover, 	
		\begin{align*}
		&\abs{P^{\eps,\beta}_t \psi_0^{\eps,\beta}(x)-\overline \psi_0^{\eps,\beta}}= \abs{\eps^d \sum_{y\in \Omega_\eps}\tonde{\frac{p^{\eps,\beta}_t(x,y)}{\eps^d}-\frac{1}{\mu_\eps(\Omega_\eps)}}\psi_0^{\eps,\beta}(y)} \longrightarrow 0 \comma 
		\end{align*}
	uniformly for $x \in \Omega_\eps$ as $\eps \to 0$, which follows by \eqref{eq:bound_unif_psi0_eps_beta} and arguing as in Lemma~\ref{lemma:harmonic_neumann1}.
	This proves the claim in \eqref{eq:b0}. We conclude the proof by \eqref{eq:b0}, \eqref{eq:bound_unif_psi0_eps_beta}, and $\norm{\psi_0^{\eps,\beta}}_{L^2(\Omega_\eps)} =1$, which implies \ref{i:p:ground_states:gs}.
	We are left with the proof of the lower bound in \ref{i:p:ground_states:ev}. By \eqref{eq:eigen_eq},  
			\begin{equation*}
					\lambda_0^{\eps,\beta}=	\mcE^{\eps,\beta}(\psi_0^{\eps,\beta})\ge \eps^{\beta-1}\tonde{\eps^{d-1}\sum_{x\in \partial\Omega_\eps}\alpha_\eps(x)\tonde{\psi_0^{\eps,\beta}(x)}^2}\fstop
				\end{equation*}
			By \ref{i:p:ground_states:gs} and \eqref{eq:bd_weak_conv}, we have that
			\begin{align*}
					\lim_{\eps\downarrow 0}\eps^{d-1}\sum_{x\in \partial\Omega_\eps}\alpha_\eps(x)\tonde{\psi_0^{\eps,\beta}(x)}^2 = \mu_\Omega(\Omega)^{-1} \  \scalar{\sigma_{\partial \Omega}}{1}  \in (0,\infty)\comma
				\end{align*}
			from which the desired claim follows.
	\end{proof}
\end{proposition}

\begin{remark}
As a consequence of the domination property \eqref{eq:domination_HK} and the ultracontractivity in \eqref{eq:ultracontractivity}, by Corollary~\ref{c:spectral_bound_DR} ($\beta\leq 1$) and Proposition~\ref{p:ground_states}\ref{i:p:ground_states:ev} ($\beta>1$), we have that
	\begin{gather}
		\label{eq:uniform_op_bound_beta_puntuale}
	\tnorm{P_s^{\eps,\beta}}_{L^1(\Omega_\eps) \to L^\infty(\Omega_\eps)} \leq C \exp \big (- (\eps^{\beta-1} \wedge 1)s \big) \comma 
	\quad \text{for} \ \ s>1 \comma  
\\
		\label{eq:uniform_op_bound_beta}
			\sup_{\eps \in (0,1)}
		\left\|
		\ttonde{\eps^{\beta-1} \wedge  1} 
		\int_0^\infty P_t^{\eps,\beta} \dd t
		\right\|_{L^\infty(\Omega_\eps) \to L^\infty(\Omega_\eps)}
		< \infty \fstop
\end{gather}
\end{remark}

\subsection{Proof of Theorem~\ref{th:harmonic_conv}: Robin regime}\label{sec:harmonic_proof_robin}
\purple{Recall the definition~\eqref{eq:harmonic_profiles_discrete} of the discrete harmonic profile~$h^{\eps,\beta}$.}
By the master equation and \eqref{eq:feynman-kac},  we have that, for all $x \in \Omega_\eps$ and $z \in \partial_e\Omega_\eps$, 
\begin{equation}\label{eq:master-eq}
\begin{aligned}
&	p^{\eps,\beta}_\infty(x,z)= \int_0^\infty \eps^{\beta-2}\sum_{\substack{y\in \partial\Omega_\eps\\
			z \sim y
	}} \ayz	p^{\eps,\beta}_t(x,y)\, \dd t \\
 &\quad = \int_0^\infty \eps^{\beta-2} \sum_{\substack{y \in \partial\Omega_\eps\\z \sim y}} \ayz\, \mbfE^{\eps,\infty}_\purple{x}\quadre{\car_{\{y\} }(X^{\eps,\infty}_t)\exp\tonde{-\eps^{\beta-1}\int_0^t V_\eps(X^{\eps,\infty}_r)\, \dd r}}\dd t\fstop
\end{aligned}
\end{equation}
Hence, for $x \in \Omega_\eps$,  
\begin{equation*}
	h^{\eps,\beta=1}(x)=  \int_0^\infty \mbfE_x^{\eps,\infty}\quadre{V_\eps^\vartheta(X^{\eps,\infty}_t)\exp\tonde{-	\int_0^t V_\eps(X^{\eps,\infty}_r)\,\dd r}}\dd t + \cJ_\eps(x)\comma
\end{equation*}
where
$
V_\eps^\vartheta(x)\coloneqq	\eps^{-1}\car_{\partial\Omega_\eps}(x)\, 	 \vartheta(x)\, \alpha_\eps(x)
$
and $\cJ_\eps(x)$ is given by the expression
\begin{align*}
	\purple{-}\frac1{\eps}\int_0^\infty \mbfE_x^{\eps,\infty}\quadre{\sum_{y \in \partial\Omega_\eps}\car_{\set y}(X^{\eps,\infty}_t)\sum_{z \in \partial_e\Omega_\eps}\ayz \ttonde{\vartheta(y)-\vartheta(z)} \exp\tonde{-\int_0^t V_\eps(X^{\eps,\infty}_r)\, \dd r}}\dd t\fstop
	\end{align*}
Furthermore, for all $T > 0$ and $x \in \Omega_\eps$, we have that
\begin{align}
		\label{eq:tired_of_giving_name}
	\begin{aligned}
		\abs{\cJ_\eps(x)}
	\le& \tonde{\sup_{y \in \partial\Omega_\eps}\sum_{\substack{z \in \partial_e\Omega_\eps\\ z \sim y}} \ayz\abs{\vartheta(z)-\vartheta(y)}}\int_0^T \eps^{-1}\sum_{y \in \partial\Omega_\eps}p^{\eps,\infty}_t(x,y)\,\dd t\\
	+&\  2 \norm{\vartheta}_\infty	\int_T^\infty \mbfE^{\eps,\infty}_x\quadre{V_\eps(X^{\eps,\infty}_t)\exp\tonde{-\int_0^t V_\eps(X^{\eps,\infty}_r)\, \dd r}}\dd t\fstop
	\end{aligned}
\end{align}
The first term on the right-hand side vanishes by the uniform continuity of $\vartheta$, \eqref{eq:bd_ext}, and \eqref{eq:HKbd}; by \eqref{eq:feynman-kac_dt}, \eqref{eq:feynman-kac}, and  \eqref{eq:uniform_op_bound_beta_puntuale}, the second term is controlled by $e^{-\underline{\lambda}_0 T}$. 

Now, recall the stochastic representation \eqref{eq:harmonic_profile_R_LT} of $h^{\varrho=1}:\overline\Omega\to [0,\infty)$. 
Then, for all $T>1$, by \eqref{eq:joint_conv_RW_LT} and \eqref{eq:tired_of_giving_name}  we obtain that 
\begin{align*}
	\limsup_{\eps \downarrow 0}
		\tnorm{h^{\eps,\beta}-h^{\varrho=1}}_{L^\infty(\Omega_\eps)}
	\leq 
	\limsup_{\eps \downarrow 0}
		\norm{\cJ_\eps}_{L^\infty(\Omega_\eps)}
	\leq 
		2\norm{\vartheta}_\infty C e^{-\underline\lambda_0 T}
		\fstop
\end{align*}
Taking the limit as $T \to \infty$, we conclude the proof.

\subsection{Proof of Theorem~\ref{th:harmonic_conv}: Dirichlet regime}\label{sec:harmonic_dirichlet}
Throughout this section, fix~$\beta<1$.
We divide the proof into two main parts. In the first part we treat the case of smooth domains and smooth boundary data, while in the second part we prove the claim for $\Omega$ a bounded Lipschitz domain case with continuous boundary data.
\subsubsection{Case of bounded smooth domain \texorpdfstring{$\Omega$}{Omega} and smooth boundary data \texorpdfstring{$\vartheta$}{vartheta}}\label{sec:conv_dir_smooth_harmonics} Provided that $\Omega$ is a bounded $\cC^\infty$-domain and the (non-negative) boundary datum $\vartheta$ is in $\cC^\infty(\R^d)$, then there exists a unique $\cC^\infty(\overline\Omega)$-solution, say $h^\Dir$, of the corresponding Dirichlet problem on $\Omega$ (see, e.g., \cite[Thm.~2.5.1.1]{Gri85}).
On $\overline\Omega_\eps$, we define the function 
\begin{equation}\label{eq:hDireps}
	h^\Dir_\eps \coloneqq \car_{\Omega_\eps} h^\Dir + \car_{\partial_e\Omega_\eps} \vartheta\comma
\end{equation}
and note that we have the decomposition
$
	h^{\eps,\beta}= h^\Dir_\eps + \int_0^\infty P^{\eps,\beta}_tA^{\eps,\beta}h^\Dir_\eps\, \dd t 
$.
The claim of Theorem~\ref{th:harmonic_conv} ($\beta<1$) follows from  $\lim_{\eps\downarrow 0}\norm{\int_0^\infty P^{\eps,\beta}_t A^{\eps,\beta} h^\Dir_\eps\, \dd t}_{L^\infty(\Omega_\eps)}=0$.

By $\Delta h^\Dir =0$ on $\Omega$,  $h^\Dir \in \cC^3(\overline\Omega)$ and \eqref{eq:bd_unif_ellipticity},  there exists $C=C(\Omega,h^\Dir)>0$ such that
\begin{equation}
	\norm{	\Delta_\eps h^\Dir}_{L^\infty(\Omega_\eps\setminus \partial\Omega_\eps)} \le C \eps \comma\qquad
\tnorm{A^{\eps,\beta}h^\Dir_\eps}_{L^\infty(\partial\Omega_\eps)}\le  C \ttonde{\eps^{-1}\vee \eps^{\beta-1}}\fstop
\end{equation}
As a consequence, we obtain
\begin{align*}
	&\norm{\int_0^\infty P^{\eps,\beta}_t A^{\eps,\beta} h^\Dir_\eps\, \dd t}_{L^\infty(\Omega_\eps)}
	\\
	\le
	&\ \sup_{x\in \Omega_\eps}\abs{\int_0^\infty \sum_{y\in \Omega_\eps\setminus\partial\Omega_\eps} p^{\eps,\beta}_t(x,y)\, \Delta_\eps h^\Dir_\eps(y)\, \dd t}
	+ \sup_{x\in\Omega_\eps}\abs{\int_0^\infty \sum_{y\in \partial\Omega_\eps} p^{\eps,\beta}_t(x,y)\, A^{\eps,\beta}h^\Dir_\eps(y)\, \dd t}
	\\
	\le&\ C\tonde{ \eps\,	\norm{ \int_0^\infty P^{\eps,\beta}_t \, \dd t}_{L^\infty(\Omega_\eps)\to L^\infty(\Omega_\eps)}
	+ \tonde{\eps^{-1}\vee \eps^{\beta-1}} \sup_{x\in\Omega_\eps}\int_0^\infty \sum_{y\in \partial\Omega_\eps} p^{\eps,\beta}_t(x,y)\, \dd t}\fstop
\end{align*}
We conclude the proof by \eqref{eq:uniform_op_bound_beta}  and \eqref{eq:bd_kernel_scaling}.

\subsubsection{General case}
Recall from \S\ref{sec:harmonic-profiles} that the Dirichlet problem~\eqref{eq:Dir_problem} with boundary condition~$\vartheta\vert_{\partial\Omega} \in \cC(\partial\Omega)$ (see~\eqref{eq:BoundaryCondition}) admits a unique solution~$h^\Dir$ with~$h^\Dir\in \cC(\overline\Omega)\cap \cC^\infty(\Omega)$.
Further, recall from Section \ref{sec:conv_dir_lipschitz} the  approximating smooth sets $U_n\nearrow \Omega$,  corresponding stopped semigroups $P^{\Dir,n}_t$ and $P^{\eps,n}_t$, as well as $(\Omega\setminus U_n)_\eps\eqdef\Omega_\eps \setminus (\Omega_\eps \cap U_n)$. Then, for all~$n \in \N$, consider the harmonic profile $h^{\eps,n}:\Omega_\eps\to [0,\infty)$ associated to  $P^{\eps,n}_t$ with boundary data $h^\Dir$ on $(\Omega\setminus U_n)_\eps$, i.e., 
\begin{equation}
	h^{\eps,n}(x)\coloneqq \lim_{t\to\infty} P^{\eps,n}_t h^\Dir(x)\comma \qquad x \in \Omega_\eps\fstop
\end{equation}
Note that $h^{\eps,n}= h^\Dir$ on $(\Omega\setminus U_n)_\eps$.
By the smoothness of the domains $U_n$ and of the boundary data $h^\Dir\vert_{\partial U_n}$, the arguments in Section \ref{sec:conv_dir_smooth_harmonics}  ensure 
\begin{equation}
	\lim_{\eps\downarrow0}\norm{h^{\eps,n}-h^\Dir}_{L^\infty(\Omega_\eps)}=\lim_{\eps\downarrow0}\norm{h^{\eps,n}-h^\Dir}_{L^\infty(\Omega_\eps\cap U_n)}=0\comma\qquad n \in \N\fstop
\end{equation}
To conclude the proof, we will show that 
\begin{align}	\label{eq:final_harm_d}
	\lim_{n \to \infty}
	\limsup_{\eps \downarrow 0}
		\tnorm{h^{\eps,\beta}-h^{\eps,n}}_{L^\infty(\Omega_\eps)} = 0 \fstop
\end{align}
	
Recall $h^\Dir_\eps$ from \eqref{eq:hDireps}; then, by the strong Markov property,
\begin{align}	\label{eq:h_eps_beta_n}
	\tnorm{h^{\eps,\beta}-h^{\eps,n}}_{L^\infty(\Omega_\eps)} \leq \sup_{y \in (\Omega\setminus U_n)_\eps}\tabs{P^{\eps,\beta}_\infty h^\Dir_\eps(y)-h^\Dir_\eps(y)}\comma \qquad n \in \N\fstop
\end{align} For every $\delta>0$,  let $w_{\delta}(h^\Dir)$ denote the $\delta$-modulus of continuity of $h^\Dir\in \cC(\overline\Omega)$; then, for fixed~$n \in \N$,  and for all $\eps \in (0,1)$ small enough, by \eqref{eq:bd_ext} and the uniform continuity of $\vartheta$,  
\begin{equation}	\label{eq:P_infty_h_eps_D}
	\norm{P^{\eps,\beta}_\infty h^\Dir_\eps-h^\Dir_\eps}_{L^\infty((\Omega\setminus U_n)_\eps)} \le w_{2\delta}(h^\Dir)+	 2 \norm{\vartheta}_{\cC_b(\R^d)} \sup_{y\in (\Omega\setminus U_n)_\eps} \sum_{\substack{z\in\partial_e\Omega_\eps\\ \abs{z-y}\ge \delta}} p^{\eps,\beta}_\infty(y,z)\fstop
\end{equation}
Define, for all $\delta>0$ and $y \in \Omega$ (cf.\ \eqref{eq:alpha_eps}), 
\begin{equation}
	\overline	V_\eps^{\delta,y}(x)\coloneqq   \eps^{-1}\car_{x\in \partial\Omega_\eps}\car_{\abs{x-y}\ge \delta/2}\, \alpha_\eps(x)\comma\qquad x \in \Omega_\eps 	\fstop
\end{equation}
Hence, by \eqref{eq:master-eq} and \eqref{eq:feynman-kac}, for $y \in (\Omega\setminus U_n)_\eps$, 
\begin{align*}
&\sum_{\substack{z\in \partial_e\Omega_\eps\\
			\abs{z-y}\ge\delta}} p^{\eps,\beta}_\infty(y,z)
		 \le
		  \int_0^\infty \eps^{\beta-2}\sum_{\substack{x\in \partial\Omega_\eps\\
		\abs{x-y}\ge \delta/2} } p^{\eps,\beta}_s(y,x)\, \alpha_\eps(x)\, \dd s\\
	&\qquad= 
		 \int_0^\infty \mbfE^{\eps,\infty}_y\quadre{\eps^{\beta-1}\,\overline V^{\delta,y}_\eps(X^{\eps,\infty}_s)\exp\tonde{-\eps^{\beta-1}\int_0^s V_\eps(X^{\eps,\infty}_r)\, \dd r}}\dd s\fstop
\end{align*}

Split the above integral at a fixed time $t>0$. By $\overline V_\eps^{\delta,y}\le V_\eps$, 
\begin{equation*}
	\begin{aligned}
	&	\sup_{y\in(\Omega\setminus U_n)_\eps} \int_t^\infty \mbfE^{\eps,\infty}_y\quadre{\eps^{\beta-1}\,\overline V^{\delta,y}_\eps(X^{\eps,\infty}_s)
		\exp\tonde{-\eps^{\beta-1}\int_0^s V_\eps(X^{\eps,\infty}_r)\, \dd r}
	}\dd s\\
&\qquad \le  \sup_{y\in(\Omega\setminus U_n)_\eps} \int_t^\infty \mbfE^{\eps,\infty}_y\quadre{\eps^{\beta-1}\, V_\eps(X^{\eps,\infty}_s)
	\exp\tonde{-\eps^{\beta-1}\int_0^s V_\eps(X^{\eps,\infty}_r)\, \dd r}
}\dd s\purple{\fstop}
\end{aligned}
\end{equation*}
By \eqref{eq:feynman-kac_dt}, arguing as in the proof of Theorem~\ref{t:MainSemigroups} ($\beta<1$, end of \purple{\S\ref{sec:conv_dir_lipschitz}}), the right-hand side above vanishes as $\eps\to0$, $n\to \infty$, for fixed $t>	 0$. On the other hand, again by $\overline V_\eps^{\delta,y}\le V_\eps$,
\begin{equation*}
	\begin{aligned}
		&	\sup_{y\in(\Omega\setminus U_n)_\eps} \int_0^t \mbfE^{\eps,\infty}_y\quadre{\eps^{\beta-1}\,\overline V^{\delta,y}_\eps(X^{\eps,\infty}_s)
			\exp\tonde{-\eps^{\beta-1}\int_0^s V_\eps(X^{\eps,\infty}_r)\, \dd r}
		}\dd s\\
		&\qquad \le  \sup_{y\in(\Omega\setminus U_n)_\eps} \int_0^t \mbfE^{\eps,\infty}_y\quadre{\eps^{\beta-1}\, \overline V^{\delta,y}_\eps(X^{\eps,\infty}_s)
			\exp\tonde{-\eps^{\beta-1}\int_0^s \overline V^{\delta,y}_\eps(X^{\eps,\infty}_r)\, \dd r}
		}\dd s\comma
	\end{aligned}
\end{equation*}
which, by \eqref{eq:feynman-kac_dt} and \eqref{eq:exit-time}, vanishes (uniformly in $\eps$ and $n$) as $t\to 0$.  
Combining these estimates with \eqref{eq:h_eps_beta_n} and \eqref{eq:P_infty_h_eps_D}, we obtain \eqref{eq:final_harm_d}. This concludes the proof of Theorem \ref{th:harmonic_conv}.

\section{Proofs of Theorems \ref{t:MainHydrodynLim},~\ref{th:hydrostatic}, and~\ref{t:LocalEquilibrium}}	\label{s:proofs-IPS}
In this section we prove the hydrodynamic  and  hydrostatic limits (Thm.s~\ref{t:MainHydrodynLim} and  \ref{th:hydrostatic}, respectively), as well as Theorem~\ref{t:LocalEquilibrium} on  stationary correlations. 
While \S\ref{s:duality} presents the  auxiliary  dual processes and some of their main properties, the proofs of Theorems~\ref{t:MainHydrodynLim}, \ref{th:hydrostatic} and \ref{t:LocalEquilibrium} are the subjects of~\S\ref{sec:proof-HDL}, \S\ref{sec:hydrostatic} and~\S\ref{sec:hk-conv}, respectively.

\subsection{Dual particle systems, duality functions, and properties}\label{s:duality} For fixed~$k, \ell \in \N$, 
\begin{equation*}
	\bx\eqdef(x_1,\ldots, x_k)\in (\overline\Omega_\eps)^k\qquad \text{and}\qquad \by\eqdef(y_1,\ldots, y_\ell)\in(\overline\Omega_\eps)^\ell\comma
\end{equation*}
we define 
\begin{equation}\label{eq:bx-by-etc}
\begin{aligned}
	\hat\bx_i&\eqdef(x_1,\ldots, x_{i-1},x_{i+1},\ldots, x_k)\in (\overline\Omega_\eps)^{k-1}\comma
	\\
	\bx_i^{y}&\eqdef(x_1,\ldots, x_{i-1},y,x_{i+1},\ldots, x_k)\in(\overline\Omega_\eps)^k\comma && i \in \set{1,\ldots, k}, y\in \overline\Omega_\eps\comma
	\\
	\bx\:\by&\eqdef(x_1,\ldots,x_k,y_1,\ldots, y_\ell)\in(\overline\Omega_\eps)^{k+\ell}\comma	
\end{aligned}
\end{equation}
and, for all \purple{$\omega\in \R^{(\overline\Omega_\eps)^k}$},	
\begin{align}\label{eq:omega-bx}
	\omega[\bx]&\eqdef \prod_{i=1}^k\tonde{\omega(x_i)+\sigma\sum_{j=1}^{i-1}\car_{x_j}(x_i)}\comma
		\\
	\label{eq:omega-bx2}
	\omega[\by|\bx]&\eqdef \begin{cases} \omega[\bx\:\by]/\omega[\bx] &\text{if}\ \omega[\bx]\neq 0\\
		1&\text{otherwise}
	\end{cases}\fstop
\end{align}
Note that~$\car_{\Omega_\eps}[\emparg]\colon (\overline\Omega_\eps)^k\to\R$, whereas~$\car_{\Omega_\eps}(\emparg)\colon \overline\Omega_\eps\to\R$.

\subsubsection{Particle dynamics}\label{sss:k-particle-dynamics} For all $\beta\in \R$, we denote by $\ttonde{\ttonde{\mbfX^{\eps,\beta,k}_t}_{t\ge0},\ttonde{\mbfP^{\eps,\beta,k}_\bx}_{\bx\in (\overline\Omega_\eps)^k}}$ the continuous-time Markov chain in the Skorokhod space $\mcD(\R_0^+;(\overline\Omega_\eps)^k)$ with generator
\begin{equation}\label{eq:generator_dual}
	\begin{aligned}
		A^{\eps,\beta,k} f (\bx)&\eqdef \eps^{-2}\sum_{x\in \Omega_\eps}\sum_{\substack{y\in\Omega_\eps\\ y\sim x}}\sum_{i=1}^k \car_{x_i}(x)\,\car_{\Omega_\eps}[y|\hat\bx_i]\ttonde{f(\bx_i^y)-f(\bx)}\\
		&\qquad+\eps^{\beta-2}\sum_{x\in\partial\Omega_\eps} \sum_{\substack{z\in\partial_e\Omega_\eps\\z \sim x}}\axz\sum_{i=1}^k \car_{x_i}(x)\ttonde{f(\bx_i^z)-f(\bx)}\fstop
	\end{aligned}
\end{equation}
We further let $\ttonde{P^{\eps,\beta,k}_t}_{t\ge 0}$ be the corresponding Markov semigroup on $\R^{(\overline\Omega_\eps)^k}$, with corresponding heat kernel
\begin{equation}\label{eq:HK-k}
	p^{\eps,\beta,k}_t(\bx,\by)\coloneqq \mbfP_\bx^{\eps,\beta,k}\tonde{\mbfX^{\eps,\beta,k}_t=\by}\fstop
\end{equation}
\begin{remark}[Accessible configurations]
	When $\sigma=-1$, configurations $\bx \in (\overline\Omega_\eps)^k$ for which $x_i=x_j\in \Omega_\eps$ for some $i,j \in \set{1,\ldots, k}$, $i\neq j$, are inaccessible since $\car_{\Omega_\eps}[\bx]=0$, hence they  could be discarded; when $\sigma = 1$, all configurations are accessible. In order to keep track of possible restrictions,  we let $\overline\Omega_\eps^k\eqdef \overline\Omega_\eps^{k,\sigma}$, resp.\ $\Omega_\eps^k\eqdef \Omega_\eps^{k,\sigma}$, denote the subset of accessible configurations in~$(\overline\Omega_\eps)^k$, resp.\ $(\Omega_\eps)^k$. 
\end{remark}

Note that  $A^{\eps,\beta,k}$ with $k=1$ coincides with the generator $A^{\eps,\beta}$ in \eqref{eq:generator_RW}, while for $k\ge 2$ the process $\mbfX_t^{\eps,\beta,k}$ describes the position of $k$ (labeled) particles diffusively evolving and interacting on $\Omega_\eps$, until eventually ---~and independently~--- absorbed in $\partial_e\Omega_\eps$ at rates proportional to~$\eps^{\beta-2}$. When $\sigma=-1$, resp.\ $\sigma=1$, particles undergo the exclusion, resp.\ inclusion, interaction rule.  Moreover, since this Markovian dynamics  does not depend on the particles' labels but only on their positions, the  projection of $P^{\eps,\beta,k}_t$ onto symmetric functions~$f \in \R^{\overline\Omega_\eps^k}_\sy$   again corresponds to a Markov process.

For every $k\in \N$, $A^{\eps,\beta,k}$ and $P^{\eps,\beta,k}_t$ globally fix the space of all functions identically vanishing on~$\partial_e\Omega_\eps^k\eqdef \overline\Omega_\eps^k\setminus \Omega_\eps^k$. As already done for the case $k=1$, we identify the latter space with~$L^p(\Omega_\eps^k)$ for any $p\in [1,\infty]$, endowed with the (weighted) norm
\begin{equation*}
	\norm{f}^p_{L^p(\Omega_\eps^k)}\eqdef \eps^{kd}\sum_{\bx\in 		\Omega_\eps^k} \abs{f(\bx)}^p\car_{\Omega_\eps}[\bx]\comma\quad p \in [1,\infty)\comma\qquad \norm{f}_{L^\infty(\Omega_\eps^k)}\eqdef \sup_{\bx\in\Omega_\eps^k} \abs{f(\bx)}\fstop
\end{equation*}
Further note that $A^{\eps,\beta,k}$ and the corresponding semigroup are self-adjoint in $L^2(\Omega_\eps^k)$.

\subsubsection{Duality and consistency}
For $\eta \in \Xi^\eps$, we define inductively on $k\in \N$ the functions on $\overline\Omega_\eps^k$ 
\begin{align}
\label{eq:duality_function}			
\begin{aligned}
		D(y,\eta) &\eqdef 
	\begin{cases}
		\eta(y)  &\text{if} \ y \in \Omega_\eps \\
		\vartheta(y)  &\text{if} \ y \in \partial_e\Omega_\eps  
	\end{cases} \comma  \quad \text{and for}\ \   \bx\in \Omega_\eps^k\comma \\ 
	D(\bx\:y,\eta)&\eqdef 
	D(\bx,\eta)\times 
	\begin{cases}
		\displaystyle\frac{\eta(y)-\sum_{i=1}^k\car_{x_i}(y)}{\car_{\Omega_\eps}[y|\bx]}&\text{if}\ y \in \Omega_\eps\\
		\vartheta(y) &\text{if} \ y \in \partial_e\Omega_\eps 
	\end{cases} \fstop 
\end{aligned}
\end{align}
For each $\sigma \in \set{-1,1}$, such functions serve as \emph{duality functions} between the processes $\mbfX^{\eps,\beta,k}_t$ and the corresponding particle system $\eta^{\eps,\beta}_t$, viz.,
\begin{equation}\label{eq:duality}
	\E_\eta^{\eps,\beta}\quadre{D\ttonde{\bx,\eta^{\eps,\beta}_t}}=P^{\eps,\beta,k}_t D(\emparg,\eta)(\bx)\comma\qquad \eta \in \Xi^\eps\comma \bx \in 	\overline\Omega_\eps^k\comma t \ge0\fstop
\end{equation}
In view of this relation between $\eta^{\eps,\beta}_t$ and~$\mbfX_t^{\eps,\beta,k}$, the latter processes are usually referred to as the \emph{dual processes}  \cite{carinci_duality_2013-1,floreani_boundary2020}.

The following \textquoteleft consistency\textquoteright\ property for the corresponding unlabeled particle systems is well-known (see, e.g., \cite{carinci_consistent_2019, floreani_boundary2020} for a proof).

\begin{proposition}[Consistency, \cite{carinci_consistent_2019}]
For every $k\in \N$, let
$J^{\eps,k}: \R^{\overline\Omega_\eps^k}\to \R^{\overline\Omega_\eps^{k+1}}$ be the  annihilation operator
\begin{equation}\label{eq:annihilation}
	J^{\eps,k}f(\bx)\eqdef \sum_{i=1}^{k+1} f(\hat\bx_i)\comma\qquad  f\in \R^{\overline\Omega_\eps^k}\comma \bx\in \overline\Omega_\eps^{k+1}\fstop
\end{equation} 
Then, \begin{equation}\label{eq:consistency}
	P_t^{\eps,\beta,k+1}J^{\eps,k} f = J^{\eps,k} P_t^{\eps,\beta,k}f\comma\qquad f \in \R^{\overline\Omega_\eps^k}_\sy\comma  t \ge0\fstop
\end{equation} 
\end{proposition}

\subsubsection{Ultracontractivity  and moment estimates}\label{sss:ultracontractivity}
In this section, we prove  that the  semigroups associated to finitely-many SIP dual particles on lattice approximations of bounded Lipschitz domains are uniformly ultracontractive (Prop.~\ref{cor:nash-k}); this result is of independent interest (see Rmk.~\ref{rmk:nash-k-SEP} below for the SEP analogue), and yields effective estimates on the  moments of the particle systems $\eta^{\eps,\beta}_t$ (Cor.~\ref{pr:nash}). 

	\begin{proposition}[Ultracontractivity for the $k$-particle  semigroup, $\sigma=1$]\label{cor:nash-k}
	For every~$k \in \N$, there exists $C=C(\Omega,d,k)>0$ such that
	\begin{equation}\label{eq:nash-k}
		\tnorm{P^{\eps,k}_t f}_{L^\infty(\Omega_\eps^k)}\le C\ttonde{1+t^{-kd/2}}\norm{f}_{L^1(\Omega_\eps^k)}\comma \qquad f \in \R^{\Omega_\eps^k}\comma t >0\fstop
	\end{equation}	
\end{proposition}
The proof of Proposition~\ref{cor:nash-k} goes through establishing a Nash inequality  for the dual SIP $\mbfX^{\eps,k}_t$. We achieve this by comparing~$P^{\eps,k}_t$ with $(P^\eps_t)^{\otimes k}$, that is, the semigroup corresponding to a system of $k$ \emph{independent} copies of the random walk $X^\eps_t$. This is the content of the following lemma.

\begin{lemma}[Comparison of norms and Dirichlet forms, $\sigma=1$]\label{lem:comparison_dirichlet_forms}
	For every $k \in \N$,	
	\begin{equation}\label{eq:comparison_norms}
		\norm{f}_{L^p(\Omega_\eps)^{\otimes k}}\le \norm{f}_{L^p(\Omega_\eps^k)}\le (k!)^{1/p} \norm{f}_{L^p(\Omega_\eps)^{\otimes k}}\comma \qquad f \in \R^{\Omega_\eps^k}\comma p \in [1,\infty)\comma
	\end{equation}
	and
	\begin{equation}\label{eq:comparison_dirichlet_forms}
		\mcE^{\eps,k}_\otimes(f)\le 	\mcE^{\eps,k}(f)\comma\qquad f \in \R^{\Omega_\eps^k}\comma
	\end{equation}	
	where $\mcE^{\eps,k}_\otimes$, resp.\ $\mcE^{\eps,k}$, denotes the Dirichlet form associated to $(P^\eps_t)^{\otimes k}$, resp.\ $P^{\eps,k}_t$.
	\begin{proof}
		The inequality \eqref{eq:comparison_norms} follows at once from the definition of the spaces $L^p(\Omega_\eps^k)$ and the fact that, letting~$\car_{\Omega_\eps}[\bx]$ be as in~\eqref{eq:omega-bx} with~$\car_{\Omega_\eps}$ in place of~$\omega$,
		\begin{equation}\label{eq:713}
			1 \le \car_{\Omega_\eps}[\bx]\le k!\comma\qquad \bx \in \Omega_\eps^k\fstop
		\end{equation}
	\purple{The claim in \eqref{eq:comparison_dirichlet_forms} follows by the first inequality in \eqref{eq:713} and a straightforward comparison of the jump rates $r^{\eps,k}$ and $r^{\eps,k}_\otimes$  of $k$ inclusion and $k$ independent particles, respectively: for all $\bx \in \Omega_\eps^k$, $y \in \overline\Omega_\eps$, and $i=1,\ldots, k$ (cf.\ \eqref{eq:generator_dual}), 
	\begin{align*}
		r^{\eps,k}(\bx,\bx^y_i)&\eqdef  \sum_{x\in \Omega_\eps}\car_{x_i}(x)\tonde{	\eps^{-2}\,\car_{\Omega_\eps}(y)\,\car_{x\sim y}\tonde{1+\car_{\Omega_\eps}[y|\hat \bx_i]}+
		\eps^{\beta-2}\, \car_{\partial_e\Omega_\eps}(y)\, \alpha_\eps^{xy}}
	 \\
	 &\ge \sum_{x\in \Omega_\eps}\car_{x_i}(x)\tonde{	\eps^{-2}\,\car_{\Omega_\eps}(y)\,\car_{x\sim y}+
	 	\eps^{\beta-2}\, \car_{\partial_e\Omega_\eps}(y)\, \alpha_\eps^{xy}}
		\defeq r^{\eps, k}_\otimes(\bx,\bx^y_i)\fstop
	\end{align*}}
This concludes the proof of the lemma.
	\end{proof}
	\begin{proof}[Proof of Proposition~\ref{cor:nash-k}]
		For~$i=1,\ldots,5$ let $C_i=C_i(\Omega,d,k)>0$.
		By  tensorization of \eqref{eq:ultracontractivity},
		\begin{equation*}
			\tnorm{(P^\eps_t)^{\otimes k} f}_{L^\infty(\Omega_\eps)^{\otimes k}}\le C_1 \ttonde{1+t^{-kd/2}} \norm{f}_{L^1(\Omega_\eps)^{\otimes k}}\ ,\qquad f \in \R^{\Omega_\eps^k}\comma t >0\fstop
		\end{equation*}
		By self-adjointness of~$(P^\eps_t)^{\otimes k}_t$ in~$(L^2(\Omega_\eps))^{\otimes k}$, the latter inequality is equivalent  to (see, e.g., \cite[Thm.~2.3.7]{saloff1997lectures})
		\begin{equation*}
			\norm{f}_{L^2(\Omega_\eps)^{\otimes k}}^{2(1+1/kd)}\le C_2 \tonde{\mcE^{\eps,k}_\otimes(f)+ C_3 \norm{f}_{L^2(\Omega_\eps)^{\otimes k}}^2} \norm{f}^{4/kd}_{L^1(\Omega_\eps)^{\otimes k}}\comma\qquad f \in \R^{\Omega_\eps^k}\fstop
		\end{equation*}
		By  comparison, Lemma \ref{lem:comparison_dirichlet_forms} yields the following Nash inequality
		\begin{equation*}
			\norm{f}_{L^2(\Omega_\eps^k)}^{2(1+1/kd)}\le C_4 \tonde{\mcE^{\eps,k}(f)+ C_5 \norm{f}_{L^2(\Omega_\eps^k)}^2} \norm{f}^{4/kd}_{L^1(\Omega_\eps^k)}\comma\qquad f \in \R^{\Omega_\eps^k}\comma
		\end{equation*}
		which implies the desired claim (see, e.g., \cite[Thm.~2.3.4]{saloff1997lectures}).
	\end{proof}
\end{lemma}
\begin{remark}[Ultracontractivity for the $k$-particle semigroup, $\sigma=-1$]\label{rmk:nash-k-SEP}
	As  observed in, e.g., \cite[Rmk.~2]{landim_decay1998}, the claim in Proposition~\ref{cor:nash-k} holds also for SEP.	Indeed, recalling \eqref{eq:duality_function},  Liggett's comparison inequality \cite[Prop.~VIII.1.7]{liggett_interacting_2005-1} implies
	\begin{equation}
		P^{\eps,k}_t D(\emparg ,\eta)(\bx)\le (P^\eps_t)^{\otimes k} D(\emparg ,\eta)(\bx)\comma\quad t\ge 0\comma\bx\in \overline\Omega_\eps^k\comma \eta\in \Xi^\eps\comma\vartheta\in [0,1]^{\partial_e\Omega_\eps}\fstop
	\end{equation}
	As a consequence,  letting $\mfS_k$ denote the symmetric group of degree $k$ (recall \eqref{eq:HK-k}), 
	\begin{equation}
		p^{\eps,k}_t(\bx,\by)
		\le \sum_{\varsigma\in \mfS_k} p^\eps_t(x_1,y_{\varsigma(1)})\cdots p^\eps_t(x_k,y_{\varsigma(k)})\comma\qquad t \ge 0\comma \bx, \by \in \overline\Omega_\eps^k\fstop
	\end{equation}
	By combining this with \eqref{eq:ultracontractivity}, the desired claim follows.
\end{remark}

We conclude this section by deriving useful moment estimates for the particle systems $\eta^{\eps,\beta}_t$.	In what follows, unless specified otherwise, $\sigma=\pm 1$, and  $\ttseq{\nu_\eps}_\eps$ represents a family of generic probability distributions on $\ttseq{\Xi^\eps}_\eps$. Moreover,  $\beta\in \R$ is fixed and suppressed from the notation, while $\E^\eps_{\nu_\eps}$ denotes expectation with respect to the law of	 $\eta^\eps_t$ with $\eta^\eps_0$ distributed as $\nu_\eps$.

\begin{corollary}\label{pr:nash}
	For every $k \in \N$, there exists $C=C(\Omega,d,k,\vartheta)>0$ such that
	\begin{equation}\label{eq:decay-kth-moment-sip}
		\sup_{\bx \in \Omega_\eps^k}	\E^\eps_{\nu_\eps}\quadre{\prod_{i=1}^k \eta^\eps_t(x_i)}\le C
		(1+t^{-kd/2})\tonde{1+\E_{\nu_\eps}\quadre{\norm{\eta}^k_{L^1(\Omega_\eps)}}}\comma
	\end{equation}
	\begin{equation}
		\label{eq:bound-l1}
		\sup_{s\ge 0}\E^\eps_{\nu_\eps}\quadre{\tnorm{\eta^\eps_s}_{L^1(\Omega_\eps)}^k}\le C\tonde{1+\E_{\nu_\eps}\quadre{\norm{\eta}_{L^1(\Omega_\eps)}^k}}\comma	\end{equation}
	and
	\begin{equation}\label{eq:max_principle}
		\sup_{s\ge 0}\sup_{\bx\in \Omega_\eps^k} \E^\eps_{\nu_\eps}\quadre{\prod_{i=1}^k \eta^\eps_s(x_i)}\le C\tonde{1+\sup_{\bx\in \Omega_\eps^k}\E_{\nu_\eps}\quadre{\prod_{i=1}^k\eta(x_i)}}
	\end{equation}
	hold true for all $t>0$ and $\ttseq{\nu_\eps}_\eps$.
\end{corollary}
\begin{proof}
All claims are trivial for SEP ($\sigma=-1$)	due to the maximal occupancy of one particle per site; hence,  fix $\sigma=1$ all throughout this proof.

	We start by proving \begin{equation}\label{eq:nash-D}
		\sup_{\bx \in \Omega_\eps^k}	P^{\eps,k}_t D(\emparg,\eta)(\bx)\le C
		\ttonde{1+t^{-kd/2}}\ttonde{1+\norm{\eta}^k_{L^1(\Omega_\eps)}}\comma\qquad t >0\comma \eta \in \Xi^\eps\comma
	\end{equation} for every $k\in \N$, from which the claim in \eqref{eq:decay-kth-moment-sip} follows (up to redefining the constants).  Indeed, for all $\bx \in \Omega_\eps^k$ and $\eta \in \Xi^\eps$, 
	\begin{equation*}
		\prod_{i=1}^k \eta(x_i)= \sum_{\ell=0}^k \sum_{\substack{\by\in \Omega_\eps^\ell\\
				\by\le \bx}} a(\bx,	\by)\, \car_{\Omega_\eps}[\by]\, D(\by,\eta)\comma
	\end{equation*}
	for some non-negative $a(\bx,\by)\le C_1=C_1(k)$. Thus,  for some $C_2=C_2(k)>0$, by \eqref{eq:duality},
	\begin{align*}
		\sup_{\bx\in \Omega_\eps^k}\E^\eps_{\nu_\eps}\quadre{\prod_{i=1}^k\eta_t^\eps(x_i)}	&\le C_2  \max_{\ell \le k}\sup_{\by \in \Omega_\eps^\ell}  \E^\eps_{\nu_\eps}\quadre{D(\by,\eta^\eps_t)}\\
		&\qquad = C_2  \max_{\ell \le k}\sup_{\by \in \Omega_\eps^\ell}  \int_{\eta\in \Xi^\eps }P^{\eps,\ell}_t D(\emparg,\eta)(\by)\, \nu_\eps(\dd \eta) \fstop
	\end{align*}
	
	Let us prove \eqref{eq:nash-D} by induction on $k\in \N$. For $k=1$, 
	\begin{align*}
		\sup_{x\in \Omega_\eps}P^\eps_t D(\emparg,\eta)(x)
		&\le \sup_{x\in \Omega_\eps} P^\eps_t\ttonde{D(\emparg,\eta)\car_{\Omega_\eps}}(x)+\sup_{x\in \Omega_\eps}P^\eps_t\ttonde{D(\emparg,\eta)\car_{\partial_e\Omega_\eps}}(x)\\
		&\qquad\le C\ttonde{1+t^{-d/2}}\norm{D(\emparg,\eta)}_{L^1(\Omega_\eps)} + \norm{\vartheta}_{\cC_b(\R^d)}\fstop
	\end{align*}
	Here, the last step is a consequence of  \eqref{eq:nash-k} with $k=1$ and $f=D(\emparg,\eta)\car_{\Omega_\eps}\in \R^{\Omega_\eps}$, $\norm{D(\emparg,\eta)\car_{\Omega_\eps}}_{L^1(\Omega_\eps)}=\norm{\eta}_{L^1(\Omega_\eps)}$ and $D(\emparg,\eta)\car_{\partial_e\Omega}\le \norm{\vartheta}_{\cC_b(\R^d)}$.

	Assume now that \eqref{eq:nash-D} holds for $k\in \N$; then,
	\begin{align*}
		\tnorm{
			P^{\eps,k+1}_t D(\emparg,\eta)}_{L^\infty(\Omega_\eps^{k+1})} &\le \tnorm{P^{\eps,k+1}_t\ttonde{D(\emparg,\eta)\car_{\Omega_\eps^{k+1}}}}_{L^\infty(\Omega_\eps^{k+1})}\\
		&\qquad	 +	\tnorm{P^{\eps,k+1}_t\ttonde{D(\emparg,\eta)\car_{\partial_e\Omega_\eps^{k+1}}}}_{L^\infty(\Omega_\eps^{k+1})}\comma\qquad t >0\fstop
	\end{align*}
	For the first term on the right-hand side above, we apply \eqref{eq:nash-k} with $f=D(\emparg,\eta)\car_{\Omega_\eps^{k+1}}\in \R^{\Omega_\eps^{k+1}}$ and use that $\tnorm{D(\emparg,\eta)\car_{\Omega_\eps^{k+1}}}_{L^1(\Omega_\eps^{k+1})}\le \norm{\eta}_{L^1(\Omega_\eps)}^{k+1}$ to obtain that
	\begin{equation*}
		\tnorm{P^{\eps,k+1}\ttonde{D(\emparg,\eta)\car_{\Omega_\eps^{k+1}}	}}_{L^\infty(\Omega_\eps^{k+1})}\le C\ttonde{1+t^{(k+1)d/2}}\ttonde{1+\norm{\eta}_{L^1(\Omega_\eps)}^{k+1}}\comma \qquad t>0\fstop
	\end{equation*}
	For the second term, recall \eqref{eq:consistency}; then, 	for all $\bx \in \Omega_\eps^{k+1}$, 
	\begin{equation*}
		D(\bx,\eta)\car_{\partial_e\Omega_\eps^k}(\bx)=\sum_{i=1}^{k+1}  D(\hat \bx_i,\eta)\, \vartheta(x_i)\, \car_{\partial_e\Omega_\eps}(x_i)\le \norm{\vartheta}_{\cC_b(\R^d)} J^{\eps,k}D(\emparg,\eta)(\bx)\fstop
	\end{equation*}
	Since $D(\emparg,\eta)$ are symmetric functions of the particles' labels, \eqref{eq:consistency} yields 
	\begin{align*}
		&\tnorm{P^{\eps,k+1}\ttonde{D(\emparg,\eta)\car_{\partial_e\Omega_\eps^{k+1}}}}_{L^\infty(\Omega_\eps^{k+1})}
		\\
		&\qquad\le \norm{\vartheta}_{\cC_b(\R^d)} \tnorm{P^{\eps,k+1}_t J^{\eps,k} D(\emparg,\eta)}_{L^\infty(\Omega_\eps^{k+1})}\\
		&\qquad = \norm{\vartheta}_{\cC_b(\R^d)} \tnorm{J^{\eps,k} P^{\eps,k}_t D(\emparg,\eta)}_{L^\infty(\Omega_\eps^{k+1})}\\
		&\qquad \le \norm{\vartheta}_{\cC_b(\R^d)}(k+1) \tnorm{P^{\eps,k}_t D(\emparg,\eta)}_{L^\infty(\Omega_\eps^k)}\comma \qquad t>0\fstop
	\end{align*}
	By  the induction hypothesis on  $t\mapsto \tnorm{P^{\eps,k}_t D(\emparg,\eta)}_{L^\infty(\Omega_\eps^k)}$,  rearranging all constants (which depend only on $\Omega, d, k+1$ and $\vartheta$) concludes the proof of the claim in \eqref{eq:decay-kth-moment-sip}.
	
	The claim in \eqref{eq:bound-l1} follows by an analogous induction argument employing \eqref{eq:consistency}, while the claim in \eqref{eq:max_principle} is an immediate consequence of duality and the Maximum Principle for the semigroups $P^{\eps,\ell}_t$ for $\ell \le k$.
\end{proof}

\subsection{Proof of Theorem \ref{t:MainHydrodynLim}}\label{sec:proof-HDL}
In what follows,  $\beta \in \R$ and $T>0$ are fixed and,  for notational convenience, we omit the specification of $\beta\in \R$;  the evaluation operator~$\Pi_\eps$ introduced in \S\ref{sss:FCLT} is omitted as well \purple{whenever no confusion may arise}.
Recall that, for  $\ttseq{\nu_\eps}_\eps$ given as in Theorem~\ref{t:MainHydrodynLim}, $\P^\eps_{\nu_\eps}$~and~$\E^\eps_{\nu_\eps}$ denote the laws and corresponding expectations of the particle system~$\eta^\eps_t$ introduced in \S\ref{sec:particle_systems} with $\eta^\eps_0$ distributed as $\nu_\eps$.
Further,  let $\cF^\eps_t=\ttseq{\cF^\eps_t}_{t\ge 0}$ denote the natural filtration associated to the process $\eta^\eps_t$.

In this section we prove that, for all $n\in \N$, $0\leq t_1\leq \ldots \leq t_n $ and $f_1,\ldots, f_n \in \cC^\bd$, 
	\begin{equation}
		\label{eq:HDL-fdd}
		\lim_{\eps\downarrow0}\P^\eps_{\nu_\eps}\tonde{
			\abs{
				\ttonde{
					\ttscalar{\cX^\eps_{t_1}}{f_1}
					,\dotsc,
					\ttscalar{\cX^\eps_{t_n}}{f_n}	
				}
				-\ttonde{
					\ttscalar{u_{t_1}}{f_1},
					\dotsc,\ttscalar{u_{t_n}}{f_n}
				}
			}	
			\ge\delta
		}=0\comma\quad \delta>0\fstop
	\end{equation}
Key facts in our proof are the convergences in Theorems~\ref{t:MainSemigroups} and~\ref{th:harmonic_conv}, as well as the duality relations in \S\ref{s:duality}.
In order to illustrate this, note that  \eqref{eq:duality_function} and \eqref{eq:duality} with $k=1$ yield	
\begin{equation*}
	\E^\eps_{\nu_\eps}\quadre{\eta^\eps_t(x)|\cF^\eps_s}= P^\eps_{t-s}D(\emparg,\eta^\eps_s)(x)\qquad\text{and}\qquad \E^\eps_{\nu^\eps_\stat}\quadre{\eta^\eps_t(x)}=\lim_{r\to \infty}\E^\eps_{\nu_\eps}\quadre{\eta^\eps_r(x)}= h^\eps(x)\comma
\end{equation*}
for all $0 \le s \le t$ and $x \in \Omega_\eps$.
Since $P^\eps_{t-s} h^\eps = h^\eps$ and
\begin{equation*}
	\eta^\eps_t(x)= \ttonde{\eta^\eps_t(x)-\E^\eps_{\nu_\eps}\quadre{\eta^\eps_t(x)|\cF^\eps_s}} + P^\eps_{t-s}\ttonde{D(\emparg,\eta^\eps_s)-h^\eps}(x) + h^\eps(x)\comma
\end{equation*}
we have a decomposition of  the corresponding  fields into three terms:
\begin{equation}\label{eq:3-term-decomposition}
	\begin{aligned}
		\scalar{\cX^\eps_t}{\purple{\Pi_\eps} f}&= \scalar{\cX^\eps_t-\E^\eps_{\nu_\eps}\quadre{\cX^\eps_t|\cF^\eps_s}}{\purple{\Pi_\eps}  f}\\
		&\qquad+  	\scalar{\cX^\eps_s-h^\eps\mu_\eps}{P^\eps_{t-s}\purple{\Pi_\eps} f} + 	\scalar{h^\eps\mu_\eps}{\purple{\Pi_\eps} f}\comma\qquad 0\le s \le t\comma f \in \cC^\bd\fstop
	\end{aligned}
\end{equation}
We stress that for the rewriting of the second term on the right-hand side above we crucially exploited the symmetry of $P^\eps_{t-s}$ in $L^2(\Omega_\eps)$ and the fact that  both functions \mbox{$D(\emparg,\eta)-h^\eps$} and~$\purple{\Pi_\eps} f$ on~$\overline\Omega_\eps$ identically vanish on $\partial_e\Omega_\eps$ \purple{(the second one by definition of~$\Pi_\eps$)}.

Setting $s=0$ in \eqref{eq:3-term-decomposition}, we get, for all $t \ge 0$ and $f \in \cC^\bd$,  convergence  in probability of the real-valued random variables $\ttseq{\scalar{\cX^\eps_t}{f}}_\eps$ to $$\tscalar{u_t}{f}=\tscalar{\pi_0-h^\bd\mu_\Omega}{P^\bd_t f}+\tscalar{h^\bd\mu_\Omega}{f}$$
through the following steps:
\begin{itemize}
	\item by Theorems \ref{t:MainSemigroups} and \ref{th:harmonic_conv}, $\tscalar{h^\eps\mu_\eps}{P^\eps_t\Peps f}\to \tscalar{h^\bd\mu_\Omega}{P^\bd_t f}$ and $\tscalar{h^\eps\mu_\eps}{f}\to \tscalar{h^\bd \mu_\Omega}{f}$; 
	\item by Markov inequality, the first estimate in Assumption	 \ref{as:secondII}, and Theorem \ref{t:MainSemigroups}, 
	\begin{equation*}
		\lim_{\eps\downarrow 0}\P^\eps_{\nu_\eps}\tonde{\tabs{\tscalar{\cX^\eps_0}{P^\eps_t\Peps f-\Peps P^\bd_t f}}\ge\delta}=0\comma\qquad  \delta>0\semicolon
	\end{equation*} 
	\item since $P^\bd_t:\cC^\bd\to \cC^\bd$, Assumption \ref{as:WLLN} on $\ttseq{\nu_\eps}_\eps$ ensures that
	\begin{equation*}
		\lim_{\eps\downarrow 0}\P^\eps_{\nu_\eps}\tonde{\tabs{\tscalar{\cX^\eps_0}{\Peps P^\bd_t f}-\tscalar{\pi_0}{P^\bd_t f}}\ge\delta}=0\comma\qquad \delta>0\semicolon
	\end{equation*}
	\item   Assumption \ref{as:secondII} and Lemma \ref{cor:variance-ub} (cf.\ \S\ref{sss:proof-variance} below) yield
	\begin{equation}\label{eq:variance_vanishes}
		\lim_{\eps\downarrow 0}\E^\eps_{\nu_\eps}\quadre{\tonde{\tscalar{\cX^\eps_t-\E^\eps_{\nu_\eps}\quadre{\cX^\eps_t|\cF^\eps_0}}{f}}^2}=0\fstop
	\end{equation}
\end{itemize}
Since convergence in probability of  marginals implies convergence in probability of finite-dimensional distributions, this would prove~\eqref{eq:HDL-fdd}, thus, Theorem~\ref{t:MainHydrodynLim}.

	\subsubsection{Proof of~\eqref{eq:variance_vanishes}}\label{sss:proof-variance}
	In this section we prove~\eqref{eq:variance_vanishes} in two steps (Lems.~\ref{pr:variance} and~\ref{cor:variance-ub}), thus, concluding the proof of Theorem~\ref{t:MainHydrodynLim}. 
	
	First, recall the definition of the infinitesimal generator~$\cL^\eps$ in~\eqref{eq:gen} and that of duality functions in \eqref{eq:duality_function}. Define, for all $\eta\in \Xi^\eps, 
	$\begin{equation}\label{eq:Veps}
		\cV^\eps\ttonde{(x,y),\eta}\eqdef D(x,\eta)+D(y,\eta)+2\sigma D\ttonde{(x,y),\eta}\comma\qquad (x,y) \in \overline\Omega_\eps^{k=2} \fstop
	\end{equation}
	The functions in \eqref{eq:Veps} show up when computing the second moments of the empirical density fields. Indeed, by expanding
	\[
	\mcL^\eps(\scalar{\mcX^\eps_t}{f})^2-2\,\scalar{\mcX^\eps_t}{f}\,\mcL^\eps \scalar{\mcX^\eps_t}{f} \comma
	\] 
	a simple manipulation yields, for all $\eta\in\Xi^\eps$ and $f \in \cC^\bd$,
	\begin{equation}\label{eq:variance}
		\E^\eps_\eta\quadre{\tonde{\scalar{\cX^\eps_t-\E^\eps_\eta\quadre{\cX^\eps_t}}{f}}^2} = \int_0^t \cG^\eps_{r,t}(f,\eta)\,\dd r\comma\qquad t \ge 0\comma
	\end{equation}
	where
	\begin{equation*}
		\begin{aligned}
			\cG^\eps_{r,t}(f,\eta)&\eqdef  \frac{\eps^d}{2} \sum_{\substack{x,y\in \Omega_\eps\\x\sim y}} \tonde{\nabla_{x,y}^\eps P^\eps_{t-r}\Peps f }^2 \ttonde{\eps^d\,\E^\eps_\eta\quadre{\cV^\eps((x,y),\eta^\eps_r)}}\\
			&\quad 	+ \eps^{d-1}\sum_{x\in \partial\Omega_\eps} \eps^{\beta-1}  \tonde{P^\eps_{t-r}\Peps f(x)}^2 \sum_{\substack{z\in \partial_e\Omega_\eps\\ z\sim x}}\axz \ttonde{\eps^d\,\E^\eps_\eta\quadre{\cV^\eps((x,z),\eta^\eps_r)}}	\fstop
		\end{aligned}
	\end{equation*}
Note that, when $\sigma = -1$, the deterministic upper bound $\cV^\eps\le 2$ holds, while no such a bound exists when $\sigma =1$.

\begin{lemma}\label{pr:variance}
	Let $\ttseq{\nu_\eps}_\eps$ be a family of probability measures on $\ttseq{\Xi_\eps}_\eps$. Then, for every $f \in \cC^\bd$, there exists $C=C(\Omega,\vartheta,f)>0$ such that
	\begin{equation}\label{eq:claim1}
		\begin{aligned}
			&\E^\eps_{\nu_\eps}\quadre{\tonde{\scalar{\cX^\eps_t-\E^\eps_{\nu_\eps}\quadre{\cX^\eps_t\vert \cF^\eps_s}}{f}}^2}\\
			&\quad\le C\tonde{\purple{\phi^{\eps,k=2}_{s,r}} \tonde{	\norm{\Peps f-P^\eps_{r-s}\Peps f}_{L^\infty(\Omega_\eps)}	 + \purple{\phi^{\eps,k=2}_{r,t}}}}  \end{aligned}
	\end{equation}
hold for all $0\le s \le r \le  t$, where  $\phi^{\eps,k}_{u,u}\eqdef 0$ and
	\begin{equation}\label{eq:phi}
		\phi^{\eps,k}_{u,v}	\eqdef \eps^{kd/2}\tonde{1+ \sup_{\tau\in [u,v]}\sup_{\bx\in \Omega_\eps^k}  \E^\eps_{\nu_\eps}\quadre{\prod_{i=1}^k\eta^\eps_\tau(x_i)}}\comma \qquad 0 \le u <v  \fstop
	\end{equation}\end{lemma}
	\begin{proof}
		Splitting the time integrals in \eqref{eq:variance} and recalling the definition of the Dirichlet form in \eqref{eq:dirichlet-form-RW}, H\"older inequality  yields 
		\begin{equation}\label{eq:int_chi}
			\begin{aligned}
				&\E^\eps_{\nu_\eps}\quadre{\tonde{\scalar{\cX^\eps_t-\E^\eps_{\nu_\eps}\quadre{\cX^\eps_t\vert \cF^\eps_s}}{f}}^2}\\
				&\quad\le  \xi^{\eps,k=2}_{s,r} \int_0^{r-s} \mcE^\eps\ttonde{P^\eps_{t-s-u}\Peps f}\dd u	+	\xi^{\eps,k=2}_{r,t}\int_{r-s}^{t-s} \mcE^\eps\ttonde{P^\eps_{t-s-u}\Peps f}\dd u\comma
			\end{aligned}
		\end{equation}
		where
		\begin{equation*}
			\xi^{\eps,k=2}_{u,v}\eqdef \eps^d	\sup_{\tau\in [u,v]}\sup_{\bx\in \overline\Omega_\eps^{k=2}}  \E^\eps_{\nu_\eps}\quadre{\cV^\eps(\bx,\eta^\eps_\tau)}\fstop
		\end{equation*}
		\purple{Confronting with $\phi^{\eps,k=2}_{u,v}$ given in \eqref{eq:phi}, we have (up to a universal constant depending only on $\Omega$ and $\vartheta$) 
		\begin{equation}
			\xi^{\eps,k=2}_{u,v}\lesssim\phi^{\eps,k=2}_{u,v}\comma\qquad 0\le u<v\fstop
		\end{equation}}
		Integrating over time the terms in the right-hand side of~\eqref{eq:int_chi},
					we get the claim in~\eqref{eq:claim1}, thus, concluding the proof of the proposition.
		\end{proof}
	
		\begin{remark}[Properties of $\phi^{\eps,k}_{}$ in \eqref{eq:phi}]\label{rem:phi}
		If the family~$\ttseq{\nu_\eps}_\eps$ satisfies Assumption~\ref{as:secondII}, \purple{Corollary}~\ref{pr:nash} ensures that, for all  $t>0$ and $r\in (0,t)$,
		\begin{equation}
			\sup_{\eps} \sup_{0\le u\le v }	 \phi^{\eps,k=2}_{u,v}<\infty\qquad \text{and}\qquad
			\lim_{\eps\downarrow 0} \phi^{\eps,k=2}_{r,t}=0\fstop	
		\end{equation}
	\end{remark}

	By combining the above result, the following triangle inequality
	\begin{equation}\label{eq:triangle-semigroups}
		\tnorm{\Peps f-P^\eps_{r-s}\Peps f}_{L^\infty(\Omega_\eps)}\le \tnorm{f-P^\bd_{r-s}f}_{\cC^\bd}+\sup_{\tau\in [0,\purple{t}]}\tnorm{P^\eps_\tau \Peps f-\Peps P^\bd_\tau f}_{L^\infty(\Omega_\eps)}\comma
	\end{equation}
	and the convergence of semigroups in Theorem \ref{t:MainSemigroups}, we derive the following  lemma.
	\begin{lemma}\label{cor:variance-ub}
		Let $\ttseq{\nu_\eps}_\eps$ satisfy Assumption~\ref{as:secondII}. Then \eqref{eq:variance_vanishes} holds for all $t > 0$ and~$f \in \cC^\bd$.  
		\begin{proof}
			Fix $t>0$ and $f\in \cC^\bd$.
						Then,  by Remark \ref{rem:phi},  \eqref{eq:claim1} in Lemma \ref{pr:variance} with  $s=0$ and $r\in (0,t)$, \eqref{eq:triangle-semigroups} and Theorem \ref{t:MainSemigroups}, 
			\begin{equation}
				\lim_{\eps\downarrow 0}\E^\eps_{\nu_\eps}\quadre{\tonde{\tscalar{\cX^\eps_t-\E^\eps_{\nu_\eps}\quadre{\cX^\eps_t|\cF^\eps_0}}{f}}^2}\le C  \tonde{\sup_\eps \sup_{u\ge 0}\phi^{\eps,k=2}_{0,u} }\tnorm{f-P^\bd_r f}_{\cC^\bd}\fstop
			\end{equation}
			By the strong continuity of the semigroup $P^\bd_t$ on $\cC^\bd$, letting~$r\to 0$  concludes the proof.
		\end{proof}
	\end{lemma}

	\subsection{Proof of Theorem \ref{th:hydrostatic}}\label{sec:hydrostatic}
	This section is devoted to the proof of the hydrostatic limit. 
	
	We recall that, by definition of $\nu^{\eps,\beta}_\stat$ and duality (see \S\ref{sec:ness} and \eqref{eq:duality}, respectively), 
	\begin{equation}
		h^{\eps,\beta}(x)=\E_{\nu^{\eps,\beta}_\stat}\quadre{\eta(x)}\comma\qquad x \in \Omega_\eps\fstop
	\end{equation} 
	Thus, by the triangle inequality, the convergence of the discrete harmonic profiles in Theorem~\ref{th:harmonic_conv} and Markov's inequality, Theorem \ref{th:hydrostatic} follows from
	\begin{equation}\label{eq:hydrostatic_L2}
		\lim_{\eps\downarrow 0} \E_{\nu^{\eps,\beta}_\stat}\quadre{\tonde{\eps^d \sum_{x\in \Omega_\eps}\tonde{\eta(x)-h^{\eps,\beta}(x)}f(x)}^2}=0\fstop
	\end{equation}
	The above identity is a consequence of Lemma \ref{lemma:hydrostatic_L2} below.
	\begin{lemma}\label{lemma:hydrostatic_L2}
		There exists $C=C(\vartheta)>0$ such that, for every $f \in \R^{\Omega_\eps}$, 
		\begin{equation*}
			\E_{\nu^{\eps,\beta}_\stat}\quadre{\tonde{\eps^d \sum_{x\in \Omega_\eps}\tonde{\eta(x)-h^{\eps,\beta}(x)}f(x)}^2}\leq C\,\eps^d\, \norm{f}_{L^\infty(\Omega_\eps)}\norm{f}_{L^1(\Omega_\eps)}\fstop
		\end{equation*}	
		\begin{proof}
			Recall the definition of the  semigroups $P^{\eps,\beta,k}_t$ and $\ttonde{P^{\eps,\beta}_t}^{\otimes k}$ from \S\ref{sss:k-particle-dynamics} and \S\ref{sss:ultracontractivity}, respectively.
			\purple{
			Note that 
			\begin{align*}
				\lim_{t\to \infty} P^{\eps,\beta,k=2}_t D(\emparg,\eta)
					=
				\lim_{t\to \infty}
					P^{\eps,\beta,k=2}_t \ttonde{h^{\eps,\beta}\otimes h^{\eps,\beta}}
			\end{align*}
			on $\overline \Omega_\varepsilon^{k=2}$, since  we have $D((x,y),\eta) = \ttonde{h^{\eps,\beta} \otimes h^{\eps,\beta}}(x,y) = \vartheta(x)\vartheta(y)$ for $x,y \in \partial_e \Omega_\eps$.
			}
			By  duality (see \eqref{eq:duality_function}--\eqref{eq:duality} as well as \eqref{eq:omega-bx}) with $k=1$ and $k=2$, and stationarity of $\nu^{\eps,\beta}_\stat$, we obtain that
			\begin{align*}\nonumber
				&\E_{\nu^{\eps,\beta}_\stat}\quadre{\tonde{\eps^d \sum_{x\in \Omega_\eps}\tonde{\eta(x)-h^{\eps,\beta}(x)}f(x)}^2}=	 \eps^d \tonde{\eps^d \sum_{x\in \Omega_\purple{\eps}} h^{\eps,\beta}(x)\ttonde{1+\sigma h^{\eps,\beta}(x)} f(x)^2}\\
								&\qquad + \lim_{t\to \infty} \tscalar{\ttonde{P^{\eps,\beta,k=2}_t -P^{\eps,\beta}_t\otimes P^{\eps,\beta}_t}\ttonde{h^{\eps,\beta}\otimes h^{\eps,\beta}}}{f\otimes f}_{L^2(\Omega_\eps^{k=2})}
				\fstop
							\end{align*}
			Since $\sup_{\eps}\norm{h^{\eps,\beta}}_{L^\infty(\Omega_\eps)}\leq \norm{\vartheta}_{\cC_b(\R^d)}<\infty$, the first term on the right-hand side is smaller than $C\eps^d\norm{f}_{L^2(\Omega_\eps)}^2\le C\eps^d \norm{f}_{L^\infty(\Omega_\eps)}\norm{f}_{L^1(\Omega_\eps)}$, for some $C=C(\vartheta)>0$. As for the second term,  by the integration by parts formula (here:~$A^{\eps,\beta}\oplus A^{\eps,\beta}\eqdef A^{\eps,\beta}\otimes \car + \car \otimes A^{\eps,\beta}$)
			\begin{equation*}
				\begin{aligned}
					&\ttonde{P^{\eps,\beta,k=2}_t -P^{\eps,\beta}_t\otimes P^{\eps,\beta}_t}\ttonde{h^{\eps,\beta}\otimes h^{\eps,\beta}}\\				
					&\quad= \int_0^t P^{\eps,\beta,k=2}_{t-s}\ttonde{A^{\eps,\beta,k=2}-A^{\eps,\beta}\oplus A^{\eps,\beta}}\ttonde{P^{\eps,\beta}_s h^{\eps,\beta}\otimes P^{\eps,\beta}_sh^{\eps,\beta}}\, \dd s
										\comma			\end{aligned}
			\end{equation*}
			and $P^{\eps,\beta}_s h^{\eps,\beta}=h^{\eps,\beta}$, we obtain that
			\begin{equation}\label{eq:Reps}
				\begin{aligned}
					\cR^{\eps,\beta}\eqdef&\	\lim_{t\to \infty} \tscalar{\ttonde{P^{\eps,\beta,k=2}_t -P^{\eps,\beta}_t\otimes P^{\eps,\beta}_t}\ttonde{h^{\eps,\beta}\otimes h^{\eps,\beta}}}{f\otimes f}_{L^2(\Omega_\eps^{k=2})}\\
					=&\ \sigma\eps^d \tonde{ \frac{\eps^d}{2} \sum_{\substack{x,y\in \Omega_\eps\\ x \sim y}}\ttonde{\nabla^\eps_{x,y}h^{\eps,\beta}}^2 \tonde{ \int_0^\infty 2\, P^{\eps,\beta,k=2}_t(f\otimes f)(x,y)\, \dd t}}\fstop
				\end{aligned}
			\end{equation}
			In the above identity, we used the explicit form of $A^{\eps,\beta,k=2}-A^{\eps,\beta}\oplus A^{\eps,\beta}$  (see \eqref{eq:generator_dual}), the fact that both $(A^{\eps,\beta,k=2}-A^{\eps,\beta}\oplus A^{\eps,\beta})(h^{\eps,\beta}\otimes h^{\eps,\beta})$ and $f\otimes f$ identically vanish on $\partial_e\Omega_\eps^{k=2}$ and
			the symmetry of $P^{\eps,\beta,k=2}_{t-s}$ in $L^2(\Omega_\eps^{k=2})$.
			Since $\sigma \in \set{-1,1}$, 
			\begin{equation}\label{eq:171}
				\cR^{\eps,\beta}\le
				\eps^d  \tonde{ \frac{\eps^d}{2} \sum_{\substack{x,y\in \Omega_\eps\\ x \sim y}}\ttonde{\nabla^\eps_{x,y}h^{\eps,\beta}}^2 \tonde{ \int_0^\infty 2 P^{\eps,\beta,k=2}_t(|f|\otimes |f|)(x,y)\, \dd t}}\fstop
			\end{equation}
			Furthermore, for all $x, y \in \overline\Omega_\eps$ and $t\ge0$, 
			\begin{equation}\label{eq:super}
				\begin{aligned}
					2	P^{\eps,\beta,k=2}_t(|f|\otimes |f|)(x,y)\leq&\ \norm{f}_{L^\infty(\Omega_\eps)}P^{\eps,\beta,k=2}_t(|f|\otimes \car_{\overline \Omega_\eps}+\car_{\overline\Omega_\eps}\otimes |f|)(x,y)\\
					=&\ \norm{f}_{L^\infty(\Omega_\eps)}\ttonde{P^{\eps,\beta}_t|f|(x)+P^{\eps,\beta}_t|f|(y)}\comma
				\end{aligned}
			\end{equation}
			where the last identity follows from \eqref{eq:consistency}. Now, define	 \begin{equation}
				g^{\eps,\beta}(x)\eqdef
								\int_0^\infty P^{\eps,\beta}_t|f|(x)\, \dd t \comma\qquad x\in \overline\Omega_\eps\comma
			\end{equation}
			and let	 $\varGamma^{\eps,\beta=\infty}(h^{\eps,\beta})$ denote  the \emph{carr\'e du champ} associated to $A^{\eps,\beta=\infty}$ acting on $h^{\eps,\beta}$ (cf.~\eqref{eq:carre-def}), viz.,	  $\varGamma^{\eps,\beta=\infty}(h^{\eps,\beta})\eqdef 0$ on $\partial_e\Omega_\eps$ and
			\begin{equation}
				\begin{aligned}
					2\varGamma^{\eps,\beta=\infty}(h^{\eps,\beta})(x)\eqdef&\ A^{\eps,+\infty}(h^{\eps,\beta})^2(x)-2 h^{\eps,\beta}(x)A^{\eps,+\infty}h^{\eps,\beta}(x)\\
					=&\ \sum_{\substack{y\in \Omega_\eps\\y\sim x}} \tonde{\nabla^\eps_{x,y}h^{\eps,\beta}}^2\comma
				\end{aligned} \quad x \in \Omega_\eps\fstop
			\end{equation} 
			Hence, the inequality in \eqref{eq:super} yields
			\begin{equation}\label{eq:id2}
				\cR^{\eps,\beta}\le \eps^d  \norm{f}_{L^\infty(\Omega_\eps)} \tonde{\eps^d \sum_{x\in \Omega_\eps} g^{\eps,\beta}(x)\, 2\varGamma^{\eps,\beta=\infty}(h^{\eps,\beta})(x)}\fstop
			\end{equation}
			Since $A^{\eps,\beta}h^{\eps,\beta}\equiv 0$, we have that, for all~$x \in \Omega_\eps$,
			\begin{align*}
				&2\varGamma^{\eps,\beta=\infty}(h^{\eps,\beta})(x) \\
				&\quad= (A^{\eps,\beta=\infty}-A^{\eps,\beta})(h^{\eps,\beta})^2(x)-2 h^{\eps,\beta}(x)\ttonde{A^{\eps,\beta=\infty}-A^{\eps,\beta}} h^{\eps,\beta}(x)+ A^{\eps,\beta}(h^{\eps,\beta})^2(x)\\
				&\quad= -\car_{\partial \Omega_\eps}(x)\, \eps^{\beta-2}\sum_{\substack{z\in \partial_e\Omega_\eps\\ z\sim x}}\axz  \tonde{\vartheta(z)-h^{\eps,\beta}(x)}^2+ A^{\eps,\beta}(h^{\eps,\beta})^2(x) \fstop
			\end{align*}
			Plugging the above identity into  \eqref{eq:id2}, we obtain
			\begin{align*}
				\cR^{\eps,\beta}&\le -\eps^d \norm{f}_{L^\infty(\Omega_\eps)}\tonde{\eps^{d-1} \sum_{x\in \partial\Omega_\eps} g^{\eps,\beta}(x)\tonde{\eps^{\beta-1}\sum_{\substack{z\in \partial_e\Omega_\eps\\z\sim x}}\axz\tonde{\vartheta(z)-h^{\eps,\beta}(x)}^2}}\\
				&\qquad+ \eps^d \norm{f}_{L^\infty(\Omega_\eps)} \int_0^\infty\tonde{\eps^d \sum_{x\in \Omega_\eps} P^{\eps,\beta}_t|f|(x)\, A^{\eps,\beta}(h^{\eps,\beta})^2(x)} \dd t\\
				&\qquad\le  \eps^d \norm{f}_{L^\infty(\Omega_\eps)}	\int_0^\infty\tonde{\eps^d \sum_{x\in \Omega_\eps} P^{\eps,\beta}_t|f|(x)\, A^{\eps,\beta}(h^{\eps,\beta})^2(x)} \dd t \comma	
							\end{align*}
			where we estimated the first term on the right-hand side above by zero.
			Moreover, since both~$\abs{f}$ and $A^{\eps,\beta}(h^{\eps,\beta})^2$ are zero on $\partial_e\Omega_\eps$, by symmetry of $P^{\eps,\beta}_t$ on $L^2(\Omega_\eps)$, we further have that
			\begin{align*}
				\cR^{\eps,\beta}&\le  \eps^d  \norm{f}_{L^\infty(\Omega_\eps)}\int_0^\infty \tonde{\eps^d \sum_{x\in \Omega_\eps} |f(x)|\, P^{\eps,\beta}_tA^{\eps,\beta}(h^{\eps,\beta})^2(x)} \dd t
				\\
				&\qquad=\eps^d  \norm{f}_{L^\infty(\Omega_\eps)} \tonde{\eps^d \sum_{x\in \Omega_\eps}|f(x)|\int_0^\infty \frac{\dd}{\dd t}\, P^{\eps,\beta}_t(h^{\eps,\beta})^2(x)\, \dd t}
				\\
				&\qquad=\eps^d \norm{f}_{L^\infty(\Omega_\eps)}\tonde{ \eps^d \sum_{x\in \Omega_\eps} |f(x)| \lim_{t\to \infty}\tonde{P^{\eps,\beta}_t(h^{\eps,\beta})^2(x)-(P^{\eps,\beta}_th^{\eps,\beta}(x))^2}}\\
				&\qquad\le\eps^d \norm{\vartheta}_{\cC_b(\R^d)}^2 \norm{f}_{L^\infty(\Omega_\eps)}\norm{f}_{L^1(\Omega_\eps)}\fstop
			\end{align*}
			This concludes the proof of the lemma.
		\end{proof}
	\end{lemma}

\subsection{Proof of Theorem~\ref{t:LocalEquilibrium}}\label{sec:hk-conv}
Let us define,  for all $k \in \N$ and $\bx = (x_1,\ldots, x_k) \in \overline\Omega_\eps^k$,
\begin{equation}\label{eq:h-eps-k}
	h^{\eps,\otimes k}(\bx)\eqdef (h^\eps)^{\otimes k}(\bx)\qquad \text{and}\qquad	h^{\eps,k}(\bx)\eqdef \lim_{t\to \infty} P^{\eps,k}_th^{\eps,\otimes k}(\bx)\comma 
\end{equation}
or, alternatively, in terms of the duality functions in \eqref{eq:duality_function},
\begin{equation}\label{eq:h-eps-k-duality}
	h^{\eps,\otimes k}(\bx)= \prod_{i=1}^k \E_{\nu^\eps_\stat}\quadre{D(x_i,\eta)}\qquad \text{and}\qquad h^{\eps,k}(\bx)= \E_{\nu^\eps_\stat}\quadre{D(\bx,\eta)}\fstop
\end{equation}
Note that,	in general, $h^{\eps,\otimes k}$ and $h^{\eps,k}$  do not coincide. 

Before presenting the proof of Theorem~\ref{t:LocalEquilibrium}, we derive a corollary to be employed later in the proof of Theorem~\ref{th:fluctuations-stat} (\S\ref{sec:proof-flu}). In what follows, keeping the analogy with \eqref{eq:EKConvergence},  we write
\begin{equation}
	g_\eps\to g\qquad \text{if}\qquad 	g_\eps \in L^\infty(\overline\Omega_\eps^k)\comma g \in \cC(\overline\Omega)^{\otimes k}\comma \quad \lim_{\eps\downarrow 0}\tnorm{g_\eps-\Pi_\eps^{\otimes k} g}_{L^\infty(\Omega_\eps^k)}=0\fstop
\end{equation}

\begin{corollary}[Cf.\ Thm.~\ref{th:harmonic_conv}]\label{pr:harmonics_conv-k} $h^{\eps,k}\to h^{\bd,k}\eqdef (h^\bd)^{\otimes k}$ for every $k\in \N$.
	\begin{proof}
		By triangle inequality and since~$h^{\eps,\otimes k}\to h^{\bd,k}$ (Thm.\ \ref{th:harmonic_conv}), it suffices to show
			\begin{equation}\label{eq:hk-conv}
			\lim_{\eps\downarrow 0}\tnorm{h^{\eps,k}-h^{\eps,\otimes k}}_{L^\infty(\Omega_\eps^k)}=0\comma\qquad k \in \N\comma
		\end{equation}
	which is precisely the statement of Theorem~\ref{t:LocalEquilibrium}.
	\end{proof}
\end{corollary} 

\rosso{\begin{lemma}
For every~$k\in\N$ there exists $C=C(\Omega,\vartheta,k)>0$ such that 
		\begin{equation}\label{eq:hk-conv1.3}
			\tnorm{h^{\eps,k}-h^{\eps,\otimes k}}_{L^1(\Omega_\eps^k)}\le C\tonde{\frac{\eps^{2d}}{2}\sum_{\substack{x,y\in\Omega_\eps\\ x\sim y}} \ttonde{\nabla^\eps_{x,y}h^\eps}^2\int_0^\infty 2 P^{\eps,k=2}_t\car_{\Omega_\eps^{k=2}}(x,y)\,\dd t	} \fstop
		\end{equation}
\begin{proof}
The ideas behind most  of the steps are all already contained in the proof of Lemma~\ref{lemma:hydrostatic_L2}.
We provide the full argumentation for completeness.
Since the $k$-particle dual system differs from a system of~$k$ independent dual random walks only when pairs of particles are located on nearest-neighboring sites on~$\Omega_\eps$, for $\mathbf x\eqdef (x_1,\ldots, x_k)\in \overline\Omega_\eps^k$ we have
	\begin{align*}
		\tonde{A^{\eps,k}-A^{\eps,\otimes k}}f^{\otimes k}(\mathbf x)=\sigma\sum_{\substack{i,j=1\\i< j}}^k	\car_{x_i\sim x_j\in \Omega_\eps}\eps^{-2}\tonde{f(x_i)-f(x_j)}^2
		\prod_{k\neq i, j} f(x_k)	
			\fstop
	\end{align*}
In particular, the above expression is either non-positive ($\sigma=-1$), or non-negative ($\sigma=1$). Furthermore, 	since $P^\eps_s h^\eps=h^\eps$, 
	\begin{align}
	\nonumber
	&\eps^{kd}\sum_{\mathbf x\in \Omega_\eps^k}\abs{	\tonde{P^{\eps,k}_t-P^{\eps,\otimes k}_t}h^{\eps,\otimes k}(\mathbf x)}\car_{\Omega_\eps}[\mathbf x]
	\\
	\nonumber
	&\qquad 	=\eps^{kd}\sum_{\mathbf x\in \Omega_\eps^k}\abs{ \int_0^t P^{\eps,k}_{t-s}\tonde{A^{\eps,k}-A^{\eps,\otimes k}}P^{\eps,\otimes k}_s h^{\eps,\otimes k}(\mathbf x)\, {\rm d}s} \car_{\Omega_\eps}[\mathbf x]
	\\
	\nonumber
		&\qquad 	=\eps^{kd}\sum_{\mathbf x\in \Omega_\eps^k}\abs{ \int_0^t P^{\eps,k}_{t-s}\tonde{A^{\eps,k}-A^{\eps,\otimes k}} h^{\eps,\otimes k}(\mathbf x)\, {\rm d}s} \car_{\Omega_\eps}[\mathbf x]
		\\
		\label{eq:l:AuxiliaryConsistency:1}
		&\qquad \le c\eps^{kd}\sum_{\mathbf x\in \Omega_\eps^k} \tonde{\int_0^t \sum_{\mathbf y\in \overline\Omega_\eps^k} p^{\eps,k}_s(\mathbf x,\mathbf y)\, \sum_{\substack{i,j=1\\i<j} }^k\car_{y_i\sim y_j\in \Omega_\eps}\ttonde{\nabla^\eps_{y_i,y_j}h^\eps}^2 {\rm d}s }\car_{\Omega_\eps}[\mathbf x] \comma
	\end{align}
where $c=c(k,f) \in \R_+$.
Now, set
\begin{align*}
	\varphi(x,y)\eqdef \car_{x\sim y\in \Omega_\eps}\tonde{\nabla^\eps_{x,y}\,h^\eps}^2\comma\qquad x, y \in \overline\Omega_\eps\comma
\end{align*}
and observe that 
\begin{equation}\label{eq:l:AuxiliaryConsistency:2}
	 \sum_{\substack{i,j=1\\i<j} }^k\car_{y_i\sim y_j\in \Omega_\eps}\ttonde{\nabla^\eps_{y_i,y_j}h^\eps}^2 = C\,  J^{\eps,k-1} J^{\eps,k-2}\cdots J^{\eps,2}\varphi(\mathbf y)\comma\qquad 	\mathbf y\in \overline \Omega_\eps^k\comma
\end{equation}
where $C=C(k)>0$ is a universal combinatorial factor.
Combining~\eqref{eq:l:AuxiliaryConsistency:1} and~\eqref{eq:l:AuxiliaryConsistency:2} and applying $k-2$ times the consistency property~\eqref{eq:consistency}, we get
\begin{align*}
&	\eps^{kd}\sum_{\mathbf x\in \Omega_\eps^k}\abs{	\tonde{P^{\eps,k}_t-P^{\eps,\otimes k}_t}h^{\eps,\otimes k}(\mathbf x)} \car_{\Omega_\eps}[\mathbf x]\\
&\qquad\lesssim\eps^{kd}\sum_{\mathbf x \in \Omega_\eps^k} \tonde{P^{\eps,k}_s J^{\eps,k-1} J^{\eps,k-2}\cdots J^{\eps,2}\varphi(\mathbf x)}\car_{\Omega_\eps}[\mathbf x]\\
&\qquad
	 =\eps^{kd}\sum_{\mathbf x \in \Omega_\eps^k} \tonde{J^{\eps,k-1}P^{\eps,k-1}_s J^{\eps,k-2}\cdots J^{\eps,2}\varphi(\mathbf x)}\car_{\Omega_\eps}[\mathbf x]\\
	&\qquad = \eps^{kd}\sum_{\mathbf x\in \Omega_\eps^k} \tonde{J^{\eps,k-1}J^{\eps,k-2}\cdots J^{\eps,2} P^{\eps,2}_s \varphi(\mathbf x)}\car_{\Omega_\eps}[\mathbf x]\lesssim \eps^{2d} \sum_{x,y \in \Omega_\eps} P^{\eps,2}_s\varphi(x,y)\, \car_{\Omega_\eps}[(x,y)] \comma
\end{align*}
where the last estimate is a consequence of the definition of the $J^\eps$-operators and the fact that~$\mu_\eps(\Omega_\eps)=O(1)$. The desired estimate now follows by symmetry of $P^{\eps,2}_s$.
\end{proof}
\end{lemma}}

\begin{proof}[Proof of Theorem~\ref{t:LocalEquilibrium}]
		We prove~\eqref{eq:hk-conv} 	 by induction on $k\in \N$. The case $k=1$ is Theorem \ref{th:harmonic_conv}; thus, let $k\ge 2$ and assume  \eqref{eq:hk-conv}  for all $\ell < k$.
		
		Fix $t>0$. Since $P^{\eps,k}_t h^{\eps,k}=h^{\eps,k}$,  by triangle inequality, 
		\begin{equation}\label{eq:hk-conv-two-terms}
			\begin{aligned}
				\tnorm{h^{\eps,k}-h^{\eps,\otimes k}}_{L^\infty(\Omega^k_\eps)}\leq&\ 	\tnorm{P^{\eps,k}_t\ttonde{h^{\eps,k}-h^{\eps,\otimes k}}}_{L^\infty(\Omega^k_\eps)}
				\\
				&\qquad+	\tnorm{P^{\eps,k}_th^{\eps,\otimes k}-h^{\eps,\otimes k}}_{L^\infty(\Omega^k_\eps)}\fstop
			\end{aligned}
		\end{equation}
		Note that, since  $\vartheta \in \cC_b(\R^d)$, the induction hypothesis implies $\tnorm{h^{\eps,k}-h^{\eps,\otimes k}}_{L^\infty(\partial_e\Omega_\eps^k)}\to0$ (\purple{recall that~$\partial_e\Omega_\eps^k= \overline\Omega_\eps^k\setminus \Omega_\eps^k$}).	
For the first term on the right-hand side of \eqref{eq:hk-conv-two-terms} we thus have
		\begin{equation}
			\label{eq:hk-conv1.1}
			\lim_{\eps\downarrow 0} \tnorm{P^{\eps,k}_t\ttonde{h^{\eps,k}-h^{\eps,\otimes k}}}_{L^\infty(\Omega^k_\eps)}\le \lim_{\eps\downarrow 0} \tnorm{P^{\eps,k}_t\ttonde{\car_{\Omega_\eps^k}\ttonde{h^{\eps,k}-h^{\eps,\otimes k}}}}_{L^\infty(\Omega^k_\eps)}\fstop
		\end{equation}
		By ultracontractivity of the $k$-particle semigroups (Cor.\ \ref{cor:nash-k} for $\sigma=1$, Rmk.\ \ref{rmk:nash-k-SEP} for $\sigma=-1$),
		\begin{equation}\label{eq:hk-conv1.2}
			\tnorm{P^{\eps,k}_t\ttonde{\car_{\Omega_\eps^k}\ttonde{h^{\eps,k}-h^{\eps,\otimes k}}}}_{L^\infty(\Omega^k_\eps)}\le C\ttonde{1+t^{-kd/2}}\tnorm{h^{\eps,k}-h^{\eps,\otimes k}}_{L^1(\Omega_\eps^k)}\fstop
		\end{equation}
\rosso{Using~\eqref{eq:hk-conv1.3}, we conclude from the above inequality that, for some $C=C(\Omega,\vartheta,k)>0$, 
		\begin{align*}
			&\tnorm{P^{\eps,k}_t\ttonde{\car_{\Omega_\eps^k}\ttonde{h^{\eps,k}-h^{\eps,\otimes k}}}}_{L^\infty(\Omega^k_\eps)}
			\\
			&\qquad \leq C\tonde{\frac{\eps^{2d}}{2}\sum_{\substack{x,y\in\Omega_\eps\\ x\sim y}} \ttonde{\nabla^\eps_{x,y}h^\eps}^2\int_0^\infty 2 P^{\eps,k=2}_t\car_{\Omega_\eps^{k=2}}(x,y)\,\dd t	} \comma
		\end{align*}
		where we note that the expression between parenthesis equals $\cR^{\eps,\beta}$ given in \eqref{eq:Reps} with $f=\car_{\Omega_\eps}$.}
		Then, combining \eqref{eq:hk-conv1.1}--\eqref{eq:hk-conv1.3} with the estimates carried out in the proof of Lemma \ref{lemma:hydrostatic_L2} ensures that the first term in the right-hand side of \eqref{eq:hk-conv-two-terms} vanishes in~$\eps$ for every fixed~$t>0$.

		We now show that the second term in the right-hand side of~\eqref{eq:hk-conv-two-terms} vanishes letting first~$\eps\to 0$ and then~$t\to 0$.
		 To this end, for every~$\delta>0$ and $x\in \Omega_\eps$, let $Q^\eps_\delta(x)=Q^{\eps,\beta}_\delta(x)$ be given by
		 \begin{align*}
		 Q^\eps_\delta(x)\eqdef \begin{cases} \set{y\in\Omega_\eps : \abs{x-y}<\delta} & \text{if } \beta\geq 1\\
		 \set{y\in\Omega_\eps : \abs{x-y}<\delta} \cup \set{z\in\partial_e\Omega_\eps: \purple{z\sim y \in \partial \Omega_\eps}, \abs{x-y}<\delta} & \text{if } \beta<1
		 \end{cases} \fstop
		 \end{align*}
		  For~$\bx\in \Omega_\eps^k$, further set $Q^{\eps,k}_\delta(\bx)\eqdef \prod_{i=1}^k Q^\eps_\delta(x_i)\subset \overline\Omega_\eps^k$.
		  The desired claim follows as soon as we show that 
		 \begin{equation}\label{eq:8.15}
		 	 \sup_{\bx\in \Omega_\eps^k}\sum_{\by\in  Q^{\eps,k}_\delta(\bx)} p^{\eps,k}_t(\bx,\by) \tabs{h^{\eps,\otimes k}(\by)-h^{\eps,\otimes k}(\bx)}+ \sup_{\bx\in \Omega_\eps^k}  \mbfP_\bx^{\eps,k}\tonde{ \mbfX_t^{\eps,k}\notin  	Q^{\eps,k}_\delta(\bx)}
		 \end{equation} 
	 vanishes taking the limits $\eps\to 0$, $t\to 0$ and $\delta \to 0$, in this order.
	 By definition of~$Q^{\eps,k}_\delta(\bx)$ and since~$h^\eps\to h^\bd \in \cC(\overline\Omega)$ (Thm.~\ref{th:harmonic_conv}; also recall~\eqref{eq:bd_ext} and that~$h^\Dir\vert_{\partial\Omega}=\vartheta\vert_{\partial\Omega}$), the first term in~\eqref{eq:8.15} vanishes in the limit.
	 As for the second term, we now prove that, for every $\delta>0$,
	 \begin{equation*}
	 	\lim_{t\downarrow 0}\limsup_{\eps\downarrow 0}\sup_{\bx\in\Omega_\eps^k} \mbfP_\bx^{\eps,k}\tonde{ \mbfX_t^{\eps,k}\in  \ttonde{Q^{\eps,k}_\delta(\bx)}^c\cap 	\Omega_\eps^k}+ \mbfP_\bx^{\eps,k}\tonde{ \mbfX_t^{\eps,k}\in  \ttonde{Q^{\eps,k}_\delta(\bx)}^c\cap \partial_e\Omega_\eps^k}=0\fstop
	 \end{equation*}
 
The first term above vanishes uniformly by the $k$-particle analogues  of the exit-time estimates in \eqref{eq:exit-time} for all $\beta \in \R$. The derivation of such estimates is a consequence \cite[Thm.~2.7]{bass_hsu1991} of off-diagonal estimates for   $p^{\eps,k}_t(\emparg,\emparg):\Omega_\eps^k\times \Omega_\eps^k\to[0,1]$, which in turn follow from the ultracontractivity of the $k$-particle semigroups (Cor.\  \ref{cor:nash-k} and Rmk.\ \ref{rmk:nash-k-SEP}) by means of \emph{Davies' method} (see, e.g., \cite[\S3]{carlen_upper_1986}, or \cite[\S2.4]{saloff-coste_aspects2001}).  

 \purple{As for the second term, we divide the proof into two cases. (Recall that the definition of $Q^\eps_\delta(x)$ just above \eqref{eq:8.15} changes depending on whether $\beta<1$ or $\beta\ge 1$.) Since, for $\beta<1$, 
\begin{align}	\label{eq:ancoraa}
	\set{\mbfX_t^{\eps,k}\in  \ttonde{Q^{\eps,k}_\delta(\bx)}^c\cap \partial_e\Omega_\eps^k} 
		\subset 
	\bigcup_{s < t} \set{\mbfX_s^{\eps,k}\in  \ttonde{Q^{\eps,k}_\delta(\bx)}^c\cap \Omega_\eps^k} \comma
\end{align}
the strong Markov property and  off-diagonal estimates ensure the vanishing of the second term in this case.
 If $\beta \ge 1$, by definition, $\ttonde{Q^{\eps,k}_\delta(\bx)}^c\cap \partial_e\Omega_\eps^k = \partial_e \Omega_\eps^k$. (Note that \eqref{eq:ancoraa} fails in this case). Hence, we must show that}
 \begin{equation}\label{eq:8.17}
 	\lim_{t\downarrow 0}\limsup_{\eps\downarrow 0} \sup_{\bx\in \Omega_\eps^k} \mbfP^{\eps,k}_\bx\tonde{\mbfX^{\eps,k}_t \in \partial_e\Omega_\eps^k}=0\fstop
 \end{equation}
By employing a Feynmann--Kac representation formula for the $k$-particle semigroup with parameter $\beta \ge 1$, the very same argument as in the proof of Lemma \ref{l:UnifCont} yields (recall $\eps^\infty\eqdef 0$)
\begin{equation}\label{eq:8.18}
\mbfP^{\eps,k}_\bx\tonde{\mbfX^{\eps,k}_t \in \partial_e\Omega_\eps^k}\le \eps^{\beta-1}\int_0^t \mbfE^{\eps,\beta=\infty,k}_\bx\quadre{\sum_{i=1}^k V_\eps\ttonde{(\mbfX^{\eps,\beta=\infty,k}_s)_i}}\dd s\comma\qquad
\end{equation}
where $(\mbfX^{\eps,\beta=\infty,k}_s)_i$ denotes the $i^{\textrm{th}}$ entry of $\mbfX^{\eps,\beta=\infty,k}_s\in \overline\Omega_\eps^k$, with $\mbfE^{\eps,\beta=\infty,k}_\bx$  the corresponding expectation.
By consistency \eqref{eq:consistency} and by~\eqref{eq:8.18},
\begin{equation*}
\mbfP^{\eps,k}_\bx\tonde{\mbfX^{\eps,k}_t \in \partial_e\Omega_\eps^k}\le 	C\, \eps^{\beta-1} \sup_{x\in\Omega_\eps} \int_0^t \mbfE^{\eps,\beta=\infty}_x\quadre{V_\eps(X^{\eps,\beta=\infty}_s)}\dd s\comma
\end{equation*}
for some constant $C=C(k)>0$.	By \eqref{eq:HKbdV} and $\beta \ge 1$, \eqref{eq:8.17} follows. 
\end{proof}

\section{Proof of Theorem \ref{th:fluctuations-stat}} \label{sec:proof-flu}
The main step in the proof of Theorem \ref{th:fluctuations-stat} consists in showing that, for all $f\in \cS^\bd(\Omega)$, the real-valued random variables $\tseq{\tscalar{\mcY^{\eps,\beta}_0}{f}}_\eps$ are asymptotically Gaussian, with mean zero and variance $\var_\bd(f)\eqdef \cov_\bd(f,f)$ as in \eqref{eq:covariance}, i.e., 
\begin{equation}\label{eq:asymptotic_gaussian}
	\tscalar{\cY^{\eps,\beta}_0}{f} \underset{\eps\downarrow 0} \Longrightarrow Y^{\bd,f}\sim \cN(0,\var_\bd(f))\comma\qquad f\in \cS^\bd(\Omega)\fstop
\end{equation}
Thus, by L\'evy's Continuity Theorem for characteristic functionals on the nuclear space~$\cS^\bd(\Omega)$ (see, e.g.~\cite{Bou73}), we conclude that the $\cS^\bd(\Omega)'$-valued fields $\ttseq{\cY^\eps_0}_\eps$ converge in distribution to the unique Borel probability measure on~$\cS^\bd(\Omega)'$ with characteristic functional~$e^{-\var_\bd(f)/2}$.
The existence of the latter measure is a consequence of the classical Bochner--Minlos theorem on the nuclear space~$\mcS^\bd(\Omega)'$ (see, e.g.,~\cite[Thm.~2.3.1]{kallianpur_xiong_1995}). 
Finally, establishing asymptotic Gaussianity  of all finite--dimensional distributions of $\ttseq{\cY^{\eps,\beta}}_\eps$ goes through analogous arguments dealing with multivariate random variables; we leave the details of this last part to the reader.

 Everywhere in the following, we drop the superscript~$\beta \in \R$,  omit the operator~$\Pi_\eps$, and, for every $f, g\in \R^{\Omega_\eps}$ and $\tilde f\in  \cC^\bd$ such that $\Pi_\eps \tilde f= f$, we simply write~$\scalar{g \mu_\eps}{f}$ in place of $\tscalar{g \mu_\eps}{\tilde f}$.
\purple{Recall the definition~\eqref{eq:fluctuation-fields} of the fluctuation fields~$\cY^{\eps,\beta}$}. Then, by duality \eqref{eq:duality}  and the Markov property of $\eta^\eps_t$, for all $f\in \R^{\Omega_\eps}$, 
\begin{equation}\label{eq:2.1}
	\tscalar{\cY^\eps_t}{f}= \tscalar{\cY^\eps_0}{P^\eps_t f} + \int_0^t \tscalar{\diff \cM^\eps_s}{P^\eps_{t-s}f}\comma \qquad 
	t \in \R^+_0\comma
\end{equation}
where $\cM^\eps_t$ are \emph{c\`{a}dl\`{a}g} $\cF^\eps$-adapted $\cS^\bd(\Omega)'$-valued martingales.
Now, fix $f\in \cS^\bd(\Omega)\subset \cC^\bd$, $T\in \R^+$,  $\ttseq{\tau_\eps}_\eps\subset \R^+$ and $\ttseq{f_\eps}_\eps$ such that $f_\eps \to f$ (see \eqref{eq:EKConvergence}). Then, define the following $\cF^\eps$-adapted martingales
\begin{equation}\label{eq:2.2}
	t \in [0,T]\longmapsto	\cW^\eps_{t,T}\eqdef \int_0^{\tau_\eps t} \tscalar{\diff \cM^\eps_s}{P^\eps_{\tau_\eps T-s}f_\eps} \in \mcD([0,T];\R) \comma
\end{equation}
with  predictable quadratic variations $\cA^\eps_{t,T}$ given  by (recall that~$\eta^\eps_s \sim \nu^\eps_\stat$ and \eqref{eq:Veps})
\begin{equation}\label{eq:A-eps-t}
	\begin{aligned}
		\cA^\eps_{t,T}&\eqdef\int_0^{\tau_\eps t} \frac{\eps^d}{2} \sum_{\substack{x,y\in\Omega_\eps\\ x\sim y}}\tonde{\nabla^\eps_{x,y}P^\eps_{\tau_\eps T-s}f_\eps}^2 \cV^\eps\ttonde{(x,y),\eta^\eps_s}\, \dd s\\
		&\qquad+\int_0^{\tau_\eps t} \eps^{d-1} \sum_{x\in\partial\Omega_\eps} \eps^{\beta-1}\tonde{P^\eps_{\tau_\eps T-s}f_\eps(x)}^2 \sum_{\substack{z\in \partial_e\Omega_\eps\\ z\sim x}} \axz\, \cV^\eps((x,z),\eta^\eps_s)\,\dd s\fstop
	\end{aligned}
\end{equation}
In particular, by combining \eqref{eq:2.1} and \eqref{eq:2.2},
\begin{equation}\label{eq:decomposition-flu}
	\tscalar{\cY^\eps_{\tau_\eps T}}{f_\eps}= \tscalar{\cY^\eps_0}{P^\eps_{\tau_\eps T}f_\eps} + \cW^\eps_{T,T}\comma
\end{equation}
and the two random variables on the right-hand side above are uncorrelated. Furthermore, since~$f_\eps\to f$, Lemma~\ref{lemma:hydrostatic_L2} yields
\begin{equation}\label{eq:approximation-flu}
\lim_{\eps\downarrow0}\E^\eps_{\nu^\eps_\stat}\quadre{\tonde{\scalar{\cY^\eps_{\tau_\eps T}}{f-f_\eps}}^2}=0\fstop
\end{equation}
Now, combining Corollary~\ref{c:spectral_bound_DR} and Proposition~\ref{p:ground_states}\ref{i:p:ground_states:ev} yields that~$\lambda_0^\eps\ge \underline \lambda_0 (1\wedge \eps^{\beta-1})>0$. 
This fact together with Lemma~\ref{lemma:hydrostatic_L2} ensures then that
\begin{equation}\label{eq:l2-bounds-flu}
	\E^\eps_{\nu_\stat^\eps}\quadre{\tonde{\tscalar{\cY^\eps_{\tau_\eps T}}{f_\eps}}^2}\le C\comma\qquad
	\E^\eps_{\nu_\stat^\eps}\quadre{\tonde{\tscalar{\cY^\eps_0}{P^\eps_{\tau_\eps T}f_\eps}}^2}\le  C\,	 e^{-\underline\lambda_0 (1\wedge \eps^{\beta-1})\tau_\eps T}\comma
\end{equation}
for some $C=C(\Omega,\vartheta,f)>0$. 	If, additionally, $\beta>1$ and $\tscalar{f_\eps}{\psi_0^{\eps,\beta}}_{L^2(\Omega_\eps)}=0$, then
\begin{equation}\label{eq:l2-bounds-flu2}
	\E^\eps_{\nu_\stat^\eps}\quadre{\tonde{\tscalar{\cY^\eps_0}{P^\eps_{\tau_\eps T}f_\eps}}^2}\le  C\,	 e^{-\underline\lambda_0 	\tau_\eps T}\fstop
\end{equation}

In view of \eqref{eq:decomposition-flu}--\eqref{eq:l2-bounds-flu2} and the stationarity of $\nu^\eps_\stat$, our main goal becomes that of selecting suitable $\ttseq{\tau_\eps}_\eps$ and $\ttseq{f_\eps}_\eps$ so to:
\begin{itemize}
	\item establish, for every fixed $T\in \R^+$, a limit theorem for $\ttseq{\cW^\eps_{T,T}}_\eps$;
	\item ensure that the contributions coming from the first term in the right-hand side of \eqref{eq:decomposition-flu} vanish letting~$T\to \infty$.
\end{itemize}

As for the first step, exploiting the fact that~$\tseq{\cW^\eps_{t,T}}_\eps$ is a martingale,  the Martingale Central Limit Theorem as in, e.g.,~\cite[Thm.\ 7.1.4(b), p.~339]{ethier_kurtz_1986_Markov} applies to our case:  the condition in~\cite[Eqn.~(1.16), p.~340]{ethier_kurtz_1986_Markov} holds true since the predictable quadratic variations are $\P^\eps_{\nu^\eps_\stat}$-a.s.\ continuous, while the one in~\cite[Eqn.~(1.17), p.~340]{ethier_kurtz_1986_Markov} because~$f\in \cC^\bd$ is bounded and since the particle systems register at most one-particle jump at a time.
Thus, the convergence of~$\cW^\eps_{T,T}$ reduces to a weak law of large numbers for $\cA^\eps_{t,T}$: for some  continuous function $a=a_T:[0,T]\to \R^+$, $a(0)=0$,  
\begin{equation}
\lim_{\eps\downarrow 0}\P^\eps_{\nu^\eps_\stat}\tonde{\abs{	\cA^\eps_{t,T}- a(t)}>\delta}=0\comma\qquad t \in [0,T]\comma \delta >0\fstop
\end{equation}  The  proof of the latter claim --- a specific instance of a Boltzmann--Gibbs principle --- is the content of the next section.

\subsection{Boltzmann-Gibbs principles}\label{sec:proof-boltzmann-gibbs}In this section, for all $f \in \cS^\bd(\Omega)$, we provide suitable $\ttseq{\tau_\eps}_\eps$ and $\ttseq{f_\eps}_\eps$  ensuring a weak law of large numbers for all $\cA^\eps_{t,T}$ given in \eqref{eq:A-eps-t}.
This entails replacing the cylinder functions $\cV^\eps\ttonde{(x,y),\eta}$ in \eqref{eq:A-eps-t} with their expected values (Lem.~\ref{l:QV-var}); then, by means of Theorem \ref{t:MainSemigroups} and  Corollary~\ref{pr:harmonics_conv-k}, we prove convergence of  $\E^\eps_{\nu^\eps_\stat}[\cA^\eps_{t,T}]$ (Lem.~\ref{l:QV-mean}).

\begin{lemma}\label{l:QV-mean}
For every  $f \in \cS^\bd(\Omega)$, assume the following:
	\begin{enumerate}[$(a)$]
		\item\label{i:l:QV-mean-a}  if either $\beta \le 1$, or $\beta > 1$ and $\int f\, \dd \mu_\Omega=0$,  set $\tau_\eps \eqdef 1$ and $\ttseq{f_\eps}_\eps$ be such that $f_\eps \to f$;
		\item\label{i:l:QV-mean-b} if $\beta >1$ and $f = \frac{c}{\sqrt{\mu_\Omega(\Omega)}}\car_{\Omega}$ with $c \in \R$, set $\tau_\eps\eqdef 1/\lambda_0^\eps\purple{\vee} 1$ and $f_\eps\eqdef c\,\psi_0^\eps$, see \eqref{eq:eigen_eq}.
	\end{enumerate}
	Then, for the corresponding $\cA^\eps_{t,T}$ as in \eqref{eq:A-eps-t}, 
	\begin{equation}\label{eq:mean-A-conv-eps}
		\lim_{\eps\downarrow 0} \E^\eps_{\nu^\eps_\stat}\quadre{\cA^\eps_{t,T}} =  a_\bd(t,T)\comma \qquad T\in \R^+\comma t \in [0,T]\comma
	\end{equation}
where $a_\bd(\emparg,T)\in \cC([0,T];\R^+_0)$ and $a_\bd(0,T)=0$. Furthermore, recalling \eqref{eq:covariance},
\begin{equation}\label{eq:mean-A-conv-T}
\lim_{T\to \infty}	a_\bd(T,T) = {\rm var}_\bd(f) \eqdef \cov_\bd(f,f) \fstop
\end{equation}
\end{lemma}
\begin{proof}
			By \eqref{eq:Veps} and \eqref{eq:h-eps-k-duality}, we have that
		\begin{equation}\label{eq:Veps-overline}
			\overline\cV^\eps(x,y)\eqdef	\E_{\nu^\eps_\stat}\quadre{\cV^\eps\ttonde{(x,y),\eta}}= h^\eps(x)+h^\eps(y)+2\sigma h^{\eps,k=2}(x,y)\comma\quad x, y \in \overline\Omega_\eps\fstop
		\end{equation} 
		\purple{Recall the definition~\eqref{eq:chi} of~$\chi^\bd$}.
		By  Corollary~\ref{pr:harmonics_conv-k} and the uniform continuity of $h^\bd \in \cC(\overline\Omega)$, 
		\begin{equation}\label{eq:convergence1}
			\lim_{\eps\downarrow 0}\purple{\sup_{x \sim y\in\Omega_\eps}} \abs{\overline\cV^\eps(x,y)- \chi^\bd(x)- \chi^\bd(y)}=0\fstop
		\end{equation} 
		Note that, in general, the  convergence in \eqref{eq:convergence1} is only on $\Omega_\eps^{k=2}$, not on the whole of~$\overline\Omega_\eps^{k=2}$; outside $\Omega_\eps^{k=2}$, we have instead that
		\begin{equation}\label{eq:convergence2}
			\lim_{\eps\downarrow 0}\sup_{x\in \overline\Omega_\eps}\purple{\sup_{y\in \partial_e\Omega_\eps, \, y \sim x}}\abs{\overline\cV^\eps(x,y)-h^\bd(x)\ttonde{1+\sigma \vartheta(y)}-\vartheta(y)\ttonde{1+\sigma h^\bd(x)}} =0\fstop
		\end{equation}
		As a consequence of the above uniform convergences, the uniform continuity of $\vartheta$, \eqref{eq:bd_ext}, and 
		\begin{equation}\label{eq:boundedness-l1}
			\int_0^\infty 2\, \mcE^\eps(P^\eps_r f_\eps)\,\dd r\le  \norm{f_\eps}^2_{L^2(\Omega_\eps)}\le C=C(f)\comma
		\end{equation}
		and recalling \eqref{eq:A-eps-t}, we obtain by H\"older inequality that
		\begin{equation}
			\lim_{\eps\downarrow 0}	\E^\eps_{\nu^\eps_\stat}\quadre{\cA^\eps_{t,T}} = \lim_{\eps\downarrow 0} \overline\cA^\eps_{t,T}\comma
		\end{equation}
		where, further setting $\bar f_\eps\eqdef P^\eps_{\tau_\eps(T-t)}f_\eps$,
		\begin{equation*}
			\begin{aligned}
				\overline\cA^\eps_{t,T}&\eqdef \int_0^{\tau_\eps t} \frac{\eps^d}{2}\sum_{\substack{x,y\in\Omega_\eps\\ x\sim y}} \tonde{\nabla^\eps_{x,y}P^\eps_s\bar f_\eps}^2 \ttonde{\chi^\bd(x)+\chi^\bd(y)}\,\dd s	\\
				&\qquad+\int_0^{\tau_\eps t} \eps^{d-1} \sum_{x\in\partial\Omega_\eps} \alpha_\eps(x)\, \eps^{\beta-1}\tonde{P^\eps_s\bar f_\eps(x)}^2  \ttonde{h^\bd(x)+\vartheta(x)+2\sigma h^\bd(x)\vartheta(x)}\dd s
	 \fstop
			\end{aligned}
		\end{equation*}
	Recalling \eqref{eq:carre-def} and \eqref{eq:SigmaEps}, we rewrite the expression above as
	\begin{equation}\label{eq:mean-A}
		\begin{aligned}
		\overline\cA^\eps_{t,T}&=	\int_0^{\tau_\eps t} \tscalar{2\varGamma^{\eps,\beta}(P^\eps_s\bar f_\eps)}{\chi^\bd}_{L^2(\Omega_\eps)}\dd s\\
		&\qquad + \int_0^{\tau_\eps t} \scalar{\sigma_\eps}{{\eps^{\beta-1} (P^\eps_s \bar f_\eps)^2}\ttonde{\vartheta-h^\bd}\ttonde{1+2\sigma h^\bd}}\dd s
		\end{aligned}
	\end{equation}
	
	We now divide the proof into two parts, depending on whether $\beta >1$ or $\beta \le 1$.
	
	\paragraph{Neumann boundary conditions} Fix $\beta>1$. Recall that both $h^\Neu=\av{\vartheta}_{\partial\Omega}$ and $\chi^\Neu$ are constants. Hence, 	letting $\hslash^\Neu\eqdef 1+2\sigma h^\Neu \in \R$, 
	\begin{equation*}\begin{aligned}
	&	\overline\cA^\eps_{t,T}= \chi^\Neu\int_0^{\tau_\eps t} 2 \mcE^\eps(P^\eps_s \bar f_\eps)\, \dd s + \hslash^\Neu	\int_0^{\tau_\eps t} \scalar{\sigma_\eps}{{\eps^{\beta-1} (P^\eps_s \bar f_\eps)^2}\ttonde{\vartheta-h^\bd}}\dd s\\
		&=\chi^\Neu \tonde{\tnorm{P^\eps_{\tau_\eps(T-t)}f_\eps}^2_{L^2(\Omega_\eps)}-\tnorm{P^\eps_{\tau_\eps T }f_\eps}^2_{L^2(\Omega_\eps)}}+ \hslash^\Neu\int_0^{\tau_\eps t} \scalar{\sigma_\eps}{{\eps^{\beta-1} (P^\eps_s \bar f_\eps)^2}\ttonde{\vartheta-h^\bd}}\dd s\fstop
		\end{aligned}
	\end{equation*}
If case \ref{i:l:QV-mean-a} holds, then  $f_\eps\to f$, Theorem \ref{t:MainSemigroups} and $\beta >1$ ensure that
\begin{equation}
\lim_{\eps\downarrow 0} \overline\cA^\eps_{t,T}= \chi^\Neu\tonde{\tnorm{P^\Neu_{T-t}f}^2_{L^2(\Omega)}-\tnorm{P^\Neu_Tf}^2_{L^2(\Omega)}} \fstop
\end{equation}	
If case \ref{i:l:QV-mean-b} holds, then
\begin{equation}
	\begin{aligned}
	\lim_{\eps\downarrow 0}\overline\cA^\eps_{t,T}&= \chi^\Neu\tonde{e^{-2(T-t)}\norm{f}^2_{L^2(\Omega)}-e^{-2T} \norm{f}^2_{L^2(\Omega)}} \\
	&\qquad+\hslash^\Neu e^{-2(T-t)} \lim_{\eps\downarrow 0}  \tonde{\int_0^{\tau_\eps t}\eps^{\beta-1} e^{-2\lambda_0^\eps s}\dd s} \scalar{\sigma_\eps}{(\psi_0^\eps)^2\ttonde{\vartheta-h^\Neu}}\\
	&= \chi^\Neu\tonde{e^{-2(T-t)}\norm{f}^2_{L^2(\Omega)}-e^{-2T} \norm{f}^2_{L^2(\Omega)}}\comma
	\end{aligned}
\end{equation}
where the last step follows by Proposition \ref{p:ground_states}, \eqref{eq:bd_weak_conv} and the definition of $h^\Neu=\av{\vartheta}_{\partial\Omega}$.
  In either case, both \eqref{eq:mean-A-conv-eps} and \eqref{eq:mean-A-conv-T}  hold true.

\paragraph{Robin and Dirichlet boundary conditions} Fix $\beta \le 1$, and note that case \ref{i:l:QV-mean-a} holds.
\purple{If~$\beta<1$, the second term on the right-hand side of~\eqref{eq:mean-A} vanishes as $\eps \to 0$. Indeed, we have 
	\begin{align*}
		&\abs{\scalar{\sigma_\eps}{{\eps^{\beta-1} (P^\eps_s \bar f_\eps)^2}\ttonde{\vartheta-h^\bd}\ttonde{1+2\sigma h^\bd}}}\\
		&\qquad 
			\leq 
		\tnorm{\vartheta-h^\bd}_{L^\infty(\partial \Omega_\eps)} \ttonde{1 + 2 \norm{\vartheta}_{L^\infty(\partial \Omega_\eps)}}  \mcE^\eps(P^\eps_s \bar f_\eps) \comma
	\end{align*}
and, thus, 	 the claimed convergence to $0$ by \eqref{eq:boundedness-l1}, $h^\Dir\in \cC(\overline\Omega)$, and $h^\Dir\vert_{\partial\Omega}=\vartheta\vert_{\partial\Omega}$.
}
If instead~$\beta=1$, since~$f_\eps\to f$, Theorem \ref{t:MainSemigroups} and \eqref{eq:bd_weak_conv} imply that the second term on the right-hand side of~\eqref{eq:mean-A} converges as $\eps\to 0$ to
\begin{equation}
	\begin{aligned}
 \scalar{\sigma_{\partial\Omega}}{	\tonde{\int_{T-t}^T(P^{\varrho=1}_s f)^2\dd s}\ttonde{\vartheta-h^\bd}\ttonde{1+2\sigma h^\bd}}\fstop
	\end{aligned}
\end{equation}

Fix $T >0$ and let $\cK^\eps(t)=\cK^\eps_T(t)$ denote the first term on the right-hand side of \eqref{eq:mean-A}, 	viz.,
\begin{equation}
	\cK^\eps(t)\eqdef \scalar{ \int_{T-t}^T  2\varGamma^{\eps,\beta}(P^\eps_s f_\eps)\,\dd s}{\chi^\bd}_{L^2(\Omega_\eps)}\fstop
\end{equation} 
By non-negativity of both $\varGamma^{\eps,\beta}(P^\eps_s f)$ and $\chi^\bd\in \cC_b(\R^d)$,  it is immediate to check that  \eqref{eq:boundedness-l1} yields the relative compactness of the family $\ttseq{\cK^\eps}_\eps$ in $\cC([0,T];\R^+_0)$. In order to establish convergence in $\cC([0,T];\R^+_0)$, we observe that, by the very same arguments, the finite measures 
\begin{equation*}
\gamma^\eps_{t,T}\eqdef  \eps^d\sum_{x\in \Omega_\eps} 	\tonde{\int_{T-t}^T2\varGamma^{\eps,\beta}(P^\eps_s f_\eps)(x)\, \dd s} \delta_x \in \Mbp(\overline\Omega) \comma\qquad T\in \R^+\comma t\in [0,T]\comma
\end{equation*} 
form a tight family $\ttseq{\gamma^\eps_{t,T}}_\eps$  in $\Mbp(\overline\Omega)$. Hence, the limit of $\tseq{\cK^\eps}_\eps$ is uniquely determined by the limits of $\tseq{\tscalar{\gamma^\eps_{t,T}}{\varphi}}_\eps$ for all $t \in [0,T]$ and all $\varphi$ in a \emph{dense subspace} of $\cC(\overline\Omega)$. In what follows,  we choose the $\cC(\overline\Omega)$-Laplacian's core  $\cS^\Neu(\Omega)\subset \cC(\overline\Omega)$ as a dense subspace (Prop.~\ref{p:TestF}\ref{i:p:TestF:7}). 

Fix $\varphi \in \cS^\Neu(\Omega)$. By Theorems \ref{t:MainSemigroups} and \ref{t:Equivalence}\ref{i:t:Equivalence:3} for $\beta=\infty$ and $\bd=\Neu$, there exists  $\ttseq{\varphi_\eps}_\eps$ such that $\varphi_\eps\to \varphi$ and $A^{\eps,\beta=\infty}\varphi_\eps\to \Delta^\Neu \varphi$. As a consequence, $\tscalar{\gamma^\eps_{t,T}}{\varphi-\varphi_\eps}\to0$ as $\eps\to 0$.
Recall \eqref{eq:carre-def} and \eqref{eq:generator_RW}; then, 
\begin{align*}
	\tscalar{\gamma^\eps_{t,T}}{\varphi_\eps}&= \int_{T-t}^T\scalar{A^{\eps,\beta=\infty}(P^\eps_s f_\eps)^2-2 (P^\eps_s f_\eps)(A^\eps P^\eps_s f)}{\varphi_\eps}_{L^2(\Omega_\eps)} \dd s\\
	&\qquad =\int_{T-t}^T \scalar{(P^\eps_s f)^2}{A^{\eps,\beta=\infty}\varphi_\eps}_{L^2(\Omega_\eps)}-\scalar{\frac{\dd}{\dd s}(P^\eps_s f_\eps)^2}{\varphi_\eps}_{L^2(\Omega_\eps)}\dd s\comma
\end{align*}
yielding
\begin{align*}
	&\lim_{\eps\downarrow 0} \tscalar{\gamma^\eps_{t,T}}{\varphi}=\lim_{\eps\downarrow 0}\tscalar{\gamma^\eps_{t,T}}{\varphi_\eps}\\
	&\qquad
	 = \int_{T-t}^T -\mcE^\Neu\ttonde{(P^\bd_s f)^2, \varphi}\,\dd s +\tonde{\tscalar{(P^\bd_{T-t}f)^2}{\varphi}_{L^2(\Omega)}-\tscalar{(P^\bd_T f)^2}{\varphi}_{L^2(\Omega)}}\fstop
\end{align*}
This shows \eqref{eq:mean-A-conv-eps}.
By setting $t=T$ and observing that all relative compactness arguments above hold uniformly  over $T\in \R^+$, \eqref{eq:mean-A-conv-T} follows, thus concluding the proof.	
\end{proof}

\begin{lemma}\label{l:QV-var}For every $\ttseq{\tau_\eps}_\eps \subset \R^+$ and $\ttseq{f_\eps}_\eps$ such that $\tau_\eps \ge 1$ and $f_\eps \to f$ as in~\eqref{eq:EKConvergence},
	\begin{equation}
	\lim_{\eps\to 0} \E^\eps_{\nu^\eps_\stat}\quadre{\tonde{\cA^\eps_{t,T}-\E^\eps_{\nu^\eps_\stat}\quadre{\cA^\eps_{t,T}}}^2} =0\comma\qquad T\in \R^+\comma t \in [0,T]\fstop
	\end{equation}
\end{lemma}

\begin{proof} For $t=0$ the claim is straightforward. Fix $t\in (0,T]$.	
		Recalling \eqref{eq:Veps} and \eqref{eq:Veps-overline}, we define
		\begin{equation}
			\widehat \cV^\eps_s(x,y)\eqdef \cV^\eps\ttonde{(x,y),\eta^\eps_s}- \overline\cV^\eps(x,y)\comma\qquad x, y \in \overline\Omega_\eps\comma s \in\R^+_0\fstop
		\end{equation}
		Note that $\widehat \cV^\eps_s(x,y)=0$ if $x,y \in \partial_e\Omega_\eps$, \purple{and since $\cV^\eps(\emparg,	\eta)=J^{\eps,k=1}D(\emparg,\eta)+2\sigma D(\emparg,\eta)$, the consistency provided in \eqref{eq:consistency} implies
		\begin{equation*}
		\E^\eps_\eta\quadre{\cV^\eps((x,y),\eta^\eps_r)}= P^{\eps,k=2}_r\cV^\eps(\emparg,\eta)(x,y)\comma\qquad (x,y)\in\overline\Omega_\eps^{k=2}\comma r \ge 0\semicolon
		\end{equation*}
This yields in particular}
		\begin{equation}
			\label{eq:duality-V-hat}\E^\eps_{\nu^\eps_\stat}\quadre{\widehat \cV^\eps_s(x,y)\big|	\cF^\eps_0}=P^{\eps,k=2}_s \widehat \cV^\eps_0(x,y)\comma\qquad x, y \in \overline\Omega_\eps\comma s \in\R^+_0\fstop\end{equation}
		Further set
		\begin{equation}\label{eq:V4}
			\cU^\eps_{(x,y),(z,w)}\eqdef \E_{\nu^\eps_\stat}\quadre{\widehat \cV^\eps_0(x,y)\widehat \cV^\eps_0(z,w)}\comma\qquad (x,y), (z,w)\in \overline\Omega_\eps^{k=2}\fstop
		\end{equation}
		Let us adopt the following shorthand notation for the Dirichlet form  at $P^\eps_s f$:
		\begin{equation}
			\sum_{e} F^\eps_s(e)\eqdef 	\mcE^\eps(P^\eps_sf) \comma \qquad s \in\R^+_0\comma
		\end{equation}
		with $e, e',\ldots$ denoting (oriented) pairs $(x,y)\in \overline\Omega_\eps$ such that $x\sim y$. Note that $F^\eps_s(e)\ge 0$. Then, the stationarity of $\nu^\eps_\stat$  and \eqref{eq:duality-V-hat} yield  (recall \eqref{eq:HK-k} and \eqref{eq:V4})
		\begin{align*}
				W_\eps\eqdef&\ \E^\eps_{\nu^\eps_\stat}\quadre{\tonde{\cA^\eps_{t,T}-\E^\eps_{\nu^\eps_\stat}\quadre{\cA^\eps_{t,T}}}^2}\\
				%
				=&\ 2 \int_0^{\tau_\eps t}\dd s\,  \int_s^{\tau_\eps t} \dd r \,   \sum_{e, e'} F^\eps_{\tau_\eps T-s}(e) F^\eps_{\tau_\eps T-r}(e')
				\sum_{(x,y)\in \overline\Omega_\eps^{k=2}} p^{\eps,k=2}_{r-s}\ttonde{e',(x,y)}\, \cU^\eps_{e,(x,y)}\fstop
		\end{align*}
		Further,  duality \eqref{eq:duality} ensures that there exists $C=C(\vartheta)>0$ such that
		\begin{equation}\label{eq:upper-bound-U}
			\sup_{e, (x,y)}\tabs{\cU^\eps_{e,(x,y)}}\le C\fstop
		\end{equation}
		Hence, setting   $s_u\eqdef (s+u)\wedge\tau_\eps t$ for some fixed $u\in \R^+$, we get	 (below, $C'=C'(\vartheta,f)>0$) 
		\begin{align*}
				&	W_\eps \le C \int_0^{\tau_\eps t} \sum_e F^\eps_{\tau_\eps T-s}(e) \tonde{\norm{P^\eps_{\tau_\eps T-s}f}_{L^2(\Omega_\eps)}^2-\norm{P^\eps_{\tau_\eps T-s_u}f}_{L^2(\Omega_\eps)}^2}\dd s\\
				&	+ 2 \int_0^{\tau_\eps t}\dd s\,  \int_{s_u}^{\tau_\eps t} \dd r \,   \sum_{e, e'} F^\eps_{\tau_\eps T-s}(e) F^\eps_{\tau_\eps T-r}(e')
				\sum_{(x,y)\in \overline\Omega_\eps^{k=2}} p^{\eps,k=2}_{r-s}\ttonde{e',(x,y)}\, \cU^\eps_{e,(x,y)}\\
				&\qquad \le C' \norm{f-P^\eps_uf}_{L^\infty(\Omega_\eps)}\\
				&	+ 2 \int_0^{(\tau_\eps t-u)\vee 0}\dd s\,  \int_{s_u}^{\tau_\eps t} \dd r \,   \sum_{e, e'} F^\eps_{\tau_\eps T-s}(e) F^\eps_{\tau_\eps T-r}(e')
				\sum_{(x,y)\in \overline\Omega_\eps^{k=2}} p^{\eps,k=2}_{r-s}\ttonde{e',(x,y)}\, \cU^\eps_{e,(x,y)}
				\fstop	
		\end{align*}
The first term on the right-hand side above vanishes as $u\to 0$, uniformly in~$\eps$, $T\in \R^+$ and~$t\in [0,T]$, by Theorem \ref{t:MainSemigroups} and the strong continuity of the continuum semigroups.
The second term vanishes as $\eps\to 0$ (for fixed $u\in \R^+$) as soon as we show 
		\begin{equation}\label{eq:final}
			\lim_{\eps\downarrow 0}\sup_{t\ge u}\sup_{e,e'} \sum_{(x,y)\in\overline\Omega_\eps^{k=2}} p^{\eps,k=2}_t\ttonde{e',(x,y)}\, \cU^\eps_{e,(x,y)}=0\comma \qquad u\in \R^+\fstop
		\end{equation}
		Start by observing that, by the definitions \eqref{eq:duality_function} and \eqref{eq:V4},  Corollary~\ref{pr:harmonics_conv-k} yields
		\begin{equation}\label{eq:Ueps-off-diagonal} \lim_{\eps\downarrow 0} \sup_{x,y \notin \set{z,w}} \tabs{\cU^\eps_{e,(x,y)}}=0\comma \qquad  e=(z,w)\in \overline\Omega_\eps^{k=2}\fstop
		\end{equation}
		By the ultracontractivity of the two-particle semigroups   and \eqref{eq:upper-bound-U}, 
		we get \eqref{eq:final} with $\overline\Omega_\eps^{k=2}$ in the summation being replaced by $\Omega_\eps^{k=2}$. We are left with showing that the contributions from $\partial_e\Omega_\eps^{k=2}$ also vanish. To this purpose, we employ \eqref{eq:consistency} and argue as already done in the proof of Corollary \ref{pr:nash}. Reproducing the above argument for the single-particle system on~$\overline\Omega_\eps$, we conclude the proof of the lemma.
	\end{proof}

\subsection{Conclusion of the proof of Theorem \ref{th:fluctuations-stat}}
Combining Lemmas \ref{l:QV-mean} and \ref{l:QV-var} with the discussion at the beginning of \S\ref{sec:FCLT-flu} concludes the proof of Theorem \ref{th:fluctuations-stat}. More specifically,  by stationarity and the decomposition~\eqref{eq:decomposition-flu}, it suffices to identify the limit in distribution of $\tseq{\tscalar{\cY^\eps_{\tau_\eps T}}{f_\eps}}_{\eps,T}$ as   $\eps\to 0$ and  $T\to \infty$  for every~$f\in \cS^\bd(\Omega)$ and some $\ttseq{f_\eps}_\eps$ such that $f_\eps\to f$. We achieve this by	 setting  $\ttseq{\tau_\eps}_\eps$ and $\ttseq{f_\eps}_\eps$ as in Lemma \ref{l:QV-mean}, and observing that	 the families of real-valued random variables 
\begin{equation}\label{eq:tight-families}
	\tseq{\tscalar{\cY^{\eps,\beta}_0}{P^\eps_{\tau_\eps T}f_\eps}}_{\eps,T}\qquad \text{and}\qquad\ttseq{\cW^\eps_{T,T}}_{\eps,T}
\end{equation} are tight. Hence, taking first $\eps\to 0$ and then $T\to \infty$, on the one side, Lemmas \ref{l:QV-mean} and~\ref{l:QV-var} ensure that the second family in \eqref{eq:tight-families} converges to a centered Gaussian distribution with variance given as in Lemma \ref{l:QV-mean}; on the other side, the first family vanishes. Indeed,  when either $\beta \le 1$ or $\beta > 1$ and $f=c\,\car_{\Omega}$ with $c\in \R$, this latter claim is an immediate consequence of the second inequality in \eqref{eq:l2-bounds-flu};  when $\beta > 1$ and $\int f\, \dd \mu_\Omega=0$, it follows by approximating~$\tscalar{\cY^\eps_0}{P^\eps_T f_\eps}$ with 
\begin{equation}
	\tscalar{\cY^\eps_0}{P^\eps_T \tilde f_\eps}\comma\qquad \tilde f_\eps\eqdef f_\eps-\tscalar{f_\eps}{\psi_0^\eps}_{L^2(\Omega_\eps)}\psi_0^\eps\comma
\end{equation}
(note that $\tnorm{f_\eps-\tilde f_\eps}_{L^\infty(\Omega_\eps)}\to 0$ as $\eps\to 0$ by Prop.\ \ref{p:ground_states}\ref{i:p:ground_states:gs} and $f_\eps\to f$), and 
using \eqref{eq:l2-bounds-flu2}.

\section{Appendix}\label{sec:appendix}
\subsection{Laplacians on Lipschitz domains}\label{sec:appendix-laplacians}
We collect here some auxiliary results on Laplace operators with Dirichlet, Robin, and Neumann boundary conditions on Lipschitz domains. In this section, in contrast to the notation adopted in  \S\ref{sec:semigroup-conv-proofs}--\S\ref{sec:proof-flu}, we distinguish between $L^2$- and $\cC^\bd$-semigroups, generators, resolvents, etc.

\subsubsection{\texorpdfstring{$L^2$}{L2}-Laplacians}
We start by recalling some known spectral properties.
\begin{lemma}[Dominations]\label{l:Domination}
Let~$\varrho_1,\varrho_2\in\mcC(\partial\Omega)$ with~$0<\varrho_1 \leq \varrho_2$. Then,
\begin{itemize}
\item $P^\Dir_t\leq P^{\varrho_2}_t\leq P^{\varrho_1}_t \leq P^\Neu_t$ as non-negative operators for each~$t\geq 0$;
and, equivalently,
\item $\ttonde{\mcE^\Neu,\dom{\mcE^\Neu}}\leq \ttonde{\mcE^{\varrho_1},\dom{\mcE^{\varrho_1}}} \leq \ttonde{\mcE^{\varrho_2},\dom{\mcE^{\varrho_2}}}\leq \ttonde{\mcE^\Dir,\dom{\mcE^\Dir}}$ as quadratic forms.
\end{itemize}
\begin{proof}
The assertion for semigroups is~\cite[Thm.~3.1]{arendt2003dirichlet}. The assertion for the corresponding forms readily follows.
\end{proof}
\end{lemma}

For~$\lambda\in \R$, denote by~$N^\bd _\Omega(\lambda)$ the number of eigenvalues of~$-\Delta^\bd $ which are strictly smaller than~$\lambda\in \R$, counted with multiplicity; if the number of such eigenvalues is infinite, or if~$\sigma_{\mathrm{ess}}(-\Delta^\bd )\cap (-\infty,\lambda)\neq \emp$, we set~$N^\bd _\Omega(\lambda)\eqdef +\infty$.

\begin{proposition}[Weyl asymptotics]\label{p:Weyl}
Let~$\varrho_1,\varrho_2\in\cC(\partial\Omega)$ with~$0<\varrho_1<\varrho_2$ on~$\partial\Omega$. 
\begin{enumerate}[$(i)$]
\item\label{i:p:Weyl:1} $-\Delta^\bd $ has purely discrete spectrum, say~$\seq{\lambda^\bd _n}_n$, $\lambda_n^\bd \geq 0$, indexed with multiplicities;
\item\label{i:p:Weyl:2} $\lambda^\Neu_n\leq \lambda^\varrho_n\leq \lambda^\Dir_n$ and~$\lambda^{\varrho_1}_n<\lambda^{\varrho_2}_n$;
\item\label{i:p:Weyl:3} $N_\Omega^\bd (\lambda)\asymp \lambda^{d/2}$ as $\lambda \to +\infty$. 
\end{enumerate}

\begin{proof}
Assertion~\ref{i:p:Weyl:1} for~$\ttonde{-\Delta^\Dir,\dom{-\Delta^\Dir}}$ is well-known (see e.g. \cite{AreBen99}). Since~$\Omega$ is a bounded Lipschitz domain, it satisfies the \emph{$H^1$-extension property}~\cite[Thm.~6.4.3, p.~285]{Bur98}, hence the resolvent of~$-\Delta^\Neu$ is a compact operator, and the assertion for~$\ttonde{-\Delta^\Neu,\dom{-\Delta^\Neu}}$ follows.
The first assertion for~$\ttonde{-\Delta^\varrho,\dom{-\Delta^\varrho}}$ follows from the corresponding ones for the Dirichlet and Neumann Laplacians and Lemma~\ref{l:Domination}.
The second assertion is shown in~\cite[Thm.~3.2]{rohleder_strict2014}.
Assertion~\ref{i:p:Weyl:2} follows combining Lemma~\ref{l:Domination} and Courant minimax principle~{\cite[Thm.~2.16.1]{LebedevVorovich}} applied to the (compact strictly positive) semigroups~$P^\bd _t$ for some~$t>0$.
As a consequence of~\ref{i:p:Weyl:2}, it suffices to show~\ref{i:p:Weyl:3} for~$-\Delta^\Neu$ and~$-\Delta^\Dir$. These are respectively~\cite[Cor.~1.6]{NetrusovSafarov05} and~\cite[Cor.~1.9]{NetrusovSafarov05}.
\end{proof}
\end{proposition}

The following is a standard approximation result.
\begin{lemma}	\label{l:robin_to_tutto}
The following assertions hold:
\begin{itemize}
\item[$(a_\Dir)$] \label{i:l:robin_to_tutto:1}
$\mcE^\varrho(f,g)\to \mcE^\Dir(f,g)$ as~$\varrho\to\infty$ for every~$f,g\in H^1_0(\Omega)$;
\item[$(a_\Neu)$] \label{i:l:robin_to_tutto:2}
$\mcE^\varrho(f,g)\to \mcE^\Neu(f,g)$ as~$\varrho\to 0$ for every~$f,g\in H^1(\Omega)$;
\item[$(b_\Dir)$] \label{i:l:robin_to_tutto:3}
$P^\varrho_t\to P^\Dir_t$ as~$\varrho\to\infty$, strongly on~$L^2(\Omega)$, \purple{uniformly in~$t$ on~$\R^+_0$};
\item[$(b_\Neu)$] \label{i:l:robin_to_tutto:4}
$P^\varrho_t\to P^\Neu_t$ as~$\varrho\to 0$, strongly on~$L^2(\Omega)$, \purple{locally uniformly in~$t$ on~$\R^+_0$}.
\end{itemize}
\begin{proof}
The $(a)$-assertions are straightforward.
For both $(b)$-assertions, the existence of a limit in the strong operator topology on~$L^2(\Omega)$  follows combining the monotonicity in Lemma~\ref{l:Domination}, the uniform bound~$P^\varrho_t\leq \car$, and the main result in~\cite{behrndt_monotone2010}. Note that the uniformity in $t\in \R_0^+$ in item $(b_\Dir)$ follows from the lower bound $\lambda_0^{\varrho}> 0$, $\varrho> 0$, and the ordering of eigenvalues in Proposition \ref{p:Weyl}\ref{i:p:Weyl:2}.
The identification of the limit is a consequence of the identification of the corresponding limit Dirichlet forms as in the $(a)$-assertions.
\end{proof}
\end{lemma}

\subsubsection{\texorpdfstring{$\mcC_b$}{Cb}-Laplacians}\label{ss:CbLaplacians}
For~$p\in [1,\infty]$ we denote by~$\ttonde{\Delta^{\bd,p},\dom{\Delta^{\bd,p}}}$ the $L^p(\Omega)$-Laplacian corresponding to~$\Delta^{\bd,2}\eqdef\Delta^\bd$, defined as the $L^p(\Omega)$-closure of the operator
\begin{equation*}
\begin{cases}
\Delta^{\bd,p}u=\Delta^\bd u\comma \quad u\in \dom{\Delta^\bd}\cap L^p(\Omega): \Delta^\bd u\in L^p(\Omega) & \text{if } p \in [1,\infty)\comma
\\
(\Delta^{\bd,1})^* \quad \text{the adjoint of} \; \Delta^{\bd,1}
 & \text{if } p=\infty \fstop
\end{cases}
\end{equation*}
For details about the construction of~$\Delta^{\bd,p}$ on bounded Lipschitz domains, its consistency with~$\Delta^\bd$ on~$L^2(\Omega)\cap L^p(\Omega)$, and for the properties of the associated semigroups and resolvents, see e.g.:
\begin{itemize*}
\item[] \cite[\S1]{AreBen99} for~$\Delta^{\Dir,p}$ with~$p\in [1,\infty]$;
\item[] \cite[Thm.~1.4.1]{Dav89} for~$\Delta^{\varrho,p}$ with~$p\in[1,\infty)$, and~\cite{War06} for~$\Delta^{\varrho, \infty}$; 
\item[]  \cite[\S5]{wood_maximal2007} for~$\Delta^{\Neu,p}$ with $p \in [(3+\gamma)',3+\gamma]$ for some $\gamma>0$, where~$p'$ denotes the H\"older conjugate of~$p$.
\end{itemize*}

Finally, recall that~$\cC^\bd$ is either~$\cC(\overline\Omega)$ (for Neumann and Robin boundary conditions) or~$\cC_0(\Omega)$ (for Dirichlet boundary conditions). We denote by~$\ttonde{\Delta^{\bd,c},\dom{\Delta^{\bd,c}}}$ the~$\mcC^\bd$-Laplacian respectively defined by
\begin{align}
\label{eq:DomainCbLaplacian}
\dom{\Delta^{\bd,c}} \eqdef& \set{u\in\dom{\Delta^\bd}\cap\mcC^\bd : \Delta^\bd u \in\mcC^\bd } \comma \qquad \Delta^{\bd,c}u=\Delta^\bd u\fstop
\end{align}

\begin{theorem}
Fix~$\varrho\in \R^+$.
The operator~$\ttonde{\Delta^{\bd,c},\dom{\Delta^{\bd,c}}}$ is the part on~$\mcC^\bd$ of $\ttonde{\Delta^\bd,\dom{\Delta^\bd}}$.
The corresponding semigroup~$P^{\bd,c}_t$ is a
$\mcC^0$-semigroup
satisfying
\begin{align}\label{eq:t:ContPart}
P^{\bd,c}_t f = P^\bd_t f\comma \qquad f\in\mcC^\bd\comma  t>0\fstop
\end{align}
\begin{proof}
We separate different cases.
\paragraph{Dirichlet boundary conditions} Since~$\Omega$ is a bounded Lipschitz domain, it is Dirichlet regular; see e.g.~\cite[Dfn.~6.1.1]{AreBatHieNeu11} for the definition, and~\cite[Ex.~6.1.2b)]{AreBatHieNeu11} and references therein for the assertion.
Thus, all the results in~\cite{AreBen99} apply to our setting. In particular, by~\cite[Lem.~2.2b)]{AreBen99}, the operator~$\ttonde{\Delta^{\Dir,c},\dom{\Delta^{\Dir,c}}}$ is the part of~$\Delta^\Dir$ on~$\cC_0(\Omega)$.
The second assertion is~\cite[Thm.~2.3]{AreBen99}.

\paragraph{Neumann boundary conditions} By~\cite[Thm.~2.1(ii)]{fukushima_construction1996} the operator~$\ttonde{\Delta^{\Neu,c},\dom{\Delta^{\Neu,c}}}$ is the part of~$\Delta^\Neu$ on~$\cC(\overline\Omega)$.
For the second assertion see e.g.,~\cite[Prop.~3($\star$)]{biegert_neumann2003}.

\paragraph{Robin boundary conditions} This case is discussed in~\cite{War06} under the assumption that~$0<c\leq \varrho \in L^\infty(\sigma_\Omega)$ for some constant~$c$.
In particular, the first assertion holds as consequence of~\cite[Lem.~3.1]{War06}, cf.~\cite[p.~22]{War06}.
The validity of~\eqref{eq:t:ContPart} follows from the proof of~\cite[Thm.~3.2]{War06}, \purple{cf.~also~\cite[Thm.~4.3]{nittka_regularity2011}}.
\end{proof}
\end{theorem}

We conclude this part of the appendix with the proof of Proposition \ref{p:TestF}.	
\begin{proof}[Proof of Proposition \ref{p:TestF}]
Since~$\ttseq{\psi^\bd _n}_n$ is total in~$L^2(\Omega)$, and since~$\psi^\bd _n\in\cS^\bd (\Omega)$ by definition,~$\cS^\bd (\Omega)$ is dense in $L^2(\Omega)$.
Since~$\ttseq{\psi^{\bd, s}_n}_n$ is a \textsc{cons} for~$H^\bd _s(\Omega)$, identity operators form a chain of continuous  embeddings
\begin{align*}
\cS^\bd (\Omega)\hookrightarrow H^\bd _s(\Omega) \hookrightarrow H^\bd _r(\Omega) \hookrightarrow L^2(\Omega) \hookrightarrow H^\bd _{-r}(\Omega) \hookrightarrow H^\bd _{-s}(\Omega) \hookrightarrow \cS^\bd (\Omega)'\comma \qquad r\leq s\fstop
\end{align*}
For~$\delta\geq 0$ and integer~$k\geq 1$ let~$\id_{k,\delta}\colon H^\bd _{(k+\delta)d}(\Omega)\hookrightarrow H^\bd _{kd}(\Omega)$ be the identity operator
\begin{align*}
\id_{k,\delta}(\emparg)=\sum_{n=0}^\infty (1+\lambda_n^\bd )^{-\delta d/2}\tscalar{\psi^{\bd,(k+\delta)d}_n}{\emparg}_{H^\bd _{(k+\delta)d}(\Omega)} \psi^{\bd, kd}_n\fstop
\end{align*}
If~$\delta>1/2$, then~$\sum_{n=0}^\infty (1+\lambda_n^\bd )^{-\delta d}<\infty$ by Proposition~\ref{p:Weyl}, hence~$\id_{k,\delta}$ is a Hilbert--Schmidt operator, and thus~$\cS^\bd (\Omega)=\cap_k H^\bd _{kd}(\Omega)$ is a countably Hilbert nuclear space when endowed with the locally convex topology induced by the family of norms~$\tseq{\norm{\emparg}_{H^\bd _s(\Omega)}}_{s\in\R}$.
The rest of the proof of~\ref{i:p:TestF:1}--\ref{i:p:TestF:3} follows as in~\cite[Ex.~1.3.2,~p.~40]{kallianpur_xiong_1995}.
Since~$\cS^\bd (\Omega)$ is dense in~$L^2(\Omega)$, assertion~\ref{i:p:TestF:4} follows from~\ref{i:p:TestF:3} by Nelson's Theorem, e.g.~\cite[Thm.~X.49]{ReeSim75}.
		
\ref{i:p:TestF:5}--\ref{i:p:TestF:7} We separate different cases.

\paragraph{Dirichlet boundary conditions} 
Fix~$u\in \cS^\Dir$. By definition,~$u\in \dom{(\Delta^\Dir)^k}$ for all~$k$, hence by, e.g.,~\cite[Eqn.~(1.10)]{AreBen99}, we have that~$\Delta^\Dir u \in L^\infty(\Omega)$.
Since~$R_0(\Delta^\Dir) \cS^\Dir(\Omega)\subset \mcS^\Dir(\Omega)$, we further have that~$u\in L^\infty(\Omega)$, hence that $u\in \dom{\Delta^{\Dir,\infty}}$.
By \cite[Thm.~2.4(iii)]{AreBen99}, it follows that~$u \in \cC_0(\Omega)$.
Since~$\Delta^\Dir \mcS^\Dir(\Omega)\subset \mcS^\Dir(\Omega)$, iterating this argument shows~\eqref{eq:SBoundaryDir} for~$k>0$.
As a consequence of the latter we have as well that~$\mcS^\Dir(\Omega)=\cap_{k\geq 0} \dom{(\Delta^{\Dir,c})^k}\defeq \cS^{\Dir,c}(\Omega)$.
Together with~\eqref{eq:DomainCbLaplacian} shows~\ref{i:p:TestF:6} for~$\Delta^\Dir$.
Since~$ \cC^\infty_c(\Omega)\subset \cS^\Dir(\Omega)\subset \cC_0(\Omega)$, $\mcS^\Dir(\Omega)$ is dense in~$\cC_0(\Omega)$. Finally, since~$\mcS^\Dir(\Omega)=\mcS^{\Dir,c}(\Omega)$, we have that~$P^{\Dir,c}_t \mcS^\Dir(\Omega)\subset \mcS^\Dir(\Omega)$, and therefore that~$\cS^\Dir(\Omega)$ is a core for both~$\Delta^{\Dir,c}$ and~$\Delta^\Dir$ by~\cite[Thm.~X.49]{ReeSim75}.

\paragraph{Neumann and Robin boundary conditions} 
Let~$\bd$ be either~$\Neu$ or~$\varrho$, and fix~$u\in\mcS^\bd(\Omega)$. By definition,~$u\in \dom{(1+\Delta^\bd)^k}$ for all~$k$.
By the first displayed inequality in the proof of~\cite[Thm.~4.3]{nittka_regularity2011}, we have that~$(R_1^\bd)^k L^2(\Omega)\subset \mcC(\overline\Omega)\defeq \mcC^\bd$ for~$k\gg 1$, hence that~$\mcS^\bd\subset \mcC^\bd$.
By~\eqref{eq:DomainCbLaplacian} we have therefore that~$\mcS^\bd(\Omega)=\cap_k \dom{(\Delta^{\bd,c})^k}\defeq \mcS^{\bd,c}(\Omega)$.
The rest of the proof follows as in the case~$\bd=\Dir$.
\end{proof}

\subsubsection{On the generality of Lipschitz domains}\label{sec:generality-Lipschitz}
Some of the properties of Laplacians discussed in this section may fail if $\Omega$ is non-Lipschitz.
For instance, there exist \emph{bounded} non-Lipschitz domains~$\Omega$ so that
\begin{enumerate}[$(a)$]
	\item the essential spectrum of~$\Delta^\Neu_p$ is non-empty,~\cite{HemSecSim91}, and may depend on~$p$,~\cite{kunstmann_lp2002} (for the importance of the $L^p$-spectral properties in applications see, e.g., the introduction to~\cite{JiWeb15}); 
	\item consequently, (the restriction to~$L^2(\Omega)$ of)~$P^{\Neu,p}_t$ may depend on~$p$, and~$P^{\Neu,c}_t$ may be different from the continuous part of~$P^{\Neu}_t$.
\end{enumerate}
These two properties (emptiness of the essential spectrum and  consistency of the $L^2$- and $\cC^\bd$-semigroups) on bounded Lipschitz domains are crucially used, e.g.,  in the construction and characterization of the space $\mcS^\bd(\Omega)$ of test functions (Prop.~\ref{p:TestF}).
If we allow $\Omega$  to be unbounded, then we may choose, e.g., the horn-shaped domain
\begin{equation*}
	\Omega\eqdef \set{(x,y)\in\R^2 \ : \  x>0 \comma  -e^{-x}<y<e^{-x}}\fstop
\end{equation*}

Further say that~$u$ is a \emph{weak} solution of~$\eqref{eq:HeatEquation}$ with~$\vartheta\equiv 0$ if
\begin{equation}\label{eq:WHeat}\tag{$\mathrm{H}^{\mathrm{weak}}_{\Dir,T}$}
	\begin{aligned}
		u\in& L^\infty\ttonde{(0,T);L^2_\loc(\Omega)}\cap L^2\ttonde{(0,T); H^1_\loc(\Omega)}\comma
		\\
		\iint_{\Omega\times (0,T)}& \ttonde{-u\, \partial_t v+\nabla u\cdot \nabla v} \,\diff\mu_\Omega\,\diff t=0 \comma \qquad v\in \cC^\infty_0\ttonde{\Omega\times(0,T)} \fstop
	\end{aligned}
\end{equation}
Again on unbounded non-Lipschitz domains,
\begin{enumerate}[$(a)$]\setcounter{enumi}{2}
	\item uniqueness of solutions may fail for~\eqref{eq:WHeat}, see~\cite{Mur96}.
\end{enumerate}

\purple{\subsection{Equicontinuity of semigroups}\label{sec:proof-equicontinuity}
This section is devoted to the proof of Proposition \ref{pr:equi_semi_disc}.}
\rosso{Observe that, since 
	\begin{equation*}
		\lim_{\delta\downarrow 0} \sup_{\beta\in \R}\sup_{t\geq 0}  \sup_{\substack{x,y \in \Omega_\eps\\\abs{x-y}<\delta}} \abs{P^{\eps,\beta}_t\Pi_\eps f(x)-P^{\eps,\beta}_t \Pi_\eps f(y)}=0\comma\qquad \eps \in (0,1)\comma
	\end{equation*} the claims in \eqref{eq:pr:equi_semi_disc-robin} and \eqref{eq:pr:equi_semi_disc} may be equivalently restated with $\limsup_{\eps\downarrow 0}$ replacing $\sup_{\eps\in (0,1)}$.}

\begin{proof}[Proof of Proposition \ref{pr:equi_semi_disc}] 
	Throughout the proof, we write $\mbfE_x=\mbfE_x^{\eps,\infty}$ and omit the specification of $\Pi_\eps$. For~$t >0$, by \eqref{eq:feynman-kac} and the Taylor expansion of the exponential function, 
	\begin{align}
		\label{eq:equi1}
		P^{\eps,\beta}_t  f(x) = P^{\eps,\infty}_t f(x)+\sum_{k=1}^\infty \frac{(-\eps^{\beta-1})^k}{k!}\, 		\mbfE_x\quadre{f(X^{\eps,\infty}_t)\tonde{\int_0^tV_\eps(X^{\eps,\infty}_r)\, \dd r}^k} \comma
	\end{align}
	which, by expanding the powers of the integral and by the Markov property, reads as
	\begin{equation}\label{eq:peps_vs_pinf}
		P^{\eps,\beta}_t f(x)= P^{\eps,\infty}_t f(x)- \int_0^t \sum_{z\in \partial\Omega_\eps} p^{\eps,\infty}_r(x,z) \tonde{\eps^{\beta-1}\,V_\eps(z)\, P^{\eps,\beta}_{t-r}f(z)} \dd r\fstop
	\end{equation}
	Hence, 
	\begin{equation}
		\label{eq:equi0}
		\begin{aligned}
			&\abs{P^{\eps,\beta}_tf(x)-P^{\eps,\beta}_t f(y)}\leq \abs{P^{\eps,\infty}_t f(x)-P^{\eps,\infty}_t f(y)}
			\\
			&+ \abs{\int_0^t 
				\sum_{z\in \partial\Omega_\eps} \ttonde{p^{\eps,\infty}_r(x,z)-p^{\eps,\infty}_r(y,z)} \tonde{\eps^{\beta-1}\,V_\eps(z)\, P^{\eps,\beta}_{t-r}f(z)}\dd r}\fstop
		\end{aligned}
	\end{equation}
	Of the two terms on the right-hand side above, the desired equicontinuity of the first one is known, see~\eqref{eq:equicontinuity_neumann_eps}. 
	In order to prove the equicontinuity of the second term in \eqref{eq:equi0},  for all~$t_* \in (0,t\wedge 1]$, \rosso{and $x,y \in \Omega_\eps$, $|x-y|<\delta$, we have that}
	\begin{align}
		\label{eq:proof_equicon_brutta}
		\begin{aligned}
			&\abs{\int_0^t \sum_{z\in \partial\Omega_\eps} \tonde{p^{\eps,\infty}_r(x,z)-p^{\eps,\infty}_r(y,z)} \tonde{\eps^{\beta-1}\, V_\eps(z)\, P^{\eps,\beta}_{t-r}f(z)}\dd r}
			\\
			&\quad\le 2\norm{f}_{\cC(\overline\Omega)}\sup_{x\in \Omega_\eps}\int_0^{t_*} \sum_{z\in \partial\Omega_\eps}p^{\eps,\infty}_r(x,z)\, \eps^{\beta-1}\,V_\eps(z)\, P^{\eps,\beta}_{t-r}\car_{\Omega_\eps}(z)\,\dd r 
			\\
			&\qquad+ C\norm{f}_{\cC(\overline\Omega)} t_*^{-(b+d)/2}\rosso{\delta}^a\int_{t_*}^t \tonde{\eps^d\sum_{z\in \partial\Omega_\eps} \eps^{\beta-1}\,V_\eps(z)\, P^{\eps,\beta}_{t-r}\car_{\Omega_\eps}(z)}\dd r	\comma
	\\
		&\quad 	\rosso{
		=: I_{\eps,\beta,t_*,t} + J_{\eps,\beta, \delta,t_*,t}
			\comma
		}
		\end{aligned}
	\end{align}
where to get the second expression we used \eqref{eq:holder}. 
	(Recall that the positive constants $a, b $ and $C$ depend only on $\Omega\subset\R^d$). 
	Let us estimate the second term on the right-hand side of \eqref{eq:proof_equicon_brutta}: by the symmetry of  $P^{\eps,\beta}_{t-r}$ in $L^2(\Omega_\eps)$ and \eqref{eq:feynman-kac_dt}, 
		\begin{align*}
			&\int_{t_*}^t \tonde{\eps^d\sum_{z\in \partial\Omega_\eps} \eps^{\beta-1}\,V_\eps(z)\, P^{\eps,\beta}_{t-r}\car_{\Omega_\eps}(z)}\dd r\\
			&\quad = \eps^d \sum_{x\in \Omega_\eps}\int_0^{t-t_*} \mbfE_x\quadre{\eps^{\beta-1}\, V_\eps(X^{\eps,\infty}_r)\exp\tonde{-\eps^{\beta-1}\int_0^r V_\eps\tonde{X^{\eps,\infty}_s}\, \dd s}}\dd r\le \mu_\eps(\Omega_\eps)\fstop
		\end{align*}
It is therefore clear that, for every $t_*>0$, the sequence $\sup_{\eps,\beta,t} J_{\eps,\beta,\delta,t_*,t} \to 0$ as $\delta \to 0$. We are therefore left to show that
	\begin{align}	\label{eq:final_equi}
		\lim_{t_* \downarrow 0}
		\limsup_{\eps \downarrow 0}
			\sup_{t \geq t_0}
			\sup_{\beta \in \R} 
		I_{\eps,\beta,t_*,t} = 0 
		\comma \quad t_0 >0 \fstop 
	\end{align}
	For this purpose, we fix $\eps\in(0,1)$, $t\geq 0$, and $\varrho \geq 1$, and split the supremum over $\beta \in \R$ as
	\begin{align}	\label{eq:two_suprema}
		\sup_{\beta \in \R}
			I_{\eps,\beta,t_*,t}
		=
		\max
		\left\{
			\sup_{\substack{\beta \in \R\\  \eps^{\beta-1}\leq \varrho}}
				I_{\eps,\beta,t_*,t}
		\, , \, 
			\sup_{\substack{\beta \in \R\\ \eps^{\beta-1}> \varrho}}
						I_{\eps,\beta,t_*,t}
		\right\}
			\, .
	\end{align}
	By $P^{\eps,\beta}_{t-r}\car_{\Omega_\eps}\le 1$ and \eqref{eq:HKbdV}, we can then estimate the first supremum in \eqref{eq:two_suprema} as
	\begin{align*}
		\sup_{t \geq 0}
		\sup_{\substack{\beta \in \R\\  \eps^{\beta-1}\leq \varrho}}
			I_{\eps,\beta,t_*,t} 
		\le C\, \varrho\, t_*^{1/2}\fstop
	\end{align*}
Note that, since, for $\beta \ge 1$, we have $\eps^{\beta-1}\leq 1\le \varrho$ for all  $\eps\in (0,1)$, the latter estimate and \eqref{eq:proof_equicon_brutta} with $t=t_*$ suffice to prove \eqref{eq:pr:equi_semi_disc-robin} (note that $J_{\eps,\beta, \delta,t,t}=0$). Hence, thanks to	\eqref{eq:pr:equi_semi_disc-robin}, all the arguments in \S\S\ref{sec:semigroup-conv-neumann}--\ref{sec:semigroup-conv-robin} carry through,  ensuring the validity of  Theorem \ref{t:MainSemigroups} for all $\beta \ge 1$, as well as
\begin{align}\label{eq:conv-robin-rho}
	\lim_{\eps\downarrow 0}\sup_{t\ge 0}\sup_{x\in \Omega_\eps}\abs{\mbfE_x\quadre{\exp\tonde{-\varrho\int_0^t V_\eps(X^{\eps,\infty}_s)\, \dd s}}-P^\varrho_t\car_\Omega(x)} =0\fstop
\end{align}
(The above claim holds with the supremum over $t \ge 0$ by Thm\ \ref{t:MainSemigroups} and Cor.\ \ref{c:spectral_bound_DR}.)

\rosso{Now we turn our attention to the regime $\beta<1$ and to the second supremum on the right-hand side of \eqref{eq:two_suprema}.
}
\rosso{Note that by monotonicity of $t \mapsto P^{\eps,\beta}_{t-r}\car_{\Omega_\eps}(z)$, we have $I_{\eps,\beta,t_*,t} \leq I_{\eps,\beta,t_*,t_0}$ for $t\geq t_0$}.
By the Markov property, \eqref{eq:feynman-kac}, and Tonelli's theorem, we get, for all $x\in \Omega_\eps$ and \rosso{$t_* \in (0,t_0 \wedge 1]$},
\begin{align*}
	&	\int_0^{t_*} \sum_{z\in \partial\Omega_\eps}p^{\eps,\infty}_r(x,z)\, \eps^{\beta-1}\,V_\eps(z)\, P^{\eps,\beta}_{\rosso{t_0}-r}\car_{\Omega_\eps}(z)\,\dd r\\
	&=
	\int_0^{t_*} \mbfE_x\quadre{\eps^{\beta-1}\, V_\eps(X^{\eps,\infty}_r)\, \mbfE_{X^{\eps,\infty}_r}\quadre{\exp\tonde{-\eps^{\beta-1}\int_0^{\rosso{t_0}-r}V_\eps(X^{\eps,\infty}_s)\, \dd s}}}\dd r\\
	&= \int_0^{t_*} \mbfE_x\quadre{\eps^{\beta-1}\, V_\eps(X^{\eps,\infty}_r)\exp\tonde{-\eps^{\beta-1}\int_r^{\rosso{t_0}} V_\eps(X^{\eps,\infty}_s)\, \dd s}}\dd r\\
	&= \mbfE_x\quadre{\set{\int_0^{t_*}\frac{\dd}{\dd r}\exp\tonde{\eps^{\beta-1}\int_0^r V_\eps(X^{\eps,\infty}_s)\, \dd s}\dd r}\exp\tonde{-\eps^{\beta-1}\int_0^{\rosso{t_0}} V_\eps(X^{\eps,\infty}_s)\, \dd s}}\\
	&= \mbfE_x\quadre{\exp\tonde{-\eps^{\beta-1}\int_{t_*}^{\rosso{t_0}} V_\eps(X^{\eps,\infty}_s)\, \dd s}\tonde{1-\exp\tonde{-\eps^{\beta-1}\int_0^{t_*}V_\eps(X^{\eps,\infty}_s)\,\dd s}}}
	\fstop
\end{align*}
%

For  $\gamma >0$, define $\Omega_\eps^\gamma\eqdef \set{x\in \Omega_\eps: \dist(x,\partial\Omega)\le\gamma}$. Then, by the exit-time estimate \eqref{eq:exit-time} and the above identity, we get, for all 
\rosso{
$t_*\in (0,t_0\wedge 1)$, 
for some $C=C(f)<\infty$,
\begin{align*}
	\begin{aligned}
	&\sup_{\substack{\beta\in \R\\
				\eps^{\beta-1}>\varrho}}
			I_{\eps,\beta,t_*,t_0} 
	 \le  C 
	 \left(
	 	\sup_{x\in \Omega_\eps^\gamma} \mbfE_x\quadre{\exp\tonde{-\varrho\int_{t_*}^{t_0} V_\eps(X^{\eps,\infty}_s)\, \dd s}} +   \exp\tonde{-\frac{C'\gamma}{t_*^{1/2} \vee \eps}}
	 \right)\fstop
	\end{aligned}
\end{align*} 
Taking first the limsup in $\eps \to 0$ and then in $t_*\to 0$, we have, for all $\varrho\geq 1$, $\gamma >0$  and $t_0>0$, 
\begin{align}
\label{eq:I<K}
	\limsup_{t_* \downarrow 0}
	\limsup_{\eps \downarrow 0}
		\sup_{\substack{\beta\in \R\\
				\eps^{\beta-1}>\varrho}}
		I_{\eps,\beta,t_*,t_0} 
	\leq 
		\limsup_{t_* \downarrow 0}
			\limsup_{\eps \downarrow 0}
		K_{\eps,t_*,\varrho,\gamma,t_0}
	\comma
\end{align}
where we defined the quantity
\begin{align*}
	K_{\eps,t_*,\varrho,\gamma,t_0}
		:=
	\sup_{x\in \Omega_\eps^\gamma} \mbfE_x\quadre{\exp\tonde{-\varrho\int_{t_*}^{t_0} V_\eps(X^{\eps,\infty}_s)\, \dd s}}
		\fstop 	
\end{align*}
}We now estimate this quantity. By the Markov property and the triangle inequality, we have
\rosso{
\begin{align}
\nonumber
	K_{\eps,t_*,\varrho,\gamma,t_0}
		\le &
	\sup_{x\in \Omega_\eps^\gamma}\sum_{y\in \Omega_\eps} p^{\eps,\infty}_{t_*}(x,y)\, P^\varrho_{t_0-t_*}\car_\Omega(y)
\\
\label{eq:fin_1}
	&+ \sup_{y \in \Omega_\eps}\abs{ \mbfE_y\quadre{\exp\tonde{-\varrho\int_0^{t_0-t_*}V_\eps(X^{\eps,\infty}_s)\, \dd s}} -P^\varrho_{t_0-t_*}\car_\Omega(y)}
		\fstop 
\end{align}
For fixed $t_*>0$, $\varrho \geq 1$, and $\gamma>0$, by taking $\eps \to 0$, the local CLT for $p_{t_*}^{\eps,\infty}(x,y)$ (see Rmk.~\ref{rmk:dyadic_lclt}) and the uniform convergence of $\varrho$-Robin semigroups \eqref{eq:conv-robin-rho} yield
\begin{align}
\label{eq:fin_2}
	\limsup_{\eps \downarrow 0}
		K_{\eps,t_*,\varrho,\gamma,t_0}
	\leq 
		\sup_{x \in \Omega^{2 \gamma}}
		\int_{\Omega}
			p_{t_*}(x,y) 
			P_{t_0-t_*}^\varrho \car_\Omega(y)\,	
				\dd y
		\fstop 
\end{align}
We now observe that the strong continuity of the $\varrho$-Robin semigroup implies 
\begin{align}
\label{eq:fin_3}
	\limsup_{t_* \downarrow 0}
		\sup_{x \in \Omega^{2 \gamma}}
		\int_{\Omega}
			p_{t_*}(x,y) 
			P_{t_0-t_*}^\varrho \car_\Omega(y)\,
				\dd y
		=
	\sup_{x \in \Omega^{2 \gamma}}
		P_{t_0}^\varrho \car_\Omega(x) 
			\fstop
\end{align}
We continue the estimate for a given $x \in \Omega^{2\gamma}$, for every $s \in (0,t_0/2)$, as
\begin{align}
\nonumber	
	P^\varrho_{t_0}\car_\Omega(x)
		&\le 	
	\int_\Omega p^\varrho_s(x,y)\, \abs{\ttonde{P^\varrho_{t_0-s}-P^\Dir_{t_0-s}}\car_\Omega(y)} {\rm d}y
		+
	P^\varrho_s P^\Dir_{t_0-s}\car_\Omega(x)
	\\
\nonumber
		&\le 
	\sqrt{p^\varrho_{2s}(x,x)}\sup_{r\ge 0}\norm{P^\varrho_r\car_\Omega-P^\Dir_r\car_\Omega}_{L^2(\Omega)}
		+ 
	P^\varrho_s P^\Dir_{t_0-s}\car_\Omega(x)
	\\
\label{eq:fin_4}
		&\le 	
	\sqrt{p^\Neu_{2s}(x,x)}\sup_{r\ge 0}\norm{P^\varrho_r\car_\Omega-P^\Dir_r\car_\Omega}_{L^2(\Omega)}
		+ 
	P^\Neu_s P^\Dir_{t_0-s}\car_\Omega(x)\comma
\end{align}
where in the second-to-last step we used Cauchy-Schwartz inequality, while	 in the  last step we employed the monotonicity of semigroups (Lemma~\ref{l:Domination}).
In conclusion,  Lemma~\ref{l:robin_to_tutto} yields
\begin{align}
\limsup_{\varrho\to \infty}	\sup_{x\in \Omega^{2\gamma}} \sqrt{p^\Neu_{2s}(x,x)}\sup_{r\ge 0}\norm{P^\varrho_r\car_\Omega-P^\Dir_r\car_\Omega}_{L^2(\Omega)}=0\comma\qquad s\in (0,t_0/2)\comma
\end{align}
while $P^\Dir_{t_0}\car_\Omega \in \mcC_0(\Omega)$ ensures that
\begin{align}
	\lim_{s\downarrow 0}\limsup_{\gamma\downarrow 0} 	
		\sup_{x\in \Omega^{2\gamma}} P^\Neu_s P^\Dir_{t_0-s}\car_\Omega(x) = \sup_{x\in \partial\Omega} P^\Dir_{t_0}\car_\Omega(x)=0\fstop
\end{align}
Collecting all the estimates in \eqref{eq:fin_1}, \eqref{eq:fin_2}, \eqref{eq:fin_3}, and \eqref{eq:fin_4}, we obtain that 
\begin{align*}
	\lim_{\gamma \downarrow 0}
	\limsup_{\varrho \to \infty}
	\limsup_{t_* \downarrow 0}
	\limsup_{\eps \downarrow 0}
		K_{\eps, t_*, \varrho, \gamma, t_0}
	= 0 
		\comma 
	\quad t_0 > 0 
		\comma 
\end{align*}
which, together with \eqref{eq:I<K}, shows the sought claim \eqref{eq:final_equi}, and thus concludes the proof of the proposition.}
\end{proof}

\subsection{Equivalence of convergences}\label{sec:proof-Equivalences}
\purple{Theorem \ref{t:Equivalence} --- which is instrumental to the proof of Theorem \ref{t:MainSemigroups} --- is the main result of this section. In what follows,}
we make use of the general framework of~\cite{KuwShi03}, and in particular of the notion of convergence of Hilbert spaces, bounded operators, and quadratic forms, in the sense of~\cite{KuwShi03} with respect to the pair~$(\mcC^\bd,\Pi_\eps)$. 

\subsubsection{Definitions of convergences}\label{ss:KS}
We start by recalling the main definitions of convergences, specialized to our setting.

\paragraph{Kuwae--Shioya convergence}
We refer to~\cite{KuwShi03} for the details of this construction.
Everywhere in the following, we assume that~$\gamma\in (0,1)$.
Let~$\seq{f_\eps}_\eps$, with~$f_\eps\in L^2(\Omega_\eps)$, and~$f\in L^2(\Omega)$.

\begin{definition}[Kuwae--Shioya convergences]\label{def:KS_convergence}
	We say that~$\seq{f_\eps}_\eps$
	\begin{itemize}
		\item (\emph{strongly}) \emph{KS-converges} (\emph{with respect to the pair~$(\mcC^\bd,\Pi_\eps)$}) to~$f$ if there exists~$\ttseq{\tilde f_\gamma}_\gamma\subset \mcC^\bd$ so that
		\begin{align}\label{eq:KSStrong}
			\lim_{\gamma\downarrow 0} \tnorm{\tilde f_\gamma-f}_{L^2(\Omega)}\comma \qquad \lim_{\gamma\downarrow 0} \limsup_{\eps\downarrow 0} \tnorm{\Pi_\eps\tilde f_\gamma-f_\eps}_{L^2(\Omega_\eps)}=0 \fstop
		\end{align}
		Write~$f_\eps\xrightarrow{\mcC^\bd,\Pi_\eps}f$.
		
		\item \emph{weakly KS-converges} to~$f$ if~$\lim_{\eps\downarrow 0}\scalar{f_\eps}{g_\eps}_{L^2(\Omega_\eps)}= \scalar{f}{g}_{L^2(\Omega)}$ for every~$\seq{g_\eps}_\eps$, with~$g_\eps\in L^2(\Omega_\eps)$, strongly KS-convergent to~$g\in L^2(\Omega)$.
		Write~$f_\eps\xrightharpoonup{\mcC^\bd,\Pi_\eps} f$.
	\end{itemize}
	
	Further let~$\seq{B_\eps}_\eps$, with~$B_\eps\colon L^2(\Omega_\eps)\to L^2(\Omega_\eps)$, and~$B\colon L^2(\Omega)\to L^2(\Omega)$ be bounded operators.
	We say that~$\seq{B_\eps}_\eps$
	\begin{itemize}
		\item (\emph{strongly}) \emph{KS-converges} to~$f$ if~$B_\eps f_\eps \xrightarrow{\mcC^\bd,\Pi_\eps} Bf$ for every~$f_\eps \xrightarrow{\mcC^\bd,\Pi_\eps} f$. Write~$B_\eps \xrightarrow{\mcC^\bd,\Pi_\eps} B$;
		\item \emph{compactly KS-converges} to~$f$ if~$B_\eps f_\eps \xrightarrow{\mcC^\bd,\Pi_\eps} Bf$ for every~$f_\eps \xrightharpoonup{\mcC^\bd,\Pi_\eps} f$. Write~$B_\eps \xRightarrow{\mcC^\bd,\Pi_\eps} B$.
	\end{itemize}
	
	Finally, let~$\seq{F_\eps}_\eps$, with~$F_\eps\colon L^2(\Omega_\eps)\to [0,+\infty]$, and~$F\colon L^2(\Omega)\to [0,+\infty]$ be non-negative functionals.
	We say that~$F_\eps$
	\begin{itemize}
		\item \emph{KS-Mosco-} (\emph{KSM-})\emph{converges} to~$F$ if
		\begin{gather*}
			f_\eps \xrightharpoonup{\mcC^\bd,\Pi_\eps} f \implies F(f)\leq \liminf_{\eps\downarrow 0} F_\eps(f_\eps)\semicolon
			\\
			\forall f\in L^2(\Omega) \quad \exists \seq{f_\eps}_\eps : f_\eps \xrightarrow{\mcC^\bd,\Pi_\eps} f \quad \text{and} \quad \lim_{\eps\downarrow 0} F_\eps(f_\eps)=F(f) \fstop
		\end{gather*}
		Since no confusion may arise, we write as well~$F_\eps\xrightarrow{\mcC^\bd,\Pi_\eps} F$.
	\end{itemize}
\end{definition}

\paragraph{Direct convergence} For~$x\in \R^d$ denote now by~$Q_\eps(x)$ the $d$-hypercube $x+[0,\eps)^d$ of side-length~$\eps$ with lexicographically lowest corner~$x$.
Further define an isometric operator~$\iota_\eps\colon L^2(\Omega_\eps)\to L^2(\Omega)$ by
\begin{align*}
	\iota_\eps\colon f_\eps\longmapsto \sum_{x\in \Omega_\eps} f_\eps(x) \, \car_{Q_\eps(x)\cap\Omega} \comma
\end{align*}
and denote by~$\mcQ_\eps\colon L^2(\Omega)\to L^2(\Omega_\eps)$ its adjoint operator.
For~$F_\eps\colon L^2(\Omega_\eps)\to [0,+\infty]$ set now
\begin{align*}
	\ttonde{{\iota_\eps}_* F_\eps}(f)\eqdef \begin{cases}
		F_\eps(\mcQ_\eps f) & \text{if } f=\iota_\eps\mcQ_\eps f
		\\
		+\infty & \text{otherwise}
	\end{cases}\fstop
\end{align*}

The following notion of convergence describes the `direct-limit' point of view for discretizations, commonly adopted in the literature, see e.g.~\cite{AliCic04, forkert2020evolutionary}.

\begin{definition}[Direct convergences]
	Let~$\seq{f_\eps}_\eps$, with~$f_\eps\in L^2(\Omega_\eps)$, and~$f\in L^2(\Omega)$.
	
	We say that~$\seq{f_\eps}_\eps$
	\begin{itemize}
		\item (\emph{strongly}) \emph{d-converges} to~$f$ if~$\lim_{\eps\downarrow 0}\norm{\iota_\eps f_\eps-f}_{L^2(\Omega)}=0$. Write~$f_\eps\xrightarrow{\iota_\eps}f$.
	\end{itemize}
	
	Further, let~$\seq{F_\eps}_\eps$, with~$F_\eps\colon L^2(\Omega_\eps)\to [0,+\infty]$, and~$F\colon L^2(\Omega)\to [0,+\infty]$ be non-negative functionals. We say that~$F_\eps$
	\begin{itemize}
		\item \emph{Mosco d-converges} to~$F$ if~${\iota_\eps}_* F_\eps$ Mosco converges to~$F$ on $L^2(\Omega)$ in the usual sense. Write~$F_\eps\xrightarrow{\iota_\eps}F$.
	\end{itemize}
\end{definition}

Straightforward manipulations of the above definitions show that these notions of convergence are in fact equivalent.
\begin{proposition}\label{p:ACtoKS}
	With the notation established above, we have that
	\begin{itemize}
		\item $f_\eps \xrightarrow{\mcC^\bd,\Pi_\eps} f$ if and only if~$f_\eps \xrightarrow{\iota_\eps} f$;
		\item $F_\eps \xrightarrow{\mcC^\bd,\Pi_\eps} F$ if and only if~$F_\eps \xrightarrow{\iota_\eps} F$.
	\end{itemize}
\end{proposition}

\subsubsection{Main result}\label{sss:EquivalenceConvergences}
Let~$\ttonde{B,\dom{B}}$ be an unbounded operator on some Banach space.
Recall that a linear space~$\mcS\subset \dom{B}$ is a \emph{core} for~$\ttonde{B,\dom{B}}$ if it is dense in~$\dom{B}$ in the $B$-graph norm~\purple{$\norm{v}_B\eqdef \norm{v}+\norm{Bv}$ for~$v\in\dom{B}$}.

The proof of the following theorem is the content of	 \S\ref{sec:proof-equivalences} below.
\begin{theorem}[Equivalence of convergences]\label{t:Equivalence}
Let~$\mcS^{\bd,c}$ be any core for~$\ttonde{\Delta^{\bd,c},\dom{\Delta^{\bd,c}}}$. The following assertions are equivalent:
\begin{enumerate}[$({a}_1)$]
\item\label{i:t:Equivalence:1} the assertion of Theorem~\ref{t:MainSemigroups};
\item\label{i:t:Equivalence:2} convergence of~$\mcC^\bd$-semigroups: $P^{\eps,\beta}_t \Pi_\eps f\to P^{\bd,c}_t f$ for each~$f\in \cC^\bd$ for each~$t\geq0$;
\item\label{i:t:Equivalence:3} graph-convergence of~$\mcC^\bd$-generators: for each~$f\in \mcS^{\bd,c}$ there exists~$f_\eps\in \dom{A^{\eps,\beta}}$ with~$f_\eps\to f$ and~$A^{\eps,\beta} f_\eps \to \Delta^{\bd,c} f$;
\item\label{i:t:Equivalence:4} convergence of forms: $\mcE^{\eps,\beta}\to \mcE^\bd$ compactly;
\item\label{i:t:Equivalence:5} convergence of $L^2$-semigroups: $P^{\eps,\beta}_t \to P^\bd_t$ compactly for every~$t> 0$;
\item\label{i:t:Equivalence:6} convergence of $L^2$-resolvents $R^{\eps,\beta}_\zeta \to R^\bd_\zeta$ compactly for either some or every~$\zeta>0$.
\end{enumerate}
Finally, if either of the above convergences holds, then
\begin{enumerate}[$({b}_1)$]
\item\label{i:t:Equivalence:7} graph-convergence of~$L^2$-generators: the strong Kuwae--Shioya-graph limit of~$A^{\eps,\beta}$ coincides with~$\Delta^\bd$;
\item\label{i:t:Equivalence:8} spectral convergence: for every $n \in \N_0$, denote by~$\lambda^{\eps,\beta}_n$ the $n^{\text{th}}$ eigenvalue of~$-A^{\eps,\beta}$ indexed with multiplicity, and set~$\lambda^{\eps,\beta}_n\eqdef +\infty$ for~$n\geq \dim L^2(\Omega_\eps)$.
Further let~$\lambda^\bd_n$ be the $n^{\textrm{th}}$~eigenvalue of~$\ttonde{-\Delta^\bd,\dom{-\Delta^\bd}}$ indexed with multiplicity. Then,
\begin{align*}
\lim_{\eps\downarrow 0} \lambda^{\eps,\beta}_n=\lambda^\bd_n\comma \qquad n\in \N_0\fstop
\end{align*}
\end{enumerate}
\end{theorem}
For an example of a suitable core~$\mcS^\bd=\mcS^{\bd,c}$, see Proposition~\ref{p:TestF}. Finally,  since Theorem~\ref{t:MainSemigroups} establishes~\ref{i:t:Equivalence:1} above, combining it with Theorem~\ref{t:Equivalence} shows that all assertions~\ref{i:t:Equivalence:1}--\ref{i:t:Equivalence:6}, and~\ref{i:t:Equivalence:7},~\ref{i:t:Equivalence:8} hold.

\subsubsection{Proof of Theorem \ref{t:Equivalence}}\label{sec:proof-equivalences}
The equivalence of~\ref{i:t:Equivalence:1}--\ref{i:t:Equivalence:3} is~\cite[Thm.~I.6.1]{ethier_kurtz_1986_Markov}.
The equivalence of~\ref{i:t:Equivalence:4}--\ref{i:t:Equivalence:6} is~\cite[Thm.~2.4]{KuwShi03}.
The implication~\ref{i:t:Equivalence:4} $\implies$ \ref{i:t:Equivalence:7} follows from~\cite[Thm.~2.5]{KuwShi03}.
Since~$L^2(\Omega_\eps)$ is finite-dimensional, the operator~$R^{\eps,\beta}_\zeta\eqdef\ttonde{\zeta-A^{\eps,\beta}}^{-1}$ is compact for every~$\zeta>0$.
As a consequence, the implication~\ref{i:t:Equivalence:4} $\implies$ \ref{i:t:Equivalence:8} follows from~\cite[Cor.~2.5]{KuwShi03}.

Thus, it remains to show the following:

\begin{proposition}
	Assertion~\ref{i:t:Equivalence:2} is equivalent to assertion~\ref{i:t:Equivalence:5}.
	\begin{proof}
		Concerning the implication~\ref{i:t:Equivalence:2}$\implies$\ref{i:t:Equivalence:5}, we show that~\ref{i:t:Equivalence:2} implies the strong KS-convergence~$P^{\eps,\beta}_t\xrightarrow{\mcC^\bd,\Pi_\eps} P^\partial_t$, and that the forms~$\ttonde{\mcE^{\eps,\beta},\dom{\mcE^{\eps,\beta}}}$ verify the definition~\cite[Dfn.~3.12]{KuwShi03} of \emph{asymptotic compactness}.
		The conclusion then follows from~\cite[Thm.~2.4]{KuwShi03}.
		
		\paragraph{Strong KS-convergence of semigroups}
		Let~$\ttseq{f_\eps}_\eps$, with~$f_\eps\in L^2(\Omega_\eps)$, be strongly KS-convergent to~$f\in L^2(\Omega)$.
		We show that~$P^{\eps,\beta}_t f_\eps\xrightarrow{\mcC^\bd,\Pi_\eps} P^\bd_t f$.
		By assumption, there exists~$\ttseq{\tilde f_\gamma}_\gamma$ as in~\eqref{eq:KSStrong}.
		Set~$\tilde g_\gamma\eqdef P^\bd_t \tilde f_\gamma$.
		Since~$\tilde f_\gamma\in\mcC^\bd$ by definition, we have that~$\tilde g_\gamma = P^{\bd,c}_t \tilde f_\gamma\in \mcC^\bd$ by~\eqref{eq:t:ContPart}. 
		Thus,~$\Pi_\eps \tilde g_\gamma$ is well-defined, and we may estimate 
		\begin{align*}
			&\tnorm{\Pi_\eps \tilde g_\gamma-P^{\eps,\beta}_t f_\eps}_{L^2(\Omega_\eps)}\\
			&\quad \le \tnorm{\Pi_\eps P^{\bd,c}_t \tilde f_\gamma-P^{\eps,\beta}_t \Pi_\eps \tilde f_\gamma}_{L^2(\Omega_\eps)} + \tnorm{P^{\eps,\beta}_t  \Pi_\eps \tilde f_\gamma- P^{\eps,\beta}_t f_\eps}_{L^2(\Omega_\eps)}
			\\
			&\quad \le \sqrt{\mu_\eps(\Omega_\eps)}\tnorm{\Pi_\eps P^{\bd,c}_t \tilde f_\gamma-P^{\eps,\beta}_t \Pi_\eps \tilde f_\gamma}_{L^\infty(\Omega_\eps)}
			+ \tnorm{P_t^{\eps,\beta}}_{L^2(\Omega_\eps)\to L^2(\Omega_\eps)} \tnorm{\Pi_\eps \tilde f_\gamma- f_\eps}_{L^2(\Omega_\eps)} \fstop
		\end{align*}
		Since~$\mu_\eps(\Omega_\eps)$ is uniformly bounded by~\eqref{eq:VolumeBound}, the first term vanishes as~$\eps\to 0$ by~\ref{i:t:Equivalence:2} for each fixed~$\gamma$.
		Since~$P^{\eps,\beta}_t$ is $L^2(\Omega_\eps)$-contracting, taking first~$\limsup_{\eps\downarrow 0}$ and then~$\lim_{\gamma\downarrow 0}$, the second term vanishes as well by~\eqref{eq:KSStrong}.
		This verifies~\eqref{eq:KSStrong} with~$P^{\eps,\beta}_t f_\eps$ in place of~$f_\eps$ and~$\tilde g_\gamma$ in place of~$\tilde f_\gamma$, which concludes the assertion by definition of strong KS-convergence.
		
		\paragraph{Asymptotic compactness}
		We need to verify that, if~$\ttseq{f_\eps}_\eps$, with~$f_\eps\in L^2(\Omega_\eps)$, is such that
		\begin{align}\label{eq:AsympCompact}
			\limsup_{\eps\downarrow 0} \mcE^{\eps,\beta}(f_\eps)+\norm{f_\eps}_{L^2(\Omega_\eps)}^2 <\infty\comma
		\end{align}
		then~$\ttseq{f_\eps}_\eps$ has a strongly KS-convergent subsequence.
		By Proposition~\ref{p:ACtoKS}, we may equivalently verify that it has a strongly d-convergent subsequence.
		Since~$\mcE^{\eps,\beta}\geq \mcE^{\eps,\infty}$ for every~$\beta\in\R$, it suffices to verify the conclusion assuming~\eqref{eq:AsympCompact} with~$\beta=\infty$.
		The case~$\beta=\infty$ is shown in~\cite[Prop.~6.5]{forkert2020evolutionary} choosing~$\Omega$ in~\cite{forkert2020evolutionary} as a ball containing our domain~$\Omega$, and~$A$ in~\cite{forkert2020evolutionary} as our~$\Omega$.
		We note that the assumption in~\cite[Dfn.~3.1(i)]{forkert2020evolutionary} is satisfied since, in the notation of~\cite{forkert2020evolutionary}, we are choosing~$\overline{\mfm}=\mu_\Omega$ in~\cite[Eqn.~(2.10)]{forkert2020evolutionary} and~$U_{KL}\equiv 1$ in~\cite[Eqn.~(3.9)]{forkert2020evolutionary}.
		
		\medskip
		
		Let us now turn to the implication~\ref{i:t:Equivalence:5}$\implies$\ref{i:t:Equivalence:2}.
		Fix~$f\in \mcC^\bd$, and~$t>0$.
		It is readily seen that~$\Pi_\eps f\xrightarrow{\mcC^\bd,\Pi_\eps} f$.
		The assumption of~\ref{i:t:Equivalence:5} now implies that~$P^{\eps,\beta}_t \Pi_\eps f\xrightarrow{\mcC^\bd,\Pi_\eps} P^\bd_t f$ and the conclusion follows from the equicontinuity of~$P^{\eps,\beta}_t \Pi_\eps f$ shown in Proposition~\ref{pr:equi_semi_disc}.
	\end{proof}
\end{proposition}


\end{document}